\documentclass[12pt,letterpaper]{amsart}
\title{On K-stability of Calabi-Yau fibrations}
\author{Masafumi Hattori}
\address{Department of Mathematics, Faculty of Science, Kyoto University, Kyoto, 606-8502, Japan}
\email{hattori.masafumi.47z@st.kyoto-u.ac.jp}
\usepackage{amsthm}
\usepackage{amsmath}
\usepackage{amssymb}
\usepackage{amsfonts}
\usepackage{mathrsfs}
\usepackage{enumerate}
\usepackage[a4paper, top=2cm, bottom=2cm, left=2cm, right=2cm]{geometry}
\usepackage{amscd}
\usepackage{graphicx}
\usepackage[all]{xy}
\usepackage{hyperref}
\usepackage{fancyhdr}

\hypersetup{colorlinks=true}

\newtheorem{thm}{Theorem}[section]
\newtheorem{lem}[thm]{Lemma}
\newtheorem{prop}[thm]{Proposition}
\newtheorem*{claim}{Claim}

\newtheorem{cor}[thm]{Corollary}
\newtheorem{teiri}{Theorem}[section]
\newtheorem{kei}[teiri]{Corollary}

\theoremstyle{definition}
\newtheorem{de}[thm]{Definition}
\newtheorem{ex}[thm]{Example}

\newtheorem{conj}[thm]{Conjecture}

\newtheorem{rem}[thm]{Remark}
\newtheorem{ntt}[thm]{Notation}

\begin{document}
\maketitle
\begin{abstract}
We show that Calabi-Yau fibrations over curves are uniformly K-stable in an adiabatic sense if and only if the base curves are K-stable in the log-twisted sense. Moreover, we prove that there are cscK metrics for such fibrations when the total spaces are smooth.
\end{abstract}
\section{Introduction}
\subsection{K-stability and its criteria}
To find an explicit and useful criterion for K-stability of polarized varieties is meaningful for both of differential geometry and algebraic geometry. 
Originally, K-stability, which is a purely algebro-geometric condition,  was introduced by Tian \cite{T4} for Fano manifolds and reformulated by Donaldson \cite{Dn2} for all polarized varietiers by using the Donaldson-Futaki invariants of test configurations. 
At the same time, the Yau-Tian-Donaldson (YTD) conjecture was introduced.
This conjecture predicts that the existence of constant scalar curvature K\"{a}hler (for short, cscK) metrics is equivalent to K-polystability. 
In the Fano case, the conjecture was solved by Chen-Donaldson-Sun \cite{CDS} and Tian \cite{T2} independently.
Recently, K.~Zhang \cite{Z} shows the YTD conjecture for other kinds of general polarized manifolds.
In this respect, the problem to detect K-stability for a concrete polarized variety is meaningful for differential geometry.

In algebraic geometry, K-stability is closely related to birational geometry and moduli theory.
Odaka \cite{GIToda} finds that K-stability is closely related not only to the geometric invariant theory (GIT cf.~\cite{GIT}), but also to (semi) log canonical singularities.
Odaka \cite{O2} also finds that all semi log canonical models and log Calabi-Yau klt pairs are K-stable.
He predicts the K-moduli conjecture, which states that there exists a quasi-projective moduli scheme parametrizing all K-polystable varieties with fixed some numerical data (for example, the dimension of varieties) in \cite{O5}.
This conjecture is motivated by the works on KSBA moduli, which parametrizes all semi log canonical models  (see \cite{kollar-moduli} for details) and the partial result of Fujiki-Schumacher \cite{FS} on projectivity of the moduli space of certain cscK manifolds.
For log Fano pairs, to detect K-stability is more complicated than the classes treated in \cite{O2} but
the valuative criterion 
is established by Fujita \cite{Fjt} and Li \cite{Li3}.
Later Fujita and Odaka \cite{FO} introduced the $\delta$-invariant and Blum and Jonsson \cite{BlJ} prove that this invariant completely detects K-semistability of log Fano pairs.
The 
$\delta$-invariant is one of the key ingredients to 
the remarkable progress in K-stability of log Fano pairs in this decade. 
For example,  the K-moduli space (cf.~\cite{ABHLX}, \cite{LXZ}), which parametrizes all K-polystable log Fano pairs is constructed as a projective scheme.
It is natural to expect that if there was a more explicit and useful criterion for K-stability for other kinds of polarized varieties, we could construct moduli spaces of K-polystable polarized varieties. 
However, there are few such useful criteria found for K-stability of other polarized varieties.

\subsection{K-stability and the cscK problem on fibrations}
Adiabatic K-stability, which is first introduced in \cite{Hat2}, is the following notion.
\begin{de}
    Let $f\colon(X,\Delta,H)\to (C,L)$ be a polarized algebraic fiber space pair.
    If $(X,\Delta,\epsilon H+f^*L)$ is K-stable for any sufficiently small $\epsilon>0$, we say that $(X,\Delta,H)$ is adiabatically K-stable over $(C,L)$. See Definition \ref{algfibk} for details.
\end{de}
In K\"{a}hler geometry, the existence problem of cscK metrics on fibrations in the adiabatic sense is studied by \cite{J2}, \cite{Fine}, \cite{JSS}, \cite{DS2} and \cite{Ort}. 
The author introduces adiabatic K-stability to consider the algebro-geometric counterpart of these results.
Indeed, algebraic fiber spaces treated in \cite{J2}, \cite{Fine}, \cite{JSS}, \cite{DS2} and \cite{Ort} are adiabatically K-(poly)stable (see \cite{BDL}).
Fine \cite{J2}, \cite{Fine}, Dervan-Sektnan \cite{DS2} study the case when all fibers have cscK metrics.
Ortu generalizes the result of \cite{DS2} to K-semistable fibrations later. 
In particular, they only deal with fibrations without singular fibers. 
On the other hand, Jian-Shi-Song \cite{JSS} show the existence of a cscK metric on a klt--trivial fibration $f:X\to B$ when $K_X\sim_{\mathbb{Q}}f^*L_0$, $L_0$ is ample and $f$ might have a singular fiber. 
Here, we call an algebraic fiber space $f\colon X\to B$ a {\it klt--trivial} (resp.~{\it lc--trivial}) fibration if $K_X\sim_{\mathbb{Q},B}0$ and $X$ is klt (resp.~lc). 
Later, Sj\"{o}str\"{o}m Dyrefelt \cite{Sj} and Song \cite{S} show that smooth minimal models also have cscK metrics for certain polarizations
.
Conversely, Dervan-Ross (cf.~\cite[Corollary 4.4]{DR2}) show for any polarized algebraic fiber space that adiabatic K-semistability of the total space implies twisted K-semistability of the base in the sense of \cite{De2}.

 To klt--trivial fibrations with a singular fiber and $K_X$ not nef, we cannot directly apply these results on the existence of cscK metrics in the previous paragraph. 
 For example, rational elliptic surfaces have at least one singular fiber and $K_X$ is not nef. 
 Furthermore, the base of a rational elliptic surface is necessarily twisted K-stable (see Remark \ref{rem-final}).
 Hence, we cannot directly apply \cite[Corollary 4.4]{DR2} either.
On the other hand, Miranda constructs two moduli spaces of rational elliptic surfaces by applying GIT in \cite{Mi2} and \cite{Mi}.
In the two results, Miranda realizes that GIT-stability of a rational elliptic surface $X$ depends on what fiber $X$ has, and he further conjectured in \cite{Mi3} that asymptotic Chow-stability of rational Weierstrass fibrations would be decided by their fibers similarly to the results of \cite{Mi2} and \cite{Mi}. 
 \begin{conj}[Miranda, {\cite[Conjecture 2.3]{Mi3}}]\label{conj--miranda--83}
 Let $W$ be a rational Weierstrass fibration, $X$ be its minimal resolution and $H$ be an arbitrary ample line bundle on $W$. Then, $(W,H)$ is 
\begin{enumerate}[(a)]
\item asymptotically Chow-stable if $X$ has only reduced fibers,
\item asymptotically Chow-semistable if $X$ has only reduced and $I_N^*$-type fibers for $N\ge0$,
\item asymptotically Chow-unstable if $X$ has at least one $II^*$, $III^*$ or $IV^*$-type fiber.
\end{enumerate}
\end{conj}
Indeed, Miranda shows that (c) holds for some ample line bundles.
If Miranda's conjecture would hold, then a rational Weierstrass fibration $W$ whose minimal resolution  $X$ has at least one $II^*$, $III^*$ or $IV^*$-type fiber would be adiabatically K-unstable and \cite[Corollary 4.4]{DR2} would not be optimal to detect adiabatic K-unstability of lc--trivial fibrations.

\begin{rem}
    We showed in the first version of this paper \cite{Hatpre} that Conjecture \ref{conj--miranda--83} (c) holds when the Picard number of $W$ is two.
    Furthermore, we show that if $X$ has at least one $I_{N}^*$-type fiber for some $N\ge1$, then $(W,H)$ is K-unstable for some ample line bundle $H$. 
    Therefore, Conjecture \ref{conj--miranda--83} (b) does not hold in general.
    We would write them in a forthcoming paper.
\end{rem}

\subsection{Main results}

The main purpose of the paper is to construct an explicit and useful criterion for adiabatic K-stability of klt--trivial fibrations over curves.
To be precise, we introduce the notion of {\it uniform adiabatic K-stability} (cf.~Definition \ref{algfibk}) and {\it log--twisted K-stability} (see Definition \ref{deflogtwisted}) of a log-twisted pair. 
The latter is the logarithmic version of the twisted K-stability introduced by Dervan \cite{De2}.
We note that for any klt-trivial fibration, there are natural two divisors that are called the moduli and discriminant divisors on the base (cf.~\cite{A}, \S\ref{subsec-lc-triv}).
Then the following is our main result.

\begin{teiri}\label{dd}
Let $f:(X,\Delta,H)\to (C,L)$ be a polarized klt-trivial fibration such that $C$ is a smooth curve. Let $M$ be the moduli divisor and $B$ the discriminant divisor on $C$. 

Then, $(X,\Delta,H)$ is uniformly adiabatically K-stable over $(C,L)$ if and only if $(C,B,M,L)$ is log-twisted K-stable. 
Moreover, if $\Delta=0$, $X$ is smooth and $(C,B,M,L)$ is log-twisted K-stable, then $(X,\epsilon H+L)$ has a cscK metric for any sufficiently small $\epsilon\in\mathbb{Q}_{>0}$.
\end{teiri}
See Definition \ref{ambro model} for the discriminant and moduli divisors.
Theorem \ref{dd} is useful to detect uniform adiabatic K-stability and the existence of cscK metrics since it is easy to check the log-twisted K-stability of curves (see Corollary \ref{appc}). 
We emphasize that there exists a uniformly adiabatically K-stable klt--trivial fibration over a curve with a singular and K-unstable fiber. 
We note that Theorem \ref{dd} is also a key ingredient to construct the moduli space in \cite{HH}.

To show Theorem \ref{dd}, there are three key ingredients, Theorems \ref{thm--first--ingredient}, \ref{ff} and \ref{ee}. 
Theorem \ref{ff}, which is the first key ingredient, states that adiabatic K-semistability of a lc--trivial fibration puts more restrictions on the base space than \cite[Corollary 4.4]{DR2}.
We emphasize that we do not assume $C$ to be a curve in the following theorem.

\begin{teiri}[Theorems \ref{can}, \ref{stp1}]\label{ff}
Let $f:(X,\Delta,H)\to (C,L)$ be a polarized lc-trivial fibration.
If $B_C$ is the discriminant divisor, the moduli divisor $M_C$ is $\mathbb{Q}$-Cartier and $(X,\Delta,H)$ is adiabatically K-semistable, then $(C,B_C,M_C,L)$ is log-twisted K-semistable.

Furthermore, if $(C,B_C,M_C,L)$ is a log-twisted Fano pair and $C$ is an Ambro model (cf.~Definition \ref{de-ambro}), and $(X,\Delta,H)$ is uniformly adiabatically K-stable, then $(C,B_C,M_C,L)$ is also K-stable.
\end{teiri} 

We note that if $C$ is a smooth curve, then $C$ is an Ambro model.
Thus, we obtain one implication of Theorem \ref{dd} by Theorem \ref{ff}.
On the other hand, we deduce the opposite implication from Theorems \ref{thm--first--ingredient} and \ref{ee}.

For the second ingredient, we need the notion of J-stability.
J-stability is a variant of K-stability, which is irrelevant to singularities.
G.~Chen \cite{C} found a useful criterion for J-stability without using test configurations (see also \cite{DP}, \cite{S} and \cite{Hat}).
We extend this criterion to the singular case (see Theorem \ref{thm-jst}).
For polarized klt pairs, we also show Theorem \ref{bhjz} stating that 
\[
H^{\mathrm{NA}}_{\Delta}(\mathcal{X},\mathcal{L})\ge\delta_{(X,\Delta)}(L)(I^{\mathrm{NA}}(\mathcal{X},\mathcal{L})-J^{\mathrm{NA}}(\mathcal{X},\mathcal{L}))=(\mathcal{J}^{\delta_{(X,\Delta)}(L)L})^{\mathrm{NA}}(\mathcal{X},\mathcal{L}).
\]
We note that Theorem \ref{bhjz} had been solved for the smooth case in \cite{BJ2} (see also \cite{Z} for the analytic case).
Combining Theorem \ref{bhjz} with the criterion for J-stability (Theorem \ref{thm-jst}), we deduce the following sufficient condition for uniform K-stability.

\begin{teiri}\label{thm--first--ingredient}
Let $(X,\Delta,L)$ be a polarized klt pair and set $H=\delta_{(X,\Delta)}(L)L+K_X+\Delta$.
    If $H$ is ample and there exists $\epsilon>0$ such that
    \begin{equation}\nonumber
   \left(n\frac{H\cdot L^{n-1}}{L^n}L-pH\right)\cdot L^{p-1}\cdot V\ge \epsilon(n-p)L^p\cdot V
   \end{equation}
    for any subvariety $V\subset X$ of dimension $1\le p\le n-1$, then $(X,\Delta,L)$ is uniformly K-stable.
    \end{teiri}

Now, we want to analyze the $\delta$-invariant $\delta_{(X,\Delta)}(\epsilon H+L)$ for sufficiently small $\epsilon>0$ and apply Theorem \ref{thm--first--ingredient}. 
Theorem \ref{ee}, which is the third key ingredient, states that $\delta_{(X,\Delta)}(\epsilon H+L)$ converges to the $\delta$-invariant of the base of a polarized algebraic fiber space over a curve
.

\begin{teiri}[Theorems \ref{ctn}, \ref{ctnd}]\label{ee}
Let $f:(X,\Delta,H)\to (C,L)$ be a polarized algebraic fiber space pair. Suppose that $(X,\Delta)$ is a klt pair and $C$ is a smooth curve.
Let $B$ be a $\mathbb{Q}$-divisor defined by $\sum_{p\in C}(1-\mathrm{lct}_{(X,\Delta)}(f^{-1}(p)))p$ on $C$. 
Then 
\begin{align*}
\lim_{\epsilon\to 0}\alpha_{(X,\Delta)}(\epsilon H+f^*L)&=\alpha_{(C,B)}(L),\quad\text{and}\\
\lim_{\epsilon\to 0}\delta_{(X,\Delta)}(\epsilon H+f^*L)&=\delta_{(C,B)}(L).
\end{align*} 
\end{teiri} 
We briefly explain the idea to show the K-stability of the base implies uniform adiabatic K-stability when $-(K_X+\Delta)\not\equiv0$ is nef.
We note that the base is a log-twisted K-stable Fano pair in this case and then the $\delta$-invariant of the base is larger than one. 
Theorem \ref{ee} shows that the $\delta$-invariant of the total space is also larger than one.
Then we can make use of Theorem \ref{thm--first--ingredient} to deduce uniform adiabatic K-stability.
Furthermore, we can show that K-energy is coercive in the same way and deduce the existence of cscK metrics from \cite{Ch}.

As an application of Theorem \ref{dd} to rational elliptic surfaces, we obtain the following.

\begin{kei}[Corollary \ref{K-fib}]\label{hh}
Let $X$ be a smooth rational elliptic surface and $H$ be an ample line bundle on $X$. Set $m\ge 1$ as the index of $X$ (cf.~Definition \ref{de-mult}). Then the following hold.
\begin{enumerate}
\item If $m=1$ and $X$ has at most reduced fibers, then $(X,H)$ is uniformly adiabatically K-stable and has cscK metrics.
\item If $m=2$ and $X$ has no $II^*$ or $III^*$-type fiber, then $(X,H)$ is uniformly adiabatically K-stable and has cscK metrics.
\item If $m=3$ and $X$ has no $II^*$-type fiber, then $(X,H)$ is uniformly adiabatically K-stable and has cscK metrics.
\item If $m\ge 4$, then $(X,H)$ is uniformly adiabatically K-stable and has cscK metrics.
\end{enumerate}
 On the other hand,
\begin{enumerate}[(i)]
\item If $m=1$ and $X$ admits one of $II^*$, $III^*$ or $IV^*$-type fibers, then $(X,H)$ is adiabatically K-unstable over $\mathbb{P}^1$.
\item If $m=2$ and $X$ admits a $II^*$-type fiber, then $(X,H)$ is adiabatically K-unstable over $\mathbb{P}^1$.
\end{enumerate}
\end{kei}




\subsection*{Structure of this paper and overview of proof}
In Section \ref{Notat}, we will give many terminologies and basic results on K-stability and birational geometry.
In this section, we also introduce uniform adiabatic K-stability and log-twisted K-stability.

In Section \ref{subsec-j-st}, we partially extend \cite[Theorem 1.1]{C} and show Theorem \ref{thm--first--ingredient}, which is one of the key ingredients to deduce our main theorems.
We also deal with K-stability of polarized klt minimal models.
To show uniform K-stability of polarized klt minimal models for fixed certain polarizations, we make use of the result on J-stability for singular varieties as \cite{Sj} and \cite{S} did.

In Section \ref{sec-3}, we consider the relationship between the adiabatic K-stability of Calabi-Yau fibrations and the log-twisted K-stability of the base and discuss the proof of Theorem \ref{dd}.
In Subsection \ref{unstadiab}, we prove Theorem \ref{ff}.
We explain how to show Theorem \ref{ff} briefly here.
We compare the log-twisted Donaldson-Futaki invariant of a test configuration $(\mathscr{C},\mathcal{L})$ for the base variety with the Donaldson-Futaki invariant of the test configuration $(\mathcal{X},H_{\mathbb{P}^1}+m\mathcal{L})$ induced by the base change by $X\to C$ as the proof of \cite[Corollary 4.4]{DR2}, where $H$ is an ample line bundle on $X$.
We note that the $\mathrm{DF}(\mathcal{X},H_{\mathbb{P}^1}+m\mathcal{L})$ is a polynomial of $m$ with leading coefficient the twisted Donaldson-Futaki invariant of $(\mathscr{C},\mathcal{L})$ (see the proof of \cite[Corollary 4.4]{DR2}) and bigger than the log-twisted Donaldson-Futaki invariant. 
Then, we run a minimal model program of $\mathcal{X}$ over $\mathscr{C}$ and compare the canonical divisor of the model of $\mathcal{X}$ with the canonical divisor of $\mathscr{C}$ by the canonical bundle formula introduced by \cite{A} and \cite{FG}.
Conversely, we show Theorem \ref{ee} in Subsection \ref{Weier} and obtain that log-twisted K-stability of the base Fano curve implies adiabatic K-stability of the total space of a klt-trivial fibration, which shows Theorem \ref{dd}.
To show Theorem \ref{ee} for the $\alpha$-invariant, we estimate the multiplicities of the horizontal part at any point and of the vertical part along irreducible components of fibers of any $\mathbb{Q}$-divisor $\mathbb{Q}$-linearly equivalent to $\epsilon H+L$ when $\epsilon$ is so small.
For the $\delta$-invariant, we deduce Theorem \ref{ee} from the assertion for the $\alpha$-invariant.


In Appendix \ref{appendices}, we discuss the fundamental results on log-twisted K-stability. 
In Subsection \ref{sec-logtwisteddelta}, we show Theorem \ref{bhjz}, which is the key ingredient to show Theorem \ref{dd}.
We consider log-twisted Ding stability to deduce this theorem.
We remark that  we cannot apply the main result of \cite{LX} directly in the proof of Theorem \ref{dd} since we have to treat when the twist term $T$ is not nef.
On the other hand, in Subsection \ref{sec-stp1}, we deal with the case when $(X,B,T,L)$ is log-twisted Fano and $T$ is semiample and discuss its K-stability.
This is one of the main 
ingredients to show Theorem \ref{ff}.


\subsection*{Acknowledgements}
The author is grateful to his research advisor Professor Yuji Odaka for a lot of suggestive advices and productive discussions. The author would like to thank Odaka and Doctor Keita Goto for careful reading his draft. He is grateful to Doctor Aline Zanardini for pointing out typos in an earlier version. The author is also grateful to Professor Yoshinori Hashimoto for giving him a lecture on the paper \cite{Z}. The author would like to thank Professor Kenta Hashizume and Doctor Eiji Inoue for fruitful discussions and helpful advices. 
This work is a part of his master thesis and partially 
supported by JSPS KAKENHI 	22J20059  
(Grant-in-Aid for JSPS Fellows DC1).

\section{Preliminaries}
\label{Notat}
In this paper, we work over the complex number field $\mathbb{C}$.
We call a pair $(X,L)$ a {\it polarized variety} if $X$ is a proper normal variety over $\mathbb{C}$ and $L$ is an ample $\mathbb{Q}$-line bundle  on $X$.
Let $D$ be a $\mathbb{Q}$-Weil divisor, i.e.~$D=\sum_{i=1}^ka_iD_i $ for some $a_i\in\mathbb{Q}_{\ne0}$ and distinct prime divisors $D_i$.
We set $D_{>0}:=\sum_{i\colon a_i>0}a_iD_i$,  $D_{<0}:=D_{>0}-D$, $\lfloor D\rfloor:=\sum \lfloor a_i\rfloor D_i$, $\lceil D\rceil:=\sum \lceil a_i\rceil D_i$ and $D_{=1}:=\sum_{i\colon a_i=1}a_iD_i$.
On the other hand, set $$\mathrm{red}\,D:=\sum_{i}D_i$$ and call this the {\it reduced structure} of $D$.
Let $f\colon X\to Y$ be a proper surjective morphism of normal varieties and $D$ be a $\mathbb{Q}$-divisor on $X$. 
Let $D=\sum_{i=1}^k a_i D_i$ be the irreducible decomposition.
Suppose that there exists $1\le l\le k$ such that $D_i$ for any $i\le l$ is dominant to $Y$ but $D_i$ is not so for any $i>l$.
Then we set $D_{\mathrm{hor}}=\sum_{i\le l}a_iD_i$ (resp.~$D_{\mathrm{vert}}=\sum_{i> l}a_iD_i$) and call this the {\it horizontal} (resp.~{\it vertical}) part of $D$.
Let $U\subset Y$ be an open subset.
If $f|_{f^{-1}(U)}$ is smooth and any stratum of $D|_{f^{-1}(U)}$ is smooth over $U$, then we say that $(X,D)$ is (relatively) {\it log-smooth} over $U$.
In the case when $Y=U=\mathrm{Spec}\,\mathbb{C}$, we say that $(X,D)$ is {\it log-smooth}.
In this case, we call $D$ a {\it simple normal crossing} (for short, snc) divisor.

\subsection{Birational geometry and MMP}
First, we recall the fundamental concepts of birational geometry.
\begin{de}[Log pair]
Let $X$ and $S$ be quasi projective normal varieties such that there exists a projective surjective morphism $f\colon X\to S$. 
Let $\Delta$ be a $\mathbb{Q}$-Weil divisor on $X$ such that $ K_X+\Delta$ is $\mathbb{Q}$-Cartier.
Then we call $(X,\Delta)$ a {\it log subpair} over $S$. 
If $\Delta$ is further effective
, then we call $(X,\Delta)$ a {\it log pair}
.
For any subpair $(X,\Delta)$, we take {\it a log resolution} $\sigma\colon Y\to X$ such that $\sigma$ is a projective birational morphism from a smooth variety $Y$ such that $\mathrm{Ex}(\sigma)+\sigma^{-1}_{*}\Delta$ is snc, where $\mathrm{Ex}(\sigma)$ is the exceptional locus.  
If an algebraic group $G$ acts on $X$, then there exists a log resolution $\sigma$ such that $\sigma$ is $G$-equivariant (cf.~\cite{Ko2}) and we call such $\sigma$ a $G$-equivariant log resolution.
We call $F$ a {\it prime divisor over} $X$ if there exists a log resolution $\sigma:Y\to X$ of $(X,\Delta)$ such that $F$ is a prime divisor on $Y$. 
A valuation $v$ is called {\it divisorial} if there exist a prime divisor $F$ over $X$ and $c\in\mathbb{Q}_{>0}$ such that $v=c\,\mathrm{ord}_F$.
Then we set the {\it log discrepancy} of $(X,\Delta)$ with respect to $F$ as
\[
A_{(X,\Delta)}(v)=v(K_Y+F-\sigma^*(K_X+\Delta)).
\]
In particular, if $v=\mathrm{ord}_F$, then we denote $A_{(X,\Delta)}(F)$.
A log subpair $(X,\Delta)$ is called subklt (resp.~sublc; $\delta$-sublc, where $\delta>0$) if $A_{(X,\Delta)}(v)>0$ (resp.~$\ge0$; $\ge\delta$) for any divisorial valuation $v$.
If $(X,\Delta)$ is further a pair, we call this {\it klt} (resp.~{\it lc}).
If a (sub)lc pair $(X,\Delta)$ has a log resolution $\pi\colon Y\to X$ such that $A_{(X,\Delta)}(E)>0$ for any $\pi$-exceptional divisor $E$, we say that $(X,\Delta)$ is {\it (sub)dlt}.
Furthermore,
   if $(X,\Delta)$ is lc and $A_{(X,\Delta)}(E)>0$ for any prime divisor $E$ over $X$ whose center is not a prime divisor, then we say that $(X,\Delta)$ is a {\it plt} pair.
A subvariety $Z\subset X$ is called a {\it lc center} of an lc pair $(X,\Delta)$ if there exists a prime divisor $E$ over $X$ such that $A_{(X,\Delta)}(E)=0$ and the center of $E$ is the generic point of $Z$.
It is well-known that if $(X,\Delta)$ is a dlt pair, then a lc center of $(X,\Delta)$ is a stratum of $\lfloor\Delta\rfloor$ (cf.~\cite[Theorem 4.16]{Ko}).
On the other hand, if for any Weil divisor $D$ on $X$, there exists $m\in\mathbb{Z}_{>0}$ such that $mD$ is Cartier, then we say that $X$ is {\it $\mathbb{Q}$-factorial}.
\end{de}

\begin{ex}\label{ex-ord_p}
    Let $P\in X$ be a closed point and suppose that $X$ is smooth at $P$.
    Let $\mu\colon \tilde{X}\to X$ be the blow up of $X$ at $P$ and $E$ be the $\mu$-exceptional divisor.
    Then we set $\mathrm{ord}_P:=\mathrm{ord}_E$ and this is an example of a divisorial valuation.
\end{ex}

\begin{de}[Models]
Let $(X,\Delta)$ be a log pair over $S$.
Let $\phi:X\dashrightarrow Y$ be a birational map such that $Y$ is a normal variety projective over $S$.
Let $\mathrm{Ex}(\phi)$ be the sum of all $\phi$-exceptional prime divisors.
$\phi:X\dashrightarrow Y$ is called a {\it birational contraction} if there exists no $\phi^{-1}$-exceptional divisor.
Here, suppose that 
$\phi$ is a birational contraction preserving the structure morphisms to $S$. 
Take a resolution $\Gamma$ of indetrminacy of $\phi$ and let $p:\Gamma\to X$ and $q:\Gamma\to Y$ be the canonical projections. Suppose that $K_Y+\phi_*\Delta$ is $\mathbb{Q}$-Cartier and let $E$ be the $q$-exceptional $\mathbb{Q}$-divisor on $\Gamma$ defined by
\[
E=p^*(K_X+\Delta)-q^*(K_Y+\phi_*\Delta).
\]
We say that $\phi$ is $K_X+\Delta$-non-positive (resp.~$K_X+\Delta$-negative) if $E$ is effective (resp.~and further $\mathrm{Supp}\,E$ contains all $\phi$-exceptional prime divisors).
We call $(Y,\phi_*\Delta)$ (or simply call $Y$)
\begin{itemize}
    \item a {\it good minimal model} of $(X,\Delta)$ if $K_Y+\phi_*\Delta$ is relatively semiample over $S$ and $\phi$ is $K_X+\Delta$-negative,
    \item the {\it log canonical (lc) model} of $(X,\Delta)$ if $K_Y+\phi_*\Delta$ is relatively ample over $S$ and $\phi$ is $K_X+\Delta$-non-positive.
\end{itemize}
For details, see \cite{KoMo} and \cite{BCHM}.
\end{de}

Let $D$ be a $\mathbb{Q}$-divisor on $X$. 
We say that a $\mathbb{Q}$-divisor $D'$ is {\it relatively linearly equivalent to $D$ over $S$} if there exists a Cartier divisor $M$ on $S$ such that $D\sim D'+f^*M$, where $f\colon X\to S$ is the canonical morphism, and denote $D\sim_{S}D'$.
If there exists a non-zero integer $m\in\mathbb{Z}$ such that $mD\sim_SmD'$, then we say that $D'$ is relatively $\mathbb{Q}$-linearly equivalent to $D$ over $S$ and denote $D'\sim_{S,\mathbb{Q}}D$.
We denote by $\mathbf{B}(D/S)$ a Zariski closed subset that is the intersection of the supports of all effective $\mathbb{Q}$-divisors relatively $\mathbb{Q}$-linearly equivalent to $D$.
Furthermore, we set $\mathbf{B}_{+}(D/S):=\mathbf{B}((D-\epsilon A)/S)$ for an $f$-ample divisor $A$ and a sufficiently small rational number $\epsilon>0$.
$\mathbf{B}_{+}(D/S)$ is independent of the choices of $A$
 and $\epsilon$ (cf.~\cite[\S10.3]{Laz}).
On the other hand, if $S=\mathrm{Spec}\,\mathbb{C}$ and $\mathbb{Q}$-divisors $D_1$ and $D_2$ satisfy that $C\cdot (D_2-D_1)=0$ for any curve $C\subset X$, where $D_2-D_1$ is $\mathbb{Q}$-Cartier, we say that $D_1$ and $D_2$ are numerically equivalent and denote $D_1\equiv D_2$.

 Now, we are ready to explain briefly a minimal model program with scaling (see \cite{BCHM} or \cite{fujino-foundation} for details).
 Suppose that $(X,\Delta)$ is a $\mathbb{Q}$-factorial dlt pair and let $C$ be an effective and big $\mathbb{Q}$-Cartier $\mathbb{Q}$-divisor with $\mathbf{B}_+(C/S)$ contains no lc center of $(X,\Delta)$. 
 Suppose that $(X,\Delta+C)$ is lc and $K_X+\Delta+C$ is nef.
If $K_X+\Delta$ is pseudoeffective but not nef, then there exists a rational number $0<\lambda_1\le1$ such that $K_X+\Delta+\lambda_1C$ is nef and there exists a negative extremal ray $R_1$ such that $R_1\cdot (K_X+\Delta+\lambda_1C)=0$.
 Otherwise, let $\lambda_1=0$.
Let $(X_1,\Delta_1)$ be the base of the divisorial contraction or the flip induced by $R_1$ with respect to $K_X+\Delta$ and $C_1$ be the strict transform of $C$ (cf.~\cite{KoMo}).
Then we see that $(X_1,\Delta_1)$ is dlt, $(X_1,\Delta_1+\lambda_1C_1)$ is lc, and $\mathbf{B}_+(C_1/S)$ contains no lc center of $(X_1,\Delta_1)$.
Thus, we see by the above argument that if $K_{X_1}+\Delta_1$ is not nef, then there exists $0<\lambda_2\le\lambda_1$ such that $K_{X_1}+\Delta_1+\lambda_2C_1$ is nef and there exists a negative extremal ray $R_2$ for $K_{X_1}+\Delta_1$ such that $(K_{X_1}+\Delta_1+\lambda_2C_1)\cdot R_2=0$.
Then we obtain $(X_2,\Delta_2)$ as the divisorial contraction or the flip induced by $R_2$.
By repeating this process, we obtain a non-increasing sequence $\{\lambda_i\}_{i\ge1}$ of rational numbers and a sequence of divisorial contractions or flips
\[
(X,\Delta)\dashrightarrow(X_1,\Delta_1)\dashrightarrow\cdots\dashrightarrow(X_i,\Delta_i)\dashrightarrow\cdots.
\]
If there exists $i\in\mathbb{Z}_{>0}$ such that $\lambda_i=0$, then we call the above procedure a {\it minimal model program of $(X,\Delta)$ over $S$ with scaling of $C$}.
For details, see \cite[4.4.11]{fujino-foundation}.
We say that we run a minimal model program over $S$ with ample scaling, if we take a general ample $\mathbb{Q}$-divisor $C$ and run a minimal model program over $S$ with scaling of $C$ and this program terminates.
\begin{de}[$G$-equivariant models]\label{de-g-equiv-model}
Let $G$ be a connected algebraic group and suppose that $G$ acts on $S$.
If $(X,\Delta)$ is a log pair over $S$, there exists $f\colon X\to S$ and $G$ acts on $X$ in the way that $f$ is $G$-equivariant and $\Delta$ is $G$-stable, then we say that $(X,\Delta)$ admits a $G$-action.
Suppose that $G$ acts on $(X,\Delta)$ and let $\phi\colon X\dashrightarrow Y$ be a birational map.
If $Y$ admits a $G$-action and the graph $\Gamma$ of $\phi$ is a $G$-invariant subvariety of $X\times Y$, then we say that $\phi$ is {\it $G$-equivariant}. 
If $(Y,\phi_*\Delta)$ is a good minimal model of $(X,\Delta)$ and $\phi$ is $G$-equivariant, we call $Y$ a $G$-{\it equivariant good minimal model}.
By the argument of the proof of \cite[Theorem 7]{LX} (see also \cite[1.5]{A2}), if $Y$ is a good minimal model obtained by a certain minimal model program of $(X,\Delta)$ with scaling, then $Y$ is a $G$-equivariant good minimal model.
Note also that if $Y$ is the lc model of $(X,\Delta)$, then $Y$ has the natural $G$-action and the canonical birational map $\phi\colon X\dashrightarrow Y$ is $G$-equivariant since $Y=\mathbf{Proj}_{S}(\oplus_{l\ge0}f_*\mathcal{O}_{X}(lr_0(K_{X}+\Delta)))$, where $r_0(K_{X}+\Delta)$ is Cartier.
\end{de}

\subsection{K-stability}
$(X,\Delta,L)$ is called a {\it polarized log (sub)pair} if $(X,\Delta)$ is a (sub)pair over $\mathrm{Spec}\,\mathbb{C}$ and $L$ is an ample $\mathbb{Q}$-line bundle.
In this subsection, we recall the definition of log K-stability of polarized log pairs.
See \cite{BHJ} for details.
If there is no fear of confusion and $\Delta=0$, we simply denote $(X,0,L)$ by $(X,L)$.

\begin{de}[Test configuration]\label{de-test-config}
Let $\pi:\mathcal{X}\to \mathbb{P}^1$ be a projective and flat morphism such that $\pi$ is $\mathbb{G}_m$-equivariant, where $\mathbb{P}^1$ admits the canonical $\mathbb{G}_m$-action by multiplication.
If there exists a $\mathbb{G}_m$-equivariant isomorphism $\pi^{-1}(\mathbb{P}^1\setminus\{0\})\cong X\times(\mathbb{P}^1\setminus\{0\})$ over $\mathbb{P}^1\setminus\{0\}$, where $\mathbb{G}_m$ trivially acts on the first component of $X\times(\mathbb{P}^1\setminus\{0\})$, $\mathcal{X}$ is called a {\it test configuration} for $X$.
Assume further that $\mathcal{X}$ admits a relatively
(semi)ample $\mathbb{G}_m$-linearized $\mathbb{Q}$-line bundle $\mathcal{L}$ over $\mathbb{P}^1$ such that $\mathcal{L}|_{\pi^{-1}(\mathbb{P}^1\setminus\{0\})}$ coincides with $L\otimes\mathcal{O}_{\mathbb{P}^1\setminus\{0\}}$ with the trivial $\mathbb{G}_m$-action on $\mathcal{L}|_{\pi^{-1}(\infty)}$.
Here, we say that a $\mathbb{Q}$-line bundle $\mathcal{L}$ is $\mathbb{G}_m$-linearized if there exists a positive integer $m$ such that $m\mathcal{L}$ is a $\mathbb{G}_m$-linearlized line bundle in the sense of \cite[Definition 1.6]{GIT}.
Then, $(\mathcal{X},\mathcal{L})$ is called a {\it (semi)ample test configuration} for $(X,L)$.

 If $\mathcal{X}$ is $\mathbb{G}_m$-equivariantly isomorphic to $X\times\mathbb{P}^1$ with the trivial action, then we call $\mathcal{X}$ {\it trivial}. We denote $(X\times\mathbb{P}^1,L\otimes\mathcal{O}_{\mathbb{P}^1})$ with the trivial action by $(X_{\mathbb{P}^1},L_{\mathbb{P}^1})$. 
 Similarly, we denote $\Delta\times\mathbb{P}^1\subset X_{\mathbb{P}^1}$ by $\Delta_{\mathbb{P}^1}$ for any $\mathbb{Q}$-divisor $\Delta$ on $X$.

Let $Y$ be another projective variety and $f\colon X\to Y$ be a morphism.
If $\mathcal{X}$ (resp.~$\mathcal{Y}$) is a test configuration for $X$ (resp.~$Y$), we note that there exists a unique $\mathbb{G}_m$-equivariant rational map $F\colon\mathcal{X}\dashrightarrow\mathcal{Y}$ induced by $f\times\mathrm{id}_{\mathbb{P}^1\setminus\{0\}}$.
We call such $F$ the {\it canonical} rational map.
If $F$ is further a morphism, then we say that there exists a canonical morphism $F\colon\mathcal{X}\to\mathcal{Y}$
.


\end{de}

\begin{rem}
    Usually, we call $\pi^{-1}(\mathbb{P}^1\setminus\{\infty\})$ a test configuration (see \cite{BHJ} for example) in the situation of Definition \ref{de-test-config} rather than our $\mathcal{X}$.
In Definition \ref{de-test-config}, what we defined to be test configurations are called the canonical compactifications of test configurations in \cite{BHJ}.
In this paper, we only make use of the canonical compactifications of test configurations.
This is why we define the notion of test configuration as in Definition \ref{de-test-config}.
\end{rem}

We note that if a test configuration $\mathcal{X}$ for $X$ is normal, then there exists the following relationship between irreducible components of $\mathcal{X}_0$ and divisorial valuations of $X$ (cf.~\cite[\S4]{BHJ}).
Let $E$ be an irreducible component of $\mathcal{X}_0$ and set $b_E:=\mathrm{ord}_E(\mathcal{X}_0)$.
If $E$ is not the strict transform of the central fiber of $X_{\mathbb{P}^1}$, we know that $\frac{1}{b_E}\mathrm{ord}_E|_{X}$ is a nontrivial divisorial valuation and we denote this by $v_E$ (see \cite[Lemma 4.1]{BHJ}).

\begin{ntt}
For simplicity, we will denote line bundles and divisors interchangeably
. For instance, $$L^m\cdot (kH+D):=L^{\cdot m}\cdot (H^{\otimes k}\otimes \mathcal{O}_X(D)).$$ 

\end{ntt}

We define the following functionals and the Donaldson-Futaki invariant as in \cite{BHJ}.

\begin{de}[{\cite[\S 6, \S 7]{BHJ}}]\label{nadef}
Let $(X,\Delta,L)$ be a polarized log subpair of dimension $n$.
Take an arbitrary normal semiample test configuration $(\mathcal{X},\mathcal{L})$ for $(X,L)$.
Let $\rho\colon\mathcal{X}\dashrightarrow X_{\mathbb{P}^1}$ be the canonical birational map.
It is well-known that there exist a normal test configuration $\mathcal{Y}$ and a canonical morphism $\sigma\colon\mathcal{Y}\to\mathcal{X}$ such that the canonical birational map $\rho_{\mathcal{Y}}=\rho\circ\sigma$ is a morphism.
Furthermore, let $\mathcal{Y}_{0,\mathrm{red}}$ be the reduced structure of $\mathcal{Y}_0$, $\Delta_{\mathcal{Y}}$ be the closure of $\Delta\times (\mathbb{P}^1\setminus\{0\})$ in $\mathcal{Y}$ and $K^{\mathrm{log}}_{(\mathcal{Y},\Delta_{\mathcal{Y}})/\mathbb{P}^1}:=K_{(\mathcal{Y},\Delta_{\mathcal{Y}})/\mathbb{P}^1}+(\mathcal{Y}_{0,\mathrm{red}}-\mathcal{Y}_0)$, where $\pi\colon\mathcal{Y}\to\mathbb{P}^1$ is the canonical morphism and $K_{(\mathcal{Y},\Delta_{\mathcal{Y}})/\mathbb{P}^1}:=K_{\mathcal{Y}}+\Delta_{\mathcal{Y}}-\pi^*K_{\mathbb{P}^1}$.
Then we define
\begin{itemize}
\item 
$I^{\mathrm{NA}}(\mathcal{X},\mathcal{L})=(L^n)^{-1}(\sigma^*\mathcal{L}\cdot \rho_\mathcal{Y}^*L_{\mathbb{P}^1}^n)-(L^n)^{-1}(\sigma^*\mathcal{L}- \rho_\mathcal{Y}^*L_{\mathbb{P}^1})\cdot(\sigma^*\mathcal{L})^n,$
\item
$E^{\mathrm{NA}}(\mathcal{X},\mathcal{L})=\frac{\mathcal{L}^{n+1}}{(n+1)(L^n)},$

\item 
$J^{\mathrm{NA}}(\mathcal{X},\mathcal{L})=(L^n)^{-1}(\sigma^*\mathcal{L}\cdot \rho_\mathcal{Y}^*L_{\mathbb{P}^1}^n)- E^{\mathrm{NA}}(\mathcal{X},\mathcal{L}),$
\item 
the non-Archimedean J$^H$-energy  is defined to be
\[
(\mathcal{J}^H)^{\mathrm{NA}}(\mathcal{X},\mathcal{L})=(L^n)^{-1}(\rho_\mathcal{Y}^*H_{\mathbb{P}^1}\cdot (\sigma^*\mathcal{L})^{n})- (L^n)^{-1}(nH\cdot L^{n-1})E^{\mathrm{NA}}(\mathcal{X},\mathcal{L}),
\]
where $H$ is a $\mathbb{Q}$-line bundle on $X$,
\item
the non-Archimedean entropy is defined to be
\begin{align*}
H_\Delta^{\mathrm{NA}}(\mathcal{X},\mathcal{L})&=(L^n)^{-1}\sum_Eb_EA_{(X,B)}(v_E)(E\cdot(\sigma^*\mathcal{L})^n)\\
&=(L^n)^{-1}(K^{\mathrm{log}}_{(\mathcal{Y},\Delta_{\mathcal{Y}})/\mathbb{P}^1}-\rho_{\mathcal{Y}}^*K_{(X_{\mathbb{P}^1},\Delta_{\mathbb{P}^1})/\mathbb{P}^1})\cdot(\sigma^*\mathcal{L})^n,
\end{align*}
where $E$ runs over all irreducible components of $\mathcal{Y}_0$ that is not the strict transform of $X\times\{0\}$,
\item
the log Donaldson-Futaki invariant is defined to be
$$\mathrm{DF}_{\Delta}(\mathcal{X},\mathcal{L})=(L^n)^{-1}(K_{(\mathcal{Y},\Delta_{\mathcal{Y}})/\mathbb{P}^1}-\rho_\mathcal{Y}^*K_{(X_{\mathbb{P}^1},\Delta_{\mathbb{P}^1})/\mathbb{P}^1})\cdot\mathcal{L}^n+(\mathcal{J}^{K_{X}+\Delta})^{\mathrm{NA}}(\mathcal{X},\mathcal{L}),$$
\item 
the non-Archimedean Mabuchi energy
$$M_\Delta^{\mathrm{NA}}(\mathcal{X},\mathcal{L})=H_\Delta^{\mathrm{NA}}(\mathcal{X},\mathcal{L})+(\mathcal{J}^{K_{X}+\Delta})^{\mathrm{NA}}(\mathcal{X},\mathcal{L}).$$
\end{itemize}
It is well-known that these invariants are independent of the choice of $\sigma\colon\mathcal{Y}\to\mathcal{X}$.
We say that a semiample test configuration $(\tilde{\mathcal{X}},\tilde{\mathcal{L}})$ {\it dominates} $(\mathcal{X},\mathcal{L})$ via $\mu$ if there exists a canonical morphism $\mu:\tilde{\mathcal{X}}\to\mathcal{X}$ 
such that $\tilde{\mathcal{L}}=\mu^*\mathcal{L}$.
It is proved in \cite[\S\S 7.5]{BHJ} that those functionals are {\it pullback invariant} as defined in \cite{BHJ} i.e.~if $(\tilde{\mathcal{X}},\tilde{\mathcal{L}})$ dominates $(\mathcal{X},\mathcal{L})$ then $J^{\mathrm{NA}}(\tilde{\mathcal{X}},\tilde{\mathcal{L}})=J^{\mathrm{NA}}(\mathcal{X},\mathcal{L})$
cf.~\cite[\S6, \S7]{BHJ} for example.
\end{de}

Note that $I^{\mathrm{NA}}$, $J^{\mathrm{NA}}$ and $I^{\mathrm{NA}}-J^{\mathrm{NA}}$ are nonnegative and equivalent norms to each other in the sense of \cite[Proposition 7.8]{BHJ}. $I^{\mathrm{NA}}-J^{\mathrm{NA}}$ is also introduced by \cite{Der} and called the minimum norm.
More precisely, the following holds.

\begin{lem}[cf.~{\cite[Proposition 7.8]{BHJ}}]\label{lem-equiv-norm}
    Let $(X,\Delta,L)$ be a polarized log subpair of dimension $n$ and $(\mathcal{X},\mathcal{L})$ be a normal ample test configuration for $(X,L)$.
    Then
    \begin{align*}
    0\le\frac{1}{n}J^{\mathrm{NA}}(\mathcal{X},\mathcal{L})\le I^{\mathrm{NA}}(\mathcal{X},\mathcal{L})-J^{\mathrm{NA}}(\mathcal{X},\mathcal{L})=(\mathcal{J}^L)^{\mathrm{NA}}(\mathcal{X},\mathcal{L})\le nJ(\mathcal{J}^L)^{\mathrm{NA}}(\mathcal{X},\mathcal{L}).
    \end{align*}
    Furthermore, $J^{\mathrm{NA}}(\mathcal{X},\mathcal{L})=0$ if and only if $\mathcal{X}$ is trivial.
\end{lem}


\begin{de}[K-stability and J-stability]\label{def-K-st}
Let $(X,\Delta,L)$ be a polarized log pair. $(X,\Delta,L)$ is called {\it K-stable} if 
\[
\mathrm{DF}_{\Delta}(\mathcal{X},\mathcal{L})>0
\]
for any non-trivial normal ample test configuration. On the other hand, $(X,\Delta,L)$ is called {\it uniformly K-stable} (resp, {\it K-semistable}) if there exists a constant $c>0$ (resp., $c\ge 0$) such that
\[
\mathrm{DF}_{\Delta}(\mathcal{X},\mathcal{L})\ge c J^{\mathrm{NA}}(\mathcal{X},\mathcal{L})
\]
for any semiample test configuration $(\mathcal{X},\mathcal{L})$ for $(X,L)$. The above condition is known by \cite[\S8]{BHJ} to be equivalent to that for any semiample normal test configuration $(\mathcal{X},\mathcal{L})$
\[
M^{\mathrm{NA}}_{\Delta}(\mathcal{X},\mathcal{L})\ge c J^{\mathrm{NA}}(\mathcal{X},\mathcal{L}).
\]
The definition of J$^H$-stability is of similar form (cf.~\cite{Hat}).
That is, we say that $(X,L)$ is uniformly J$^H$-stable (resp.~J$^H$-semistable) if there exists $c>0$ (resp.~$c\ge0$) such that 
\[
(\mathcal{J}^H)^{\mathrm{NA}}(\mathcal{X},\mathcal{L})\ge c J^{\mathrm{NA}}(\mathcal{X},\mathcal{L}).
\]

\end{de}


Let us recall the scaling action to a non-Archimedean metric \cite[\S6.2]{BHJ}.

\begin{ntt}\label{rramamra}
Let $(\mathcal{X},\mathcal{L})$ be a normal semiample test configuration and $d\in\mathbb{Z}_{>0}$. Then let $\mathcal{X}^{(d)}$ be the {\it normalization of the base change} of $(\mathcal{X},\mathcal{L})$ by the $d$-th power map $\mathbb{P}^1\ni t\mapsto t^d\in\mathbb{P}^1$ and $\mathcal{L}^{(d)}$ be the pullback of $\mathcal{L}$. We call $(\mathcal{X}^{(d)},\mathcal{L}^{(d)})$ the normalized base change of $(\mathcal{X},\mathcal{L})$ by the $d$-th power map $\mathbb{P}^1\ni t\mapsto t^d\in\mathbb{P}^1$. 
It is well-known that $d\,M^{\mathrm{NA}}_{\Delta}(\mathcal{X},\mathcal{L})=M^{\mathrm{NA}}_{\Delta}(\mathcal{X}^{(d)},\mathcal{L}^{(d)})$ and $d\,\mathrm{DF}_{\Delta}(\mathcal{X},\mathcal{L})\ge\mathrm{DF}_{\Delta}(\mathcal{X}^{(d)},\mathcal{L}^{(d)})$ for any $d\in\mathbb{Z}_{>0}$ by \cite[Proposition 7.16]{BHJ}.
If $\mathcal{B}$ is the strict transform of $B_{\mathbb{P}^1}$ on $\mathcal{X}$, we denote the strict transform of $B_{\mathbb{P}^1}$ on $\mathcal{X}^{(d)}$ by $\mathcal{B}^{(d)}$, where $B$ is a $\mathbb{Q}$-divisor on $X$.
\end{ntt}

Recall the definitions of the log canonical threshold, the $\alpha$-invariant and $\delta$-invariant.
\begin{de}
Let $(X,B)$ be a subpair and $D$ be an effective $\mathbb{Q}$-Cartier $\mathbb{Q}$-divisor. Suppose that there exists $l\in\mathbb{Q}$ such that $(X,B+lD)$ is sublc. Then we set the {\it log canonical threshold} of $(X,B)$ with respect to $D$ as
\[
\mathrm{lct}_{(X,B)}(D)=\inf\{r\in\mathbb{Q}|(X,B+rD) \,\mathrm{is}\,\mathrm{sublc}\}=\min_v\frac{A_{(X,B)}(v)}{v(D)},
\]
where $v$ runs over all divisorial valuations such that $v(D)\ne0$.
\end{de}

The $\alpha$-invariant was originally introduced by Tian \cite{T3} for Fano manifolds to detect the existence of K\"{a}hler-Einstein metrics, but this is very useful to detect K-stability for other kinds of polarized varieties.

\begin{de}\label{defalpha}
Let $(X,B,L)$ be a polarized klt pair. Then the {\it alpha invariant} of $(X,B)$ with respect to $L$ is
\[
\alpha_{(X,B)}(L)=\inf_D \mathrm{lct}_{(X,B)}(D),
\] where $D$ runs over all elements of $|L|_{\mathbb{Q}}$, which is the set of effective $\mathbb{Q}$-Cartier $\mathbb{Q}$-divisors $D$ such that $D\sim_{\mathbb{Q}}L$.
If we set $T_L(v)=\sup_{D\in|L|_{\mathbb{Q}}}v(D)$, then we have
\[
\alpha_{(X,B)}(L)=\inf_v\frac{A_{(X,B)}(v)}{T_L(v)},
\]
where $v$ runs over all divisorial valuations.
\end{de}

\begin{de}[\cite{FO}, \cite{BlJ}]\label{defdelta}
Let $(X,B,L)$ be a polarized klt pair. 
Take $r\in\mathbb{Z}_{>0}$ such that $rL$ is a line bundle. 
For any $m\in\mathbb{Z}_{>0}$, $D\in |L|_{\mathbb{Q}}$ is called a divisor of {\it $mr$-basis type} if there exists a basis $\{f_j\}_{j=1}^{h^0(X,mrL)}\subset H^0(X,mrL)$ such that 
\[
D=(mrh^0(X,mrL))^{-1}\sum_{j=1}^{h^0(X,mrL)}\mathrm{div}_{mrL}(f_j).
\]
Here, $\mathrm{div}_{mrL}(f_j)$ is an effective Cartier divisor defined by $f_j$ as a section of $H^0(X,mrL)$.
We denote by $|L|_{mr\text{-basis}}$ the set of all $mr$-basis type divisors $D$ with respect to $L$.
We set $S^{mr}_{L}(v):=\sup_{D\in|L|_{mr\text{-basis}}}v(D)$ for any divisorial valuation $v$ of $X$.
Then, we set
\[
\delta^{mr}_{(X,B)}(L):=\inf_{D\in|L|_{mr\text{-basis}}}\mathrm{lct}_{(X,B)}(D)=\inf_v\frac{A_{(X,B)}(v)}{S^{mr}_{L}(v)},
\]
where $v$ runs over all divisorial valuations on $X$.
By \cite[Theorem A]{BlJ}, $\{\delta^{mr}_{(X,B)}(L)\}_{m=1}^{\infty}$ converges and set the {\it delta invariant} of $(X,B)$ with respect to $L$ as
\[
\delta_{(X,B)}(L):=\lim_{m\to\infty}\delta^{mr}_{(X,B)}(L).
\]
Moreover, thanks to \cite[\S2]{BlJ}, we have that $\lim_{m\to\infty}S^{mr}_{L}(v)=S_L(v)$ for any divisorial valuation $v$ on $X$ and
\[
\delta_{(X,B)}(L)=\inf_v\frac{A_{(X,B)}(v)}{S_{L}(v)},
\]
where $v$ runs over all divisorial valuations on $X$ and $S_{L}(v):=\mathrm{vol}(L)^{-1}\int_0^{\infty}\mathrm{vol}(L-xv)dx$.
Here, $\mathrm{vol}(L-xv)$ means $\mathrm{vol}(\pi^*L-xcE)$, where $c\in\mathbb{Q}_{>0}$, $E$ is a prime divisor over $X$ and $\pi\colon Y\to X$ is a log resolution such that $v=c\,\mathrm{ord}_E$ and $E$ is a prime divisor on $Y$.
See \cite{BlJ} for details.
\end{de}
We note the following well-known facts to compute $\delta_{(X,B)}(L)$.
Let $(X,B,L)$ be a polarized klt pair.
For any $m\in\mathbb{Z}_{>0}$ such that $mL$ is a Cartier divisor, $\lambda\in\mathbb{Z}_{\ge0}$ and prime divisor $E$ over $X$, we set 
\[
\mathcal{F}_E^\lambda H^0(X,mL):=H^0(Y,m\sigma^*L-\lambda E),
\]
where $\sigma\colon Y\to X$ is a resolution of singularities such that $E$ is a prime divisor on $Y$.
Let $\{s_i\}_{i=1}^{h^0(X,mL)}$ be a basis of $H^0(X,mL)$.
If we can take a basis of $\mathcal{F}_E^\lambda H^0(X,mL)$ as a subset of $\{s_i\}_{i=1}^{h^0(X,mL)}$ for any $\lambda$, we say that $\{s_i\}_{i=1}^{h^0(X,mL)}$ is {\it compatible} with $E$.
In this situation, we have the following.
\begin{lem}\label{lem--compatible--basis}
\begin{enumerate}
\item For any $m\in\mathbb{Z}_{>0}$ such that $mL$ is a Cartier divisor and prime divisor $E$ over $X$, take a basis $\{s_i\}_{i=1}^{h^0(X,mL)}$ of $H^0(X,mL)$ compatible with $E$.
If $D$ is the $m$-basis type divisor defined by $\{s_i\}$, then
\[
S^{m}_L(\mathrm{ord}_{E})=\mathrm{ord}_{E}(D)=\frac{\sum_{\lambda\ge1}\mathrm{dim}\,\mathcal{F}^{\lambda}_EH^0(X,mL)}{mh^0(X,mL)}.
\]\label{lem--comp--(1)}
\item Let $E$ and $E'$ be prime divisors over $X$.
Then there exists a basis of $H^0(X,mL)$ compatible with both of $E$ and $E'$ for any $m\in\mathbb{Z}_{>0}$ such that $mL$ is a Cartier divisor.\label{lem--comp--(2)}
\end{enumerate}
\end{lem}

\begin{proof}
 (\ref{lem--comp--(1)}) follows from the proof of \cite[Lemma 2.2]{FO}.
 (\ref{lem--comp--(2)}) is nothing but \cite[Lemma 3.1]{AZ}. 
\end{proof}

If $v=\mathrm{ord}_F$ for some prime divisor $F$ over $X$, we will simply denote $T_L(F):=T_L(\mathrm{ord}_F)$.

We remark that for any $r\in\mathbb{Q}_{>0}$, we have that $r\delta_{(X,B)}(rL)=\delta_{(X,B)}(L)$ and $r\alpha_{(X,B)}(rL)=\alpha_{(X,B)}(L)$.
These invariants are important in the following sense. 

\begin{thm}[{\cite[Proposition 9.16]{BHJ}}, \cite{BlJ}]\label{thm--alpha}
Let $(X,B,L)$ be a klt polarized pair. Then the following hold.
\begin{enumerate}
\item $\alpha_{(X,B)}(L)>0$ and \[
H^{\mathrm{NA}}_{B}(\mathcal{X},\mathcal{L})\ge\alpha_{(X,B)}(L)I^{\mathrm{NA}}(\mathcal{X},\mathcal{L})
\]
for any semiample normal test configuration for $(X,L)$.
\item Suppose that $K_X+B=-L$. Then, $(X,B,L)$ is uniformly K-stable (resp., K-semistable) if and only if $\delta_{(X,B)}(L)>1$ (resp., $\ge1$).
\end{enumerate}
\end{thm}

\begin{ex}\label{ex-delta-alpha-curve}
    If $X$ is a curve, then $2\alpha_{(X,B)}(L)=\delta_{(X,B)}(L)$.
Indeed, for any closed point $P\in X$, we have that $$\frac{1}{\mathrm{deg}(L)}\int^{\mathrm{deg}(L)}_0\mathrm{deg}(L-xP)dx=\frac{\mathrm{deg}(L)}{2}.$$
This means that $2S_L(\mathrm{ord}_P)=T_L(\mathrm{ord}_P)$.
\end{ex}


From now, we introduce K-stability of algebraic fiber spaces.

\begin{de}[\cite{Hat2}]\label{algfib}\label{algfibk}
  $f\colon(X,\Delta,H)\to (C,L)$ is called a {\it polarized algebraic fiber space (sub)pair} if $(X,\Delta, H)$ is a polarized (sub)pair
  , $(C,L)$ is a polarized normal variety and $f\colon X\to C$ is a contraction, i.e.~$f$ is projective and $f_*\mathcal{O}_X\cong\mathcal{O}_C$.
Let $f:(X,\Delta,H)\to (C,L)$ be a polarized algebraic fiber space pair.
Then $(X,\Delta,H)$ is called {\it adiabatically K-(semi)stable over} $(C,L)$ if $(X,\Delta,\epsilon H+f^*L)$ is K-(semi)stable for sufficiently small $\epsilon\in\mathbb{Q}_{>0}$. 
Furthermore, $(X,\Delta,H)$ is called
\begin{itemize}
\item {\it uniformly adiabatically K-stable over} $(C,L)$ if there exist positive constants $\delta>0$ and $\epsilon_0>0$ such that $\mathrm{DF}_{\Delta}(\mathcal{X},\mathcal{M})\ge \delta J^{\mathrm{NA}}(\mathcal{X},\mathcal{M})$ for any $\epsilon\in(0,\epsilon_0)\cap\mathbb{Q}$ and semiample test configuration $(\mathcal{X},\mathcal{M})$ for $(X,\epsilon H+f^*L)$,
\item {\it adiabatically K-unstable over} $(C,L)$ if for any sufficiently small rational number $\epsilon>0$, there exists a semiample test configuration $(\mathcal{X},\mathcal{M})$ for $(X,\epsilon H+f^*L)$ such that $\mathrm{DF}_{\Delta}(\mathcal{X},\mathcal{M})<0$.
\end{itemize}

We say that $(X,H)$ {\it has cscK metrics adiabatically} if $X$ is smooth, $\Delta=0$ and there exists $\delta_0>0$ such that $(X,\delta H+f^*L)$ has a cscK metric for any $0<\delta<\delta_0$.
Here, we will simply denote $H+f^*L$ by $H+L$. 
\end{de}

\subsection{Log-twisted K-stability}Next, we introduce log-twisted K-stability as follows. 
 \begin{de}\label{ltp}
A quadruple $(X,B,T,L)$ is called a {\it polarized log-twisted pair} if $(X,B,L)$ is a polarized log pair and $T$ is a $\mathbb{Q}$-line bundle on $X$.

We call $(X,B,T,L)$ a {\it log-twisted Fano pair} if 
$L\sim_{\mathbb{Q}}-(K_X+B+T)$ and $(X,B)$ is klt.
\end{de}

We remark that if $B=0$, then log-twisted K-stability is nothing but twisted K-stability introduced in \cite{Der}.

\begin{de}\label{deflogtwisted}
Let $(X,B,T,L)$ be a polarized log-twisted pair and $(\mathcal{X},\mathcal{L})$ be a semiample test configuration for $(X,L)$. Then the {\it log-twisted Donaldson-Futaki invariant} is 
\[
\mathrm{DF}_{(B,T)}(\mathcal{X},\mathcal{L})=\mathrm{DF}_{B}(\mathcal{X},\mathcal{L})+(\mathcal{J}^T)^\mathrm{NA}(\mathcal{X},\mathcal{L}).
\]
We can also define the {\it log-twisted non-Archimedean Mabuchi functional} as
\[
M^\mathrm{NA}_{(B,T)}(\mathcal{X},\mathcal{L})=M^\mathrm{NA}_{B}(\mathcal{X},\mathcal{L})+(\mathcal{J}^T)^\mathrm{NA}(\mathcal{X},\mathcal{L}).
\]
 Then, $(X,B,T,L)$ is called (log-twisted)
\begin{itemize}
\item {\it K-semistable} if 
\[
M^{\mathrm{NA}}_{(B,T)}(\mathcal{X},\mathcal{L})\ge0
\]
for any normal semiample test configuration,
\item {\it K-stable} if
\[
M^{\mathrm{NA}}_{(B,T)}(\mathcal{X},\mathcal{L})>0
\]
for any non-trivial normal ample test configuration,
\item {\it uniformly K-stable} if there exists a positive constant $\epsilon>0$ such that 
\[
M^{\mathrm{NA}}_{(B,T)}(\mathcal{X},\mathcal{L})\ge \epsilon J^{\mathrm{NA}}(\mathcal{X},\mathcal{L})
\]
for any normal semiample test configuration.
\end{itemize} 

\end{de}

\begin{de}
Let $T$ be a $\mathbb{Q}$-line bundle on $X$.
 If $T$ is semiample, we say that $A\sim_{\mathbb{Q}}T$ is a {\it general effective} $\mathbb{Q}$-divisor if there exists $m\in\mathbb{Z}_{>0}$ such that $mT$ is globally generated and $mA\in|mT|$.
When $T$ is not semiample, $A\sim_{\mathbb{Q}}T$ is called {\it general} if there exist two general semiample $\mathbb{Q}$-divisors $A_1$ and $A_2$ such that $T\sim_{\mathbb{Q}}A_1-A_2$.
\end{de}

We regard log-twisted K-stability as the log K-stability when we choose a general $\mathbb{Q}$-divisor $T_1\sim_{\mathbb{Q}}T$ as follows.

\begin{lem}\label{dtw}
Let $(X,B,T,L)$ be a polarized log-twisted pair and $(\mathcal{X},\mathcal{L})$ be a normal semiample test configuration for $(X,L)$.
Fix a $\mathbb{Q}$-Cartier $\mathbb{Q}$-divisor $D\sim_{\mathbb{Q}}T$. 
If the support of $D$ does not contain the center of a divisorial valuation $v_E$ for any irreducible component $E\subset \mathcal{X}_0$ of $\mathcal{X}$, then
\begin{align*}
\mathrm{DF}_{(B,T)}(\mathcal{X},\mathcal{L})&=\mathrm{DF}_{B+D}(\mathcal{X},\mathcal{L})\\
M^\mathrm{NA}_{(B,T)}(\mathcal{X},\mathcal{L})&=M^\mathrm{NA}_{B+D}(\mathcal{X},\mathcal{L}).
\end{align*}
\end{lem}
\begin{proof}
Let $\pi:\tilde{\mathcal{X}}\to\mathcal{X}$ and $\rho:\tilde{\mathcal{X}}\to X_{\mathbb{P}^1}$ be the canonical morphisms.
Then
\begin{align*}
\mathrm{DF}_{B+D}(\mathcal{X},\mathcal{L})&=\frac{1}{L^n}\left((K_{\tilde{\mathcal{X}}/\mathbb{P}^1}+\mathcal{B}+\mathcal{D})\cdot \pi^*\mathcal{L}^n-\frac{n(K_X+B+T)\cdot L^{n-1}}{(n+1)L^n}\pi^*\mathcal{L}^{n+1}\right),\quad\text{and}\\
\mathrm{DF}_{(B,T)}(\mathcal{X},\mathcal{L})&=\frac{1}{L^n}\left((K_{\tilde{\mathcal{X}}/\mathbb{P}^1}+\mathcal{B}+\rho^*D_{\mathbb{P}^1})\cdot \pi^*\mathcal{L}^n-\frac{n(K_X+B+T)\cdot L^{n-1}}{(n+1)L^n}\pi^*\mathcal{L}^{n+1}\right),
\end{align*}
where $\mathrm{dim}\,X=n$ and $\mathcal{B}$ and $\mathcal{D}$ are the strict transforms of $B_{\mathbb{P}^1}$ and $D_{\mathbb{P}^1}$ on $\tilde{\mathcal{X}}$ respectively. 
Thus
\[
\mathrm{DF}_{(B,T)}(\mathcal{X},\mathcal{L})-\mathrm{DF}_{B+D}(\mathcal{X},\mathcal{L})=\frac{1}{L^n}(\rho^*D_{\mathbb{P}^1}-\mathcal{D})\cdot \pi^*\mathcal{L}^n.
\]
By assumption, $\rho^*D_{\mathbb{P}^1}-\mathcal{D}$ is $\pi$-exceptional and hence $(\rho^*D_{\mathbb{P}^1}-\mathcal{D})\cdot \pi^*\mathcal{L}^n=0$. 
Similarly, we can show the assertion for the non-Archimedean Mabuchi functionals.
\end{proof}
By Lemma \ref{dtw}, we can deduce the similar properties to the log K-stability and twisted K-stability in \cite{OS}, \cite{BHJ} and \cite{De2}.
\begin{thm}\label{thm--odaka-type}
If a polarized log-twisted pair $(X,B,T,L)$ is K-semistable, then $(X,B)$ is lc.
Furthermore, if $L=-(K_X+B+T)$, then $(X,B)$ is klt.
\end{thm}
\begin{proof}
We follow the notations in \cite{Hat2}.
We only prove the first assertion.
If $\lceil B\rceil$ is reduced, due to \cite[Theorem 1.1]{OX}, there exists a closed subscheme whose Rees valuations $v$ satisfy that $A_{(X,B)}(v)<0$. Then, $A_{(X,B+T_1)}(v)=A_{(X,B)}(v)<0$ for any general divisor $T_1\sim_\mathbb{Q}T$. Therefore, the theorem follows from the proof of \cite[Theorem 5.1]{Hat2}. The rest also follows from the same argument of \cite[Theorem 5.1]{Hat2} and is left to the reader.
\end{proof}

Thanks to \cite[Proposition 9.16]{BHJ}, the following holds as in the logarithmic case. 
\begin{prop}[cf., {\cite[Corollaries 9.3 and 9.4]{BHJ}}, {\cite[Theorems 3.21, 3.23]{De2}}]\label{bhj9394}
Let $(X,B,T,L)$ be a polarized log-twisted pair. Then the following hold.
\begin{enumerate}
\item If $K_X+B+T\equiv0$ and $(X,B)$ is klt (resp., lc), then $(X,B,T,L)$ is uniformly K-stable (resp., K-semistable).
\item If $K_X+B+T\equiv L$ and $(X,B)$ is lc, then $(X,B,T,L)$ is uniformly K-stable.
\end{enumerate}
\end{prop}
The proofs are left to the readers.

We can consider a special log-twisted test configuration as follows.

\begin{de}\label{sfp}
Let $(X,B,T,L=-K_X-B-T)$ be a log-twisted Fano pair, $T$ is semiample and $(\mathcal{X}^{\mathrm{s}},\mathcal{L}^{\mathrm{s}})$ be a normal ample test configuration.
Let $\rho^{\mathrm{s}}\colon\mathcal{X}^\mathrm{s}\dashrightarrow X_{\mathbb{P}^1}$ be the canonical birational morphism and $\mathcal{B}^\mathrm{s}:=(\rho^{\mathrm{s}})^{-1}_*B_{\mathbb{P}^1}$.
Then $(\mathcal{X}^{\mathrm{s}},\mathcal{L}^{\mathrm{s}})$ is a {\it log-twisted special test configuration} for $(X,B,T,L)$ if $(\mathcal{X}^{\mathrm{s}},\mathcal{B}^{\mathrm{s}}+(\rho^{\mathrm{s}})^{-1}_*T_{1,\mathbb{P}^1}+\mathcal{X}_0)$ is plt (cf., \cite[2.34]{KoMo}) and $\mathcal{L}^{\mathrm{s}}\sim_{\mathbb{Q},\mathbb{P}^1}-(K_{\mathcal{X}^{\mathrm{s}}}+\mathcal{B}^{\mathrm{s}}+(\rho^{\mathrm{s}})^{-1}_*T_{1,\mathbb{P}^1})$ for any general effective $\mathbb{Q}$-divisor $T_1\sim_{\mathbb{Q}} T$. 
\end{de}

For fundamental properties of log-twisted pairs, see Appendix \ref{appendices}.

\subsection{Lc-trivial fibrations}\label{subsec-lc-triv}
We shall recall the canonical bundle formula (cf, \cite{FM2}, \cite{A} and \cite{FG}).
 
\begin{de}[Klt-trivial fibrations]\label{ambro model}
Let $S$ be a quasi-projective normal variety.
Let $f:X\to C$ be a proper surjective morphism of normal varieties projective over $S$ such that $f_*\mathcal{O}_X\cong\mathcal{O}_C$.
$K(X)$ denotes the sheaf of quotient fields of $X$.
For any $\mathbb{Q}$-divisor $\Delta$ on $X$,
$f\colon(X,\Delta)\to C$ is called a {\it sublc-trivial fibration} (resp.~{\it subklt-trivial fibration}) if $(X,\Delta)$ is a sublc (resp. subklt) pair satisfying the following properties:
\begin{itemize}
\item $ K_{X}+\Delta\sim_{C,\mathbb{Q}}0$,
\item rank $f_*\mathcal{O}_X(\lceil \mathbf{A}^*(X,\Delta)\rceil)=1$,
\item $(X,\Delta)$ is sublc (resp. subklt) over the generic point of $C$,
\end{itemize} 
where we define $\mathcal{O}_X(\lceil \mathbf{A}^*(X,\Delta)\rceil)$ as follows (cf.~\cite[Lemma 3.22]{fujino-bpf}).
Let $\pi:Y\to X$ be a log resolution of $(X,\Delta)$ and suppose that $K_Y=\pi^*(K_X+\Delta)+\sum_ia_iE_i$ and $\pi_*(\sum_ia_iE_i)=-\Delta$. Then we set
 \[
\mathcal{O}_X(\lceil \mathbf{A}^*(X,\Delta)\rceil)=\pi_*\mathcal{O}_Y(\sum_{a_i\ne-1}\lceil a_i\rceil E_i).
\]
The definition of $\mathcal{O}_X(\lceil \mathbf{A}^*(X,\Delta)\rceil)$ is independent of the choice of $\pi$ and hence $\mathcal{O}_X(\lceil \mathbf{A}^*(X,\Delta)\rceil)$ is coherent.
We say that $f$ is a {\it lc (resp., klt)-trivial fibration} if $(X,\Delta)$ is a lc (resp., klt) pair and $K_{X}+\Delta\sim_{C,\mathbb{Q}}0$. It is easy to see that if $f$ is a lc (resp., klt)-trivial fibration, then $f$ is a sublc (resp., subklt)-trivial fibration. 
We remark that sublc-trivial fibrations are usually called lc-trivial fibrations (e.g.~in \cite{FG}). 
Furthermore, if $f\colon(X,\Delta,H)\to (C,L)$ is a polarized algebraic fiber space pair, then we say that $f\colon(X,\Delta,H)\to (C,L)$ is a {\it polarized lc (resp.~klt)-trivial fibration}. 
\end{de}

 If $f$ is a sublc-trivial fibration, then we set the {\it canonical bundle formula} as follows (cf.~\cite{A}, \cite{FG}). There exist a positive integer $b\in\mathbb{Z}_{>0}$ and a non-zero rational function $\varphi\in K(X)$ such that
\begin{equation}\label{eq--canbdleform}
K_{X}+\Delta+\frac{1}{b}(\varphi)= f^*(K_C+M_{C}+B_{C}),
\end{equation}
where $(\varphi)$ is the principal divisor and $M_{C}$ and $B_{C}$ are $\mathbb{Q}$-divisors uniquely determined as follows. 
$B_C$ is called the {\it discriminant divisor} of $f$ defined by the equation
\[
B_C=\sum_D(1-\mathrm{lct}_{(X_D,\Delta|_{X_D})}(f^*D))D,
\]
where $X_D\subset X$ is a sufficiently small neighborhood of the fiber of the generic point of $D$ and $D$ runs over all prime divisors on $C$.
$M_C$ is defined in the way that (\ref{eq--canbdleform}) holds and called the {\it moduli divisor}. 
We will simply denote the formula (\ref{eq--canbdleform}) by 
$$K_{X}+\Delta\sim_{\mathbb{Q}} f^*(K_C+M_{C}+B_{C})$$
by omitting $\frac{1}{b}(\varphi)$. 
For any birational morphism $h:C'\to C$ of normal varieties projective over $S$, we can take a sublc-trivial fibration $f':(X',\Delta')\to C'$ such that there exists a birational morphism $\tilde{h}\colon X'\to X$ making the following diagram commute
\[\xymatrix{
X' \ar[d]^{f'} \ar[r]^{\tilde{h}} \ar@{}[dr]|\circlearrowleft &X \ar[d]^{f} \\
C' \ar[r]^{h} & C \\
}\]
and satisfying that $K_{X'}+\Delta'=\tilde{h}^*(K_X+\Delta)$ and $\tilde{h}_*\Delta'=\Delta$. Then, we take $M_{C'}$ as the following holds
\[
K_{X'}+\Delta'+\frac{1}{b}(\varphi)= f'^*(K_{C'}+M_{C'}+B_{C'})
\]
and hence $M_C=h_*M_{C'}$, $B_C=h_*B_{C'}$ and $K_{C'}+M_{C'}+B_{C'}=h^*(K_C+M_C+B_C)$. 
In this case, we call $M_{C'}$ (resp.~$B_{C'}$) the moduli (resp.~discriminant) divisor {\it defined} by $M_C$ (resp.~$B_{C}$) and $f$ on $C'$ and one can see that $M_{C'}$ and $B_{C'}$ do not depend on the choice of $f'$.
For any birational map $h'':C''\dashrightarrow C$, we say that a $\mathbb{Q}$-divisor $\mathcal{M}$ on $C''$ is {\it the moduli divisor defined by} $M_C$ and $f$ if the following holds.
Suppose that there exist proper birational morphisms $p\colon C'\to C''$ and $h\colon C'\to C$. 
Let $M_{C'}$ be the moduli divisor defined by $M_C$ and $f$.
Then, $\mathcal{M}=p_*M_{C'}$.
In this situation, we define the following.
\begin{de}[Ambro model]\label{de-ambro}
 We say that $M_C$ is {\it stable under birational base change} if $M_C$ is $\mathbb{Q}$-Cartier and
\[
M_{C'}=h^*M_C
\]
for any birational morphism $h:C'\to C$, where $M_{C'}$ is the moduli divisor defined by $M_C$ and $f$.
Then, we call $C$ an {\it Ambro model} for $f$. If moreover, $M_C$ is semiample, we call $C$ a {\it good Ambro model} in this paper. Thanks to \cite[Theorem 3.6]{FG}, there exists an Ambro model $C'$ dominating $C$ for any sublc-trivial fibration and then the moduli divisor is nef. For subklt-trivial fibrations, this fact was originally proved in \cite{Am}.

It immediately follows from the same argument of \cite[\S 5]{A} that if $f:(X,\Delta)\to C$ is a sublc-trivial fibration over a curve $C$, then $C$ is an Ambro model.

Let $f\colon(X,\Delta)\to C$ be an lc (resp.~klt)-trivial fibration with ample $\mathbb{Q}$-divisors $H$ and $L$ on $X$ and $C$ respectively.
Then we call $f\colon(X,\Delta,H)\to (C,L)$ a {\it polarized lc (resp.~klt)-trivial fibration}.
\end{de}

\begin{rem}
A sublc-trivial fibration $f:(X,\Delta)\to C$ is usually called an Ambro model if $f$ is an Ambro model in the sense of this paper and $(C,B_C)$ is a log smooth subpair (cf.~\cite{FL}). 
In this paper, we do not assume that the discriminant divisor $B_C$ is snc or $C$ is smooth.

    It is well-known to experts that $\mathbf{A}^*(X,\Delta)$, $\{M_C\}_C$ and $\{B_C\}_C$ are naturally considered to be {\em b}-divisors.
\end{rem}

The following conjecture is still open in general setting.
 
\begin{conj}[B-semiampleness conjecture, cf.~{\cite[\S3]{FL}}]
For any sublc-trivial fibration, there exists a good Ambro model.
\end{conj}

When the base $C$ is a smooth curve, we know that $C$ is a good Ambro model (cf.~\cite[Theorem 0.1]{A}).
We are particularly interested in Ambro models such that the bases are log-twisted Fano and we can show the following in such cases,
 
\begin{lem}\label{abund}
Let $f:(X,\Delta)\to C$ be a klt-trivial fibration and $M$ and $B$ be the moduli divisor and the discriminant divisor on $C$ respectively. If $C$ is an Ambro model and $-(K_C+B+M)$ is big and nef, then $C$ is also a good Ambro model.
\end{lem}
 
\begin{proof}
Thanks to \cite{Am}, $M$ is nef. Therefore, if $-(K_C+B+M)$ is big and nef, then so is $-(K_C+B)$. On the other hand, $(C,B)$ is klt due to \cite[Theorem 3.1]{A}. Thus, $M$ is semiample by \cite[Theorem 3.3]{KoMo}.
\end{proof}

Next, we recall the notion of smooth elliptic surfaces.

\begin{de}
    Let $f\colon X\to C$ be a contraction such that $K_X\sim_{\mathbb{Q},C}0$.
    If $C$ is a proper smooth curve and $X$ is a smooth surface, then we call $f\colon X\to C$ (or $X$) a smooth {\it elliptic surface}.
    If $X$ is further birational to $\mathbb{P}^2$, then we call $X$ a rational elliptic surface.
\end{de}

Let $f:X\to C$ be a smooth elliptic surface. 
Thanks to Kodaira's canonical bundle formula (cf., \cite{FM2}), we have
\[
K_{X}\sim_{\mathbb{Q}}f^*\left(K_C+\frac{1}{12}J^*\mathcal{O}_{\mathbb{P}^1}(1)+\sum_{Q\in C} (1-\mathrm{lct}_{X}(f^*(Q)))(Q)\right),
\]
where $J:C\to \mathbb{P}^1$ is the {\it functional invariant} (cf., \cite[V.8]{BHPV}).
That is, $\frac{1}{12}J^*\mathcal{O}_{\mathbb{P}^1}(1)$ is the moduli divisor of $f$.
Kodaira classified singular fibers of elliptic surfaces and we follow the notations of \cite[V. \S7]{BHPV}. 
Let $F$ be a fiber of $f$. 
Then the following hold.
\begin{itemize}
\item if $F$ is a smooth or $I_N$-type fiber for some $N\ge 0$, then $\mathrm{lct}_X(F)=1$.
\item if $F$ is a $II$-type fiber, then $\mathrm{lct}_X(F)=\frac{5}{6}$.
\item if $F$ is a $III$-type fiber, then $\mathrm{lct}_X(F)=\frac{3}{4}$.
\item if $F$ is a $IV$-type fiber, then $\mathrm{lct}_X(F)=\frac{2}{3}$.
\item if $F$ is an $I_N^*$-type fiber for some $N\ge 0$, then $\mathrm{lct}_X(F)=\frac{1}{2}$.
\item if $F$ is a $II^*$-type fiber, then $\mathrm{lct}_X(F)=\frac{1}{6}$.
\item if $F$ is a $III^*$-type fiber, then $\mathrm{lct}_X(F)=\frac{1}{4}$.
\item if $F$ is a $IV^*$-type fiber, then $\mathrm{lct}_X(F)=\frac{1}{3}$.
\item if $F$ is a multiple fiber with multiplicity $m$, i.e.~$_mI_N$-type for some $N\ge0$, then $\mathrm{lct}_X(F)=\frac{1}{m}$.
\end{itemize}

Let $f:X\to C$ be a smooth rational elliptic surface.
It is well-known that then $C=\mathbb{P}^1$ and $X$ is isomorphic to a blow up of $\mathbb{P}^2$ at nine points (cf., \cite[Theorem 5.6.1]{CD}). 
Moreover, $X$ has at most one multiple fiber (cf.~\cite[Proposition 3.23]{FM}). 
\begin{de}\label{de-mult}
Let $f\colon X\to\mathbb{P}^1$ be a rational elliptic surface.
We set the {\it multiplicity} $m(X)$ of $X$ as follows.
If $X$ has a $_mI_N$-type fiber for some $m\ge2$, we set $m(X)=m$. Otherwise, we set $m(X)=1$.
\end{de}
By \cite[Proposition 3.23]{FM}, we see that $C=\mathbb{P}^1$ and $\mathrm{deg}\,(M+B)=\frac{2m(X)-1}{m(X)}$ if $M$ (resp.~$B$) is the moduli (resp.~discriminant) divisor on $\mathbb{P}^1$. 
Furthermore, $(\mathbb{P}^1,B,M,-(K_{\mathbb{P}^1}+M+B))$ is a log-twisted Fano pair. 


As klt-trivial fibrations, if $H$ is an $f$-ample divisor on a rational elliptic surface (resp.~rational Weierstrass fibration) $f\colon X\to\mathbb{P}^1$, we say that $f\colon(X,H)\to\mathbb{P}^1$ is a polarized rational elliptic surface (resp.~polarized rational Weierstrass fibration).
We always regard $f\colon(X,H)\to\mathbb{P}^1$ as a polarized klt-trivial fibration $f\colon(X,0,H)\to(\mathbb{P}^1,\mathcal{O}_{\mathbb{P}^1}(1))$.

\section{Criterion for J-stability without test configurations}\label{subsec-j-st}

In this section, we prove Theorem \ref{thm--first--ingredient} and Theorem \ref{dd} for the case when $K_X+\Delta$ is nef.
For this, we partially extend the result of \cite{C} on J-stability to the singular case.
We note that there is a quite useful criterion \cite[Theorem 1.1]{C} for J$^H$-stability without using test configurations for projective manifolds when $H$ is ample.
First, we recall the following notions (cf.~ \cite[Definition 1.1]{S}).

\begin{de}\label{de-j-positive}
    Let $(X,L)$ be a polarized variety of dimension $n$ and $H$ an $\mathbb{R}$-line bundle on $X$.
    We say that $(X,L)$ is J$^H$-{\it nef} if for any $p$-dimensional subvariety $V\subset X$, where $1\le p\le n$, 
   \begin{equation}\label{eq-j-nef}
   \left(n\frac{H\cdot L^{n-1}}{L^n}L-pH\right)\cdot L^{p-1}\cdot V\ge 0.
   \end{equation}
    If there exists $\epsilon>0$ such that $(X,L)$ is J$^{H-\epsilon L}$-nef, then we say that $(X,L)$ is uniformly J$^H$-positive.
   We note that the J$^{H-\epsilon L}$-nefness is equivalent to that
    \begin{equation}\nonumber
   \left(n\frac{H\cdot L^{n-1}}{L^n}L-pH\right)\cdot L^{p-1}\cdot V\ge \epsilon(n-p)L^p\cdot V
   \end{equation}
    for any subvariety $V\subset X$ of dimension $1\le p\le n-1$ since it holds that \begin{equation}\label{eq-j-pos}
        \left(n\frac{H_s\cdot L^{n-1}}{L^n}L-pH_s\right)\cdot L^{p-1}\cdot V=\left(n\frac{H\cdot L^{n-1}}{L^n}L-pH+s(n-p)L\right)\cdot L^{p-1}\cdot V,
    \end{equation}
    where we set $H_s:=H+sL$ for $s\in\mathbb{R}$.
\end{de}

The following is the partial generalization of \cite[Theorem 1.1]{C} for the singular case.

\begin{thm}[cf.~{\cite[Theorem 1.1]{C}}]\label{thm-jst}
    Let $(X,L)$ be an $n$-dimensional polarized normal variety with an ample (resp.~nef) $\mathbb{Q}$-line bundle $H$.
    Then $(X,L)$ is uniformly $\mathrm{J}^H$-stable (resp.~$\mathrm{J}^H$-semistable) if and only if $(X,L)$ is uniformly $\mathrm{J}^H$-positive (resp.~$\mathrm{J}^H$-nef).
\end{thm}
This theorem was first shown by \cite{C} for K\"{a}hler manifolds, and \cite{DP} and \cite{S} showed that uniform $\mathrm{J}^H$-positivity is equivalent to that
\[
\left(n\frac{H\cdot L^{n-1}}{L^n}L-pH\right)\cdot L^{p-1}\cdot V>0
\]
for any subvariety $V\subset X$.
To show Theorem \ref{thm-jst}, we prepare the following lemma.

\begin{lem}\label{lem-j-nef}
Let $(X,L)$ be a polarized normal variety of dimension $n$ and let $H$ be a $\mathbb{Q}$-line bundle.
If $(X,L)$ is $\mathrm{J}^H$-semistable, then $(X,L)$ is $\mathrm{J}^H$-nef.
\end{lem}

\begin{proof}
It suffices to show that if $(X,L)$ is not $\mathrm{J}^H$-nef, then $(X,L)$ is $\mathrm{J}^H$-unstable.
Suppose that $V$ is a $p$-dimensional subvariety such that
\[
\left(n\frac{H\cdot L^{n-1}}{L^n}L-pH\right)\cdot L^{p-1}\cdot V< 0.
\]
Let $\rho\colon\mathcal{X}\to X_{\mathbb{P}^1}$ be the normalization of the blow up along $V\times\{0\}$ and let $E$ be the inverse image of $V\times\{0\}$.
Let $\hat{X}$ be the strict transform of $X\times\{0\}$ and note that $E+\hat{X}=\mathcal{X}_0$ as Cartier divisors.
We see that the natural morphism $\pi\colon\hat{X}\to X$ is the blow up along $V$ and $\pi_*(D^{k})=0$ as a cycle for $k<n-p$, where $D=E|_{\hat{X}}$.
Note also that $\sigma^*L_{\mathbb{P}^1}^{n+1}=\sigma^*L_{\mathbb{P}^1}^{n}\cdot \sigma^*H_{\mathbb{P}^1}=0$.
Thus, we see as \cite[Lemma 7.4]{BHJ} that
\[
(\sigma^*L_{\mathbb{P}^1}-\delta E)^{n+1}=-\delta\sum_{j=0}^nE\cdot(\sigma^*L_{\mathbb{P}^1}-\delta E)^{j}\cdot \sigma^*L_{\mathbb{P}^1}^{n-j}=\delta\sum_{j=0}^n((\pi^*L-\delta D)^{j}\cdot \pi^*L^{n-j})-(L^n))
\]
for any sufficiently small rational $\delta>0$.
Similarly, we have that
\[
\sigma^*H_{\mathbb{P}^1}\cdot(\sigma^*L_{\mathbb{P}^1}-\delta E)^{n}=\delta\sum_{j=0}^{n-1}(\pi^*H\cdot(\pi^*L-\delta D)^{j}\cdot \pi^*L^{n-1-j}-(H\cdot L^{n-1})).
\]
Then, we have that
\begin{align*}
&(L^n)(\mathcal{J}^H)^\mathrm{NA}(\mathcal{X},\sigma^*L_{\mathbb{P}^1}-\delta E)\\
&= \delta \pi^*H\cdot \left(\sum_{i=0}^{n-1} (\pi^*L-\delta D)^i\cdot \pi^*L^{n-i-1}\right) -\delta \frac{nH\cdot L^{n-1}}{(n+1)L^n}\sum_{i=0}^{n} (\pi^*L-\delta D)^i \cdot \pi^*L^{n-i} \\
&= \frac{n!\delta^{n-p+1}}{(n-p+1)!p!}(-1)^{n-p-1}  D^{n-p}\cdot \left(n\frac{H\cdot L^{n-1}}{L^n}\pi^*L^p-p\pi^*H\cdot \pi^*L^{p-1}\right)+O(\delta^{n-p+2}).
\end{align*}
By \cite[\S4.3]{F}, there exists a positive constant $e>0$ such that $eV=(-1)^{n-p-1}\pi_*(D^{n-p})$ as a cycle. Thus,
\[
(-1)^{n-p-1}D^{n-p}\cdot\left(n\frac{H\cdot L^{n-1}}{L^n}\pi^*L^p-p\pi^*H\cdot \pi^*L^{p-1}\right)=e\left(n\frac{H\cdot L^{n-1}}{L^n}L^p-pH\cdot L^{p-1}\right)\cdot V<0.
\]
Therefore, $(\mathcal{J}^H)^\mathrm{NA}(\mathcal{X},\sigma^*L_{\mathbb{P}^1}-\delta E)<0$ for any sufficiently small $\delta>0$ and this means that $(X,L)$ is J$^H$-unstable.
\end{proof}

\begin{proof}[Proof of Theorem \ref{thm-jst}]
We make use of \cite[Theorem 1.1]{C}.
First,
we note that for any normal semiample test configuration $(\mathcal{X},\mathcal{L})$ for $(X,L)$, $(\mathcal{J}^{H})^{\mathrm{NA}}(\mathcal{X},\mathcal{L})$ is linear with respect to $H$.
Thus, it is easy to see that uniform J$^{H}$-stability is equivalent to J$^{H{-\epsilon}L}$-semistability for any sufficiently small $\epsilon>0$.
To show Theorem \ref{thm-jst}, we assert that it suffices to show that J$^H$-nefness implies that J$^H$-semistability when $H$ is nef.
Indeed, we see that J$^H$-semistability implies that J$^H$-nefness by Lemma \ref{lem-j-nef}.
Thus, uniform J$^H$-stability implies that uniform J$^H$-positivity.
On the other hand, if $H$ is ample, $(X,L)$ is uniformly J$^H$-positive and the first assertion holds for $(X,L)$, we see that there exists $\epsilon>0$ such that $H-\epsilon L$ is nef and $(X,L)$ is $\mathrm{J}^{H-\epsilon L}$-nef.
Then, we see that $(X,L)$ is $\mathrm{J}^{H-\epsilon L}$-semistable and hence uniformly $\mathrm{J}^{H}$-stable.

When $X$ is smooth, it follows from \cite[Theorem 1.1]{C} that $(X,L)$ is uniformly $\mathrm{J}^H$-stable if and only if $(X,L)$ is uniformly $\mathrm{J}^H$-positive (cf.~Definition \ref{de-j-positive}).
In this situation, we assert that if $(X,L)$ is $\mathrm{J}^H$-nef and $H$ is nef, then $(X,L)$ is $\mathrm{J}^H$-semistable. 
Then we see that $(X,L)$ is uniformly $\mathrm{J}^{H+\epsilon L}$-positive for any $\epsilon>0$ and hence uniformly $\mathrm{J}^{H+\epsilon L}$-stable.
Thus, $(X,L)$ is $\mathrm{J}^H$-semistable.

Next, we deal with the case when $X$ is not smooth.
By what we stated in the first paragraph, it suffices to show that J$^H$-nefness implies J$^H$-semistability of $(X,L)$.
Let $\pi\colon Y\to X$ be a resolution of singuralities and take an ample divisor $A$ on $Y$.
Let $(\mathcal{X},\mathcal{L})$ be a semiample test configuration for $(X,L)$ with a canonical morphism $\rho\colon\mathcal{X}\to X_{\mathbb{P}^1}$ and $\mathcal{Y}\subset\mathcal{X}\times_{X}Y$ be the normalization of the irreducible component dominant to $\mathcal{X}$.
Let $\sigma\colon\mathcal{Y}\to\mathcal{X}$ and $\rho_{\mathcal{Y}}\colon\mathcal{Y}\to Y_{\mathbb{P}^1}$ be the canonical morphisms.
Then it is easy to see that $(\mathcal{Y},\sigma^*\mathcal{L}+\epsilon\rho_{\mathcal{Y}}^*A_{\mathbb{P}^1})$ is also a semiample test configuration for $(Y,\pi^*L+\epsilon A)$ for any rational number $\epsilon>0$.
Set $H_{\delta,\epsilon}:=\pi^*H+\delta(\pi^*L+\epsilon A)$ and $L_{\epsilon}:=\pi^*L+\epsilon A$ for $\delta>0$.
We know that there exists a positive constant $C$ such that 
   \begin{equation}
   \left|n\frac{\pi^*H\cdot L_{\epsilon}^{n-1}}{L_{\epsilon}^n}-n\frac{H\cdot L^{n-1}}{L^n}\right|\le C\epsilon.\label{eq--C--estimate}
   \end{equation}
Since
\begin{equation}\label{eq--j-limit}
\lim_{\epsilon\to0}(\mathcal{J}^{\pi^*H+C\epsilon(\pi^*L+\epsilon A)})^{\mathrm{NA}}(\mathcal{Y},\sigma^*\mathcal{L}+\epsilon\rho_{\mathcal{Y}}^*A_{\mathbb{P}^1})=(\mathcal{J}^{H})^{\mathrm{NA}}(\mathcal{X},\mathcal{L}),
\end{equation}
it suffices to show that 
\begin{equation}\label{eq-j-nef-modif}
   \left(n\frac{H_{C\epsilon,\epsilon}\cdot L_{\epsilon}^{n-1}}{L_{\epsilon}^n}L_{\epsilon}-pH_{C\epsilon,\epsilon}\right)\cdot L_{\epsilon}^{p-1}\cdot W\ge 0
   \end{equation}
    for any small $\epsilon>0$ and subvariety $W\subset Y$ of dimension $1\le p\le n-1$.
    For this, we note by (\ref{eq-j-pos}) and (\ref{eq--C--estimate}) that 
    \begin{align*}
        &\left(n\frac{H_{C\epsilon,\epsilon}\cdot L_{\epsilon}^{n-1}}{L_{\epsilon}^n}L_{\epsilon}-pH_{C\epsilon,\epsilon}\right)\cdot L_{\epsilon}^{p-1}\cdot W\\
        &=\left(n\frac{\pi^*H\cdot L_{\epsilon}^{n-1}}{L_{\epsilon}^n}L_{\epsilon}-p\pi^*H\right)\cdot L_{\epsilon}^{p-1}\cdot W+C\epsilon(n-p)L_{\epsilon}^{p}\cdot W\\
        &=\left(n\frac{H\cdot L^{n-1}}{L^n}L_{\epsilon}-p\pi^*H\right)\cdot L_{\epsilon}^{p-1}\cdot W+\left(C\epsilon(n-p)+n\frac{\pi^*H\cdot L_{\epsilon}^{n-1}}{L_{\epsilon}^n}-n\frac{H\cdot L^{n-1}}{L^n}\right)L_{\epsilon}^{p}\cdot W\\
        &\ge\left(n\frac{H\cdot L^{n-1}}{L^n}L_{\epsilon}-p\pi^*H\right)\cdot L_{\epsilon}^{p-1}\cdot W.
    \end{align*}
   Furthermore,
   \begin{align*}
       &\left(n\frac{\pi^*H\cdot \pi^*L^{n-1}}{\pi^*L^n}L_{\epsilon}-p\pi^*H\right)\cdot L_{\epsilon}^{p-1}\cdot W\\
       &=\sum_{r}\epsilon^r\left(n\frac{H\cdot L^{n-1}}{L^n}\binom{p}{r}\pi^*L^{p-r}-p\binom{p-1}{r}\pi^*H\cdot \pi^*L^{p-r-1}\right)\cdot A^{r}\cdot  W\\
       &=\sum_{r}\epsilon^r\binom{p}{r}\left(n\frac{H\cdot L^{n-1}}{L^n}L-(p-r)H\right)\cdot L^{p-r-1}\cdot \pi_*(A^{r}\cdot  W).
   \end{align*}
   We can regard $\pi_*(A^{r}\cdot  W)$ as an effective algebraic cycle of dimension $p-r$ in $X$ and then we have
   \[
\left(n\frac{\pi^*H\cdot \pi^*L^{n-1}}{\pi^*L^n}L_{\epsilon}-p\pi^*H\right)\cdot L_{\epsilon}^{p-1}\cdot W\ge0   
   \]
   for any $\epsilon>0$ by the assumption that (\ref{eq-j-nef}) holds.
   Then (\ref{eq-j-nef-modif}) holds and $(Y,L_{\epsilon})$ is J$^{H_{C\epsilon,\epsilon}}$-semistable.
   Hence, we see that $(\mathcal{J}^{\pi^*H+C\epsilon(\pi^*L+\epsilon A)})^{\mathrm{NA}}(\mathcal{Y},\sigma^*\mathcal{L}+\epsilon\rho_{\mathcal{Y}}^*A_{\mathbb{P}^1})\ge0$ and conclude that Theorem \ref{thm-jst} holds by (\ref{eq--j-limit}).
\end{proof}
 We will make use of the following fact.
\begin{rem}\label{rem--useful}
Let $(X,L)$ be an $n$-dimensional polarized variety.
It is easy to see that if $H$ and
    \[
    n\frac{H\cdot L^{n-1}}{L^n}L-(n-1)H
    \] are nef, then $(X,L)$ is J$^H$-semistable by Theorem \ref{thm-jst}.
\end{rem}

Combining Theorems \ref{thm-jst} and \ref{bhjz}, we obtain Theorem \ref{thm--first--ingredient}, which is one of the key ingredients to show Theorem \ref{dd}.



\begin{proof}[Proof of Theorem \ref{thm--first--ingredient}]
    By Theorem \ref{bhjz}, we have that for any normal semiample test configuration $(\mathcal{X},\mathcal{L})$ for $(X,L)$, 
    \[
    H^{\mathrm{NA}}_\Delta(\mathcal{X},\mathcal{L})\ge \delta_{(X,\Delta)}(L)(\mathcal{J}^{L})^{\mathrm{NA}}(\mathcal{X},\mathcal{L}).
    \]
    This means that
    \[
    M^{\mathrm{NA}}_\Delta(\mathcal{X},\mathcal{L})\ge (\mathcal{J}^{H})^{\mathrm{NA}}(\mathcal{X},\mathcal{L}).
    \]
    We have that $(X,L)$ is uniformly J$^H$-stable by Theorem \ref{thm-jst}.
    Thus, we obtain the assertion.
\end{proof}

Finally, we show Corollary \ref{cor--minimal-stable} by Theorem \ref{thm-jst}.
This implies Theorem \ref{dd} for the case when $K_X+\Delta$ is nef.

\begin{cor}\label{cor--minimal-stable}
Let $(X,\Delta,H)$ be a polarized klt pair such that $K_X+\Delta$ is nef.   
Then for any sufficiently small rational number $\epsilon>0$, $(X,\Delta,\epsilon H+K_X+\Delta)$ is uniformly K-stable.

Furthermore, there exists $\beta>0$ and $\epsilon_0>0$ such that 
\[
M^{\mathrm{NA}}_\Delta(\mathcal{X},\mathcal{L})\ge \beta(\mathcal{J}^{\epsilon H+K_X+\Delta})^{\mathrm{NA}}(\mathcal{X},\mathcal{L})
\]
for any $\epsilon\in\mathbb{Q}\cap(0,\epsilon_0)$ and normal semiample test configuration $(\mathcal{X},\mathcal{L})$ for $(X,L)$.
\end{cor}

\begin{proof}
Put $n$ as the dimension of $X$.
First, we show the following claim.
\begin{claim}
    There exists a positive constant $C>0$ depending only on the intersection numbers $(K_X+\Delta)^i\cdot H^{n-i}$ for $0\le i\le n$ such that $(X,\epsilon H+K_X+\Delta)$ is $\mathrm{J}^{K_X+\Delta+C\epsilon(\epsilon H+K_X+\Delta)}$-semistable for any sufficiently small rational number $\epsilon>0$.
\end{claim}

\begin{proof}[Proof of Claim]
For any $p$-dimensional subvariety $V\subset X$, consider the following intersection number
\[
\left(n\frac{(K_X+\Delta)\cdot (\epsilon H+K_X+\Delta)^{n-1}}{(\epsilon H+K_X+\Delta)^n}(\epsilon H+K_X+\Delta)-p(K_X+\Delta)\right)\cdot(\epsilon H+K_X+\Delta)^{p-1}\cdot V.
\]
Set $$m:=\max\{l\,|\,\text{$(K_X+\Delta)^l\cdot H^{n-l}\ne0$}\}.$$
Then for any sufficiently small $\epsilon>0$, we have that
\[
n\frac{(K_X+\Delta)\cdot (\epsilon H+K_X+\Delta)^{n-1}}{(\epsilon H+K_X+\Delta)^n}=m+O(\epsilon),
\]
where the error term depends only on the intersection numbers $(K_X+\Delta)^i\cdot H^{n-i}$ for all $0\le i\le n$.
Thus, there exists a positive constant $C$ depending only on the intersection numbers $(K_X+\Delta)^i\cdot H^{n-i}$ such that
\begin{align*}
&\left(n\frac{(K_X+\Delta)\cdot (\epsilon H+K_X+\Delta)^{n-1}}{(\epsilon H+K_X+\Delta)^n}(\epsilon H+K_X+\Delta)-p(K_X+\Delta)\right)\cdot(\epsilon H+K_X+\Delta)^{p-1}\cdot V\\
&\ge((m-p)(K_X+\Delta)+m\epsilon H-C\epsilon(n-p)(\epsilon H+K_X+\Delta))\cdot(\epsilon H+K_X+\Delta)^{p-1}\cdot V.
\end{align*}
By this inequality, we see that if it holds that
\begin{equation}\label{eq--claim--j-cor}
    ((m-p)(K_X+\Delta)+m\epsilon H
    )\cdot(\epsilon H+K_X+\Delta)^{p-1}\cdot V\ge0,
\end{equation}
then $(X,\epsilon H+K_X+\Delta)$ is $\mathrm{J}^{K_X+\Delta+C\epsilon(\epsilon H+K_X+\Delta)}$-semistable by Theorem \ref{thm-jst} and (\ref{eq-j-pos}).
Therefore, we only have to show (\ref{eq--claim--j-cor}) for any sufficiently small $\epsilon>0$.
Set 
\[
\nu:=\max\{l\,|\,\text{$(K_X+\Delta)^l\cdot H^{p-l}\cdot V\ne0$}\}.
\]
It is not hard to see that $\nu\le\min\{m,p\}$.
Then, we have that 
\begin{align*}
    &((m-p)(K_X+\Delta)+m\epsilon H
    )\cdot(\epsilon H+K_X+\Delta)^{p-1}\cdot V\\
    &\ge m\epsilon^pH\cdot V+\sum_{i=1}^\nu\epsilon^{p-i}\left((m-p)\binom{p-1}{p-i}+m\binom{p-1}{p-i-1}\right)(K_X+\Delta)^i\cdot H^{p-i}\cdot V\\
    &=\sum_{i=0}^\nu\epsilon^{p-i}(m-i)\binom{p}{i}(K_X+\Delta)^i\cdot H^{p-i}\cdot V\ge0.
\end{align*}
Thus, we obtain (\ref{eq--claim--j-cor}) and complete the proof of Claim.
\end{proof}
Next, we show that Claim implies Corollary \ref{cor--minimal-stable}.
Let $\alpha:=\alpha_{(X,\Delta)}(H+K_X+\Delta)>0$.
We know by \cite[Proposition 2.1]{Fjtb} that 
\begin{align*}
\frac{n}{n+1}\delta_{(X,\Delta)}(\epsilon H+K_X+\Delta)\ge \alpha_{(X,\Delta)}(\epsilon H+K_X+\Delta)\ge\alpha
\end{align*}
Thus, we see that 
for any $\epsilon\in\mathbb{Q}\cap(0,\frac{(n+1)\alpha}{nC})$, $(X,\Delta,\epsilon H+K_X+\Delta)$ is uniformly K-stable by Claim and Theorem \ref{thm--first--ingredient}.
Furthermore, we have
\[
M^{\mathrm{NA}}_\Delta(\mathcal{X},\mathcal{L})\ge \frac{(n+1)\alpha}{2nC}(\mathcal{J}^{\epsilon H+K_X+\Delta})^{\mathrm{NA}}(\mathcal{X},\mathcal{L})
\]
for any $\epsilon\in\mathbb{Q}\cap(0,\frac{(n+1)\alpha}{2nC})$ and normal semiample test configuration $(\mathcal{X},\mathcal{L})$ for $(X,L)$.
\end{proof}

\section{Adiabatic K-stability versus log-twisted K-stability}\label{sec-3}

\subsection{Adiabatic K-stability implies log-twisted K-stability}\label{unstadiab}
In this subsection, we see that log-twisted K-stability of the base of a klt (lc)-trivial fibration (cf.~Definition \ref{ambro model}) is a necessary condition for adiabatic K-stability under some conditions (cf.~Theorems \ref{can} and \ref{stp1}).

First, we show the following stronger theorem than Theorem \ref{ff}.

\begin{thm}\label{can}
Let $f:(X,\Delta,H)\to (C,L)$ be a polarized lc-trivial fibration.
Let $M$ (resp.~$B$) be the moduli (resp.~discriminant) divisor. 

If $M$ is $\mathbb{Q}$-Cartier and $(C,B,M,L)$ is not log-twisted K-semistable, then $(X,\Delta,H)$ is adiabatically K-unstable over $(C,L)$.
\end{thm}


To show Theorem \ref{can}, we first prepare the following proposition on MMP.

\begin{prop}\label{model}
Let $f:(X,\Delta,H)\to (C,L)$ be a polarized lc-trivial fibration, $\pi:Y\to X$ be a log resolution of $(X,\Delta)$ and $\Theta:=\pi^{-1}_*\Delta+\mathrm{Ex}(\pi)$ be an snc $\mathbb{Q}$-divisor on $Y$. 
Let $\mathcal{Y}$ be a smooth test configuration for $Y$ dominating $Y_{\mathbb{P}^1}$ such that $(\rho_{\mathcal{Y}})^{-1}_*\Theta_{\mathbb{P}^1}+\mathcal{Y}_{0}$ is an snc divisor, where the canonical birational map $\rho_{\mathcal{Y}}\colon\mathcal{Y}\to Y_{\mathbb{P}^1}$ is a morphism. 
Let $\mathscr{C}$ be a normal test configuration for $C$ such that there exists a canonical birational morphism $\rho_{\mathscr{C}}\colon\mathscr{C}\to C_{\mathbb{P}^1}$. 
Suppose that the canonical rational map $\Pi\colon\mathcal{Y}\to\mathscr{C}$ is a morphism.

Then, there exist $\epsilon_0>0$ and two $\mathbb{G}_m$-equivariant birational contractions $h\colon\mathcal{Y}\dashrightarrow \mathcal{W}$ and $\xi:\mathcal{W}\dashrightarrow\mathcal{X}$, where $\mathcal{W}$ and $\mathcal{X}$ are normal varieties projective over $\mathscr{C}$, satifying the following.
\begin{itemize}
    \item $(\mathcal{W},\Delta_\mathcal{W})$ is a $\mathbb{Q}$-factorial $\mathbb{G}_m$-equivariant good minimal model of $(\mathcal{Y},(\rho_{\mathcal{Y}})^{-1}_*\Theta_{\mathbb{P}^1})$ over $\mathscr{C}$, where $\Delta_{\mathcal{W}}:=h_*(\rho_{\mathcal{Y}})^{-1}_*\Theta_{\mathbb{P}^1}$,
    \item if $\eta=\xi\circ h$, then $(\mathcal{X},\Delta_{\mathcal{X},\epsilon})$ is the $\mathbb{G}_m$-equivariant relative lc model of $(\mathcal{Y},(\rho_{\mathcal{Y}})^{-1}_*\Theta_{\mathbb{P}^1}+\epsilon(\pi\times\mathrm{id}_{\mathbb{P}^1}\circ\rho_{\mathcal{Y}})^*H_{\mathbb{P}^1})$ over $\mathscr{C}$ for any rational number $\epsilon\in(0,\epsilon_0)$. 
\end{itemize}
Furthermore, fix $\epsilon\in\mathbb{Q}\cap(0,\epsilon_0)$.
Then, there exist $r\in\mathbb{Q}_{>0}$ and an ample $\mathbb{Q}$-line bundle $\mathscr{M}$ on $\mathscr{C}$ such that $(\mathcal{X},\epsilon^{-1}(K_{\mathcal{X}}+\Delta_{\mathcal{X},\epsilon})+\Pi_{\mathcal{X}}^*\mathscr{M})$ is an ample test configuration for $(X,H+rL)$, where $\Pi_{\mathcal{X}}\colon\mathcal{X}\to\mathscr{C}$ is the canonical morphism.
\end{prop}
 
\begin{proof}
Since $H$ is ample, we consider $H$ to be a general effective $\mathbb{Q}$-divisor throughout this proof.
First, we note that $(Y,\Theta)$ has a good minimal model over $C$. 
Indeed, by the proof of the existence of a $\mathbb{Q}$-factorial dlt modification $(Z,\Delta_Z)$ over $X$ (cf.~\cite[4.1]{Fuj}), the $(K_Y+\Theta)$-MMP with ample scaling yields a good minimal model $(Z,\Delta_Z)$ of $(Y,\Theta)$ over $X$. 
Let $g:Z\to X$ be the structure morphism. 
Since $K_Z+\Delta_Z=g^*(K_X+\Delta)$ and $K_X+\Delta\sim_{\mathbb{Q},C}0$, $(Z,\Delta_Z)$ is also a good minimal model of $(Y,\Theta)$ over $C$. 
Hence, $(Z,\Delta_Z+\epsilon g^*H)$ is a good minimal model for any sufficiently small $\epsilon\in\mathbb{Q}_{>0}$. 
 This means that $(\mathcal{Y}\setminus\mathcal{Y}_0,(\rho_{\mathcal{Y}})^{-1}_*\Theta_{\mathbb{P}^1}\setminus\mathcal{Y}_0)\cong (Y\times(\mathbb{P}^1\setminus\{0\}),\Theta\times(\mathbb{P}^1\setminus\{0\}))$ also has a good minimal model over $\mathscr{C}\setminus \mathscr{C}_0\cong C\times(\mathbb{P}^1\setminus\{0\})$. 
Since $(\mathcal{Y},(\rho_{\mathcal{Y}})^{-1}_*\Theta_{\mathbb{P}^1}+\mathcal{Y}_{0,\mathrm{red}})$ is log smooth and dlt, every lc center of the log pair $(\mathcal{Y},(\rho_{\mathcal{Y}})^{-1}_*\Theta_{\mathbb{P}^1})$ meets $\mathcal{Y}\setminus\mathcal{Y}_0$. Therefore, we can apply \cite[Theorem 1.1]{HX} and there exists a good minimal model $(\mathcal{W},\Delta_{\mathcal{W}})$ of $(\mathcal{Y},(\rho_{\mathcal{Y}})^{-1}_*\Theta_{\mathbb{P}^1})$ over $\mathscr{C}$. 
We may assume that the contraction $h\colon\mathcal{Y}\dashrightarrow\mathcal{W}$ is the output of a minimal model program with ample scaling by \cite[Corollary 2.9]{HX}. 
In particular, we may assume that $\mathcal{W}$ is a $\mathbb{Q}$-factorial (cf.~\cite[4.3.5]{fujino-foundation}) and $h$ is $\mathbb{G}_m$-equivariant by the argument of \cite{A2} and \cite{LX} (cf.~Definition \ref{de-g-equiv-model}). 
We obtain the first assertion for $\mathcal{W}$.

We see that there exists $\epsilon'_1>0$ such that every step of the $K_{\mathcal{Y}}+(\rho_{\mathcal{Y}})^{-1}_*\Theta_{\mathbb{P}^1}$-MMP yielding $\mathcal{W}$ is also a divisorial contraction or a flip with respect to $K_{\mathcal{Y}}+(\rho_{\mathcal{Y}})^{-1}_*\Theta_{\mathbb{P}^1}+\epsilon(\pi\times\mathrm{id}_{\mathbb{P}^1}\circ\rho_{\mathcal{Y}})^*H_{\mathbb{P}^1}$ for any rational number $\epsilon\in(0,\epsilon'_1)$.
Thus, $(\mathcal{W},\Delta_{\mathcal{W}}+\epsilon h_*(\pi\times\mathrm{id}_{\mathbb{P}^1}\circ\rho_{\mathcal{Y}})^*H_{\mathbb{P}^1})$ is also dlt and has no lc center contained in $\mathcal{W}_0$ for any $\epsilon\in(0,\epsilon'_1)\cap\mathbb{Q}$ by \cite[Corollary 3.44]{KoMo}.
In this setting, we show the following claim.
\begin{claim}\label{claim-1}
There exists $\epsilon_1\in\mathbb{Q}\cap(0,\epsilon'_1)$ such that $(\mathcal{W},\Delta_{\mathcal{W}}+\epsilon h_*(\pi\times\mathrm{id}_{\mathbb{P}^1}\circ\rho_{\mathcal{Y}})^*H_{\mathbb{P}^1})$ has a good minimal model over $\mathscr{C}$ for any $\epsilon\in\mathbb{Q}\cap(0,\epsilon_1)$.
\end{claim}
\begin{proof}[Proof of Claim]
We note that every lc center of $(\mathcal{W},\Delta_{\mathcal{W}}+\epsilon h_*(\pi\times\mathrm{id}_{\mathbb{P}^1}\circ\rho_{\mathcal{Y}})^*H_{\mathbb{P}^1})$ meets $\mathcal{W}\setminus\mathcal{W}_0$.
By 
\cite[Theorem 1.1]{HX}, it suffices to show that
$(\mathcal{W}\setminus\mathcal{W}_0,(\Delta_{\mathcal{W}}+\epsilon h_*(\pi\times\mathrm{id}_{\mathbb{P}^1}\circ\rho_{\mathcal{Y}})^*H_{\mathbb{P}^1})\setminus\mathcal{W}_0)$ has a good minimal model over $\mathscr{C}\setminus\mathscr{C}_0$ for any $\epsilon\in\mathbb{Q}\cap(0,\epsilon'_1)$.
We first note that $(Z,\Delta_Z+\epsilon g^*H)$ is also a good minimal model of $(Y,\Theta+\epsilon\pi^*H)$ over $C$ for any sufficiently small $\epsilon\in\mathbb{Q}_{>0}$. 
Indeed, if $\psi:Y\dashrightarrow Z$ is the birational contraction, $A_{(Y,\Theta)}(F)<A_{(Z,\Delta_Z)}(F)$ for any $\psi$-exceptional divisor $F$. 
Thus, there exists $0<\epsilon_1<\epsilon_1'$ such
 that $A_{(Y,\Theta+\epsilon\pi^*H)}(F)<A_{(Z,\Delta_Z+\epsilon g^*H)}(F)$ also holds for any $\epsilon\in(0,\epsilon_1)\cap\mathbb{Q}$ and $\psi$-exceptional prime divisor $F$.
Furthermore, $K_Z+\Delta_Z+\epsilon g^*H$ is semiample over $C$. 
This means that $(Z\times(\mathbb{P}^1\setminus\{0\}),(\Delta_Z+\epsilon g^*H)\times(\mathbb{P}^1\setminus\{0\}))$ is a good minimal model of $(\mathcal{Y}\setminus\mathcal{Y}_0,((\rho_{\mathcal{Y}})^{-1}_*\Theta_{\mathbb{P}^1}+\epsilon(\pi\times\mathrm{id}_{\mathbb{P}^1}\circ\rho_{\mathcal{Y}})^*H_{\mathbb{P}^1})\setminus\mathcal{Y}_0)$ for any $\epsilon\in(0,\epsilon_1)\cap\mathbb{Q}$.
We assert that $(Z\times(\mathbb{P}^1\setminus\{0\}),(\Delta_Z+\epsilon g^*H)\times(\mathbb{P}^1\setminus\{0\}))$ is also a good minimal model of $(\mathcal{W}\setminus\mathcal{W}_0,(\Delta_{\mathcal{W}}+\epsilon h_*(\pi\times\mathrm{id}_{\mathbb{P}^1}\circ\rho_{\mathcal{Y}})^*H_{\mathbb{P}^1})\setminus\mathcal{W}_0)$.
Indeed, we note that $\mathcal{W}\setminus\mathcal{W}_0\dashrightarrow Z\times(\mathbb{P}^1\setminus\{0\})$ is isomorphic in codimension one points since they are good minimal models of $(\mathcal{Y}\setminus\mathcal{Y}_0,(\rho_{\mathcal{Y}})^{-1}_*\Theta_{\mathbb{P}^1}\setminus\mathcal{Y}_0)$ over $\mathscr{C}\setminus\mathscr{C}_0$ (cf.~\cite[Theorem 3.52 (2)]{KoMo}). 
Therefore, for any prime divisor $E$ over $\mathcal{Y}\setminus\mathcal{Y}_0$, we have
\[
A_{(\mathcal{W}\setminus\mathcal{W}_0,(\Delta_{\mathcal{W}}+\epsilon h_*(\pi\times\mathrm{id}_{\mathbb{P}^1}\circ\rho_{\mathcal{Y}})^*H_{\mathbb{P}^1})\setminus\mathcal{W}_0)}(E)\le A_{(Z\times(\mathbb{P}^1\setminus\{0\}),(\Delta_Z+\epsilon g^*H)\times(\mathbb{P}^1\setminus\{0\}))}(E)
\] by \cite[Proposition 3.51]{KoMo}.
This shows that $(Z\times(\mathbb{P}^1\setminus\{0\}),(\Delta_Z+\epsilon g^*H)\times(\mathbb{P}^1\setminus\{0\}))$ is also a good minimal model of $(\mathcal{W}\setminus\mathcal{W}_0,(\Delta_{\mathcal{W}}+\epsilon h_*(\pi\times\mathrm{id}_{\mathbb{P}^1}\circ\rho_{\mathcal{Y}})^*H_{\mathbb{P}^1})\setminus\mathcal{W}_0)$  for any $\epsilon\in\mathbb{Q}\cap(0,\epsilon_1)$.
\end{proof}

In this paragraph, we show the following:
There exists a $\mathbb{G}_m$-equivariant birational contraction $\xi:\mathcal{W}\dashrightarrow\mathcal{X}$ 
such that $\mathcal{X}$ is a lc model of $(\mathcal{W},\Delta_{\mathcal{W}}+\epsilon h_*(\pi\times\mathrm{id}_{\mathbb{P}^1}\circ\rho_{\mathcal{Y}})^*H_{\mathbb{P}^1})$ over $\mathscr{C}$ for any $\epsilon\in(0,\epsilon_0)\cap\mathbb{Q}$.
 We know that $(\mathcal{W},\Delta_{\mathcal{W}})$ is a good minimal model over $\mathscr{C}$ and set $$\mathcal{Z}:=\mathbf{Proj}_{\mathscr{C}}(\oplus_{l\ge0}\alpha_*\mathcal{O}_{\mathcal{W}}(lr_0(K_{\mathcal{W}}+\Delta_{\mathcal{W}}))),$$
    where $\alpha\colon\mathcal{W}\to\mathscr{C}$ is the canonical morphism and $r_0\in\mathbb{Z}_{>0}$ satisfies that $r_0(K_{\mathcal{W}}+\Delta_{\mathcal{W}})$ is Cartier. 
We note that $\mathcal{Z}$ has the canonical $\mathbb{G}_m$-action induced by the $\mathbb{G}_m$-action on $\oplus_{l\ge0}\alpha_*\mathcal{O}_{\mathcal{W}}(lr_0(K_{\mathcal{W}}+\Delta_{\mathcal{W}}))$ and that $\mathcal{Z}\setminus\mathcal{Z}_0\cong C\times(\mathbb{P}^1\setminus\{0\})$. 
Then, it follows from Claim that there exists a good minimal model of $(\mathcal{W}\setminus\mathcal{W}_0,(\Delta_{\mathcal{W}}+\epsilon h_*(\pi\times\mathrm{id}_{\mathbb{P}^1}\circ\rho_{\mathcal{Y}})^*H_{\mathbb{P}^1})|_{\mathcal{W}\setminus\mathcal{W}_0})$ over $C\times(\mathbb{P}^1\setminus\{0\})$ for any $\epsilon\in(0,\epsilon_1)\cap\mathbb{Q}$. 
By 
\cite[Theorem 1.1]{HX}, if we fix a rational number $\epsilon'\in(0,\epsilon_0)$, then there exists a good minimal model of $(\mathcal{W},\Delta_{\mathcal{W}}+\epsilon' h_*(\pi\times\mathrm{id}_{\mathbb{P}^1}\circ\rho_{\mathcal{Y}})^*H_{\mathbb{P}^1})$ over $\mathcal{Z}$.
Thus, there exists a birational contraction $\xi\colon\mathcal{W}\dashrightarrow\mathcal{X}$ such that $\mathcal{X}$ is a lc model of $(\mathcal{W},\Delta_{\mathcal{W}}+\epsilon' h_*(\pi\times\mathrm{id}_{\mathbb{P}^1}\circ\rho_{\mathcal{Y}})^*H_{\mathbb{P}^1})$ over $\mathcal{Z}$. 
Hence, $\xi$ is $\mathbb{G}_m$-equivariant.
Set $\Delta_{\mathcal{X},\epsilon}$ as the strict transform of $\Delta_{\mathcal{W}}+\epsilon h_*(\pi\times\mathrm{id}_{\mathbb{P}^1}\circ\rho_{\mathcal{Y}})^*H_{\mathbb{P}^1}$ on $\mathcal{X}$ for any $0\le\epsilon\le\epsilon'$.
Let $q_{\mathcal{X}}\colon\mathcal{X}\to\mathcal{Z}$ and $q_{\mathcal{W}}\colon\mathcal{W}\to\mathcal{Z}$ be the canonical morphisms and $G$ be a relatively ample $\mathbb{Q}$-divisor on $\mathcal{Z}$ over $\mathscr{C}$ such that $q_{\mathcal{W}}^*G\sim_{\mathbb{Q},\mathscr{C}}K_{\mathcal{W}}+\Delta_{\mathcal{W}}$.
Then, we see that $$q_{\mathcal{X}}^*G\sim_{\mathbb{Q},\mathscr{C}}K_{\mathcal{X}}+\xi_*\Delta_{\mathcal{W}}.$$
For any rational number $0<t<1$, we have
\[
K_{\mathcal{X}}+\Delta_{\mathcal{X},t\epsilon'}\sim_{\mathbb{Q},\mathscr{C}}(1-t)q_{\mathcal{X}}^*\mathcal{G}+t(K_{\mathcal{X}}+\Delta_{\mathcal{X},\epsilon'}).
\]
Since $K_{\mathcal{X}}+\Delta_{\mathcal{X},\epsilon'}$ is $q_{\mathcal{X}}$-ample and $G$ is relatively ample over $\mathscr{C}$, there exists $0<t'<1$ such that $K_{\mathcal{X}}+\Delta_{\mathcal{X},t\epsilon'}$ is relatively ample over $\mathscr{C}$ for any $t\in(0,t')\cap\mathbb{Q}$.
Thus, $\mathcal{X}$ is the lc model of $(\mathcal{W},\Delta_{\mathcal{W}}+t\epsilon' h_*(\pi\times\mathrm{id}_{\mathbb{P}^1}\circ\rho_{\mathcal{Y}})^*H_{\mathbb{P}^1})$ over $\mathcal{Z}$ for any $t\in(0,t')\cap\mathbb{Q}$ since $\xi$ is $K_{\mathcal{W}}+\Delta_{\mathcal{W}}+\epsilon' h_*(\pi\times\mathrm{id}_{\mathbb{P}^1}\circ\rho_{\mathcal{Y}})^*H_{\mathbb{P}^1}$-non-positive and $K_{\mathcal{W}}+\Delta_{\mathcal{W}}$-non-positive.
By letting $\epsilon_0:=t'\epsilon'$, we obtain 
$\xi\colon\mathcal{W}\dashrightarrow\mathcal{X}$ and $\epsilon_0$ as we asserted in the first sentence in this paragraph.

By what we have shown in the previous paragraph, we see that $\mathcal{X}$ and $\epsilon_0$ as above satisfy the second assertion of Proposition \ref{model}.
Indeed, since $h$ is $K_{\mathcal{Y}}+(\rho_{\mathcal{Y}})^{-1}_*\Theta_{\mathbb{P}^1}+\epsilon(\pi\times\mathrm{id}_{\mathbb{P}^1}\circ\rho_{\mathcal{Y}})^*H_{\mathbb{P}^1}$-negative, $\eta\colon\mathcal{Y}\dashrightarrow\mathcal{X}$ defines the lc model of $(\mathcal{Y},(\rho_{\mathcal{Y}})^{-1}_*\Theta_{\mathbb{P}^1}+\epsilon(\pi\times\mathrm{id}_{\mathbb{P}^1}\circ\rho_{\mathcal{Y}})^*H_{\mathbb{P}^1})$ for any rational number $\epsilon\in(0,\epsilon_0)$.
Thus, we obtain the second assertion of Proposition \ref{model}.

Finally, we deal with the last assertion of Proposition \ref{model}.
Note that $(X,\Delta+\epsilon H)$ is the lc model of $(Y,\Theta+\epsilon\pi^*H)$ over $C$ for any sufficiently small rational number $\epsilon>0$.
Fix such an $\epsilon$.
By the fact that $(\mathcal{X},\Delta_{\mathcal{X},\epsilon})$ is the lc model of $(\mathcal{Y},(\rho_{\mathcal{Y}})^{-1}_*\Theta_{\mathbb{P}^1}+\epsilon(\pi\circ\rho_{\mathcal{Y}})^*H_{\mathbb{P}^1})$ over $\mathscr{C}$ for any sufficiently small $\epsilon>0$, $\mathcal{X}$ is a test configuration for $X$. 
Since $\mathscr{C}$ is normal and projective over $\mathbb{P}^1$, there exists an ample $\mathbb{G}_m$-linearized $\mathbb{Q}$-line bundle $\mathscr{M}'$ on $\mathscr{C}$ (\cite[Corollary 1.6]{GIT}).
By adding $\rho^*_{\mathscr{C}}L'_{\mathbb{P}^1}$ to $\mathscr{M}'$ for some $\mathbb{Q}$-Cartier $\mathbb{Q}$-divisor $L'$ on $C$, we may assume that there exists
$r'\in\mathbb{Q}_{>0}$ such that $(\mathscr{C},\epsilon^{-1}\rho^*_{\mathscr{C}}D_{\mathbb{P}^1}+\mathscr{M})$ is an ample test configuration for $(C,r'L)$, where $D$ is a $\mathbb{Q}$-Cartier $\mathbb{Q}$-divisor on $C$ such that $f^*D\sim_{\mathbb{Q}}K_X+\Delta$.
Similarly, we obtain an ample $\mathbb{Q}$-line bundle $\mathscr{M}''$ on $\mathscr{C}$ and $r''\in\mathbb{Q}_{>0}$ such that $(\mathscr{C},\mathscr{M}'')$ is an ample test configuration for $(C,r''L)$.
Thus, $(\mathcal{X},\epsilon^{-1}(K_{\mathcal{X}}+\Delta_{\mathcal{X},\epsilon})+\Pi_{\mathcal{X}}^*(\mathscr{M}'+l\mathscr{M}''))$ is an ample test configuration for $(X,H+(r'+r''l)L)$, where $\Pi_{\mathcal{X}}\colon\mathcal{X}\to\mathscr{C}$ is the canonical morphism, for any sufficiently large $l\in\mathbb{Q}_{>0}$.
We complete the proof by letting $r:=r'+lr''$ and $\mathscr{M}:=\mathscr{M}'+l\mathscr{M}''$ for some sufficiently large $l$.
\end{proof}

Next, we explain how to construct a test configuration for $X$ whose DF invariant is very close to the log-twisted DF invariant of the test configuration for the base by applying Proposition \ref{model}.

\begin{prop}\label{cp}
Let $f:(X,\Delta,H)\to (C,L)$ be a polarized lc-trivial fibration such that $H$ is ample on $X$.
Let $B$ be the discriminant divisor and $M$ be the moduli divisor on $C$ respectively.
 Let $(\mathscr{C},\mathcal{L})$ be a semiample test configuration for $(C,L)$.
  Then, there exist $d\in\mathbb{Z}_{>0}$ and a semiample test configuration $(\mathscr{C}',\mathcal{L}')$ such that there exists a canonical morphism $r_{\mathscr{C}'}\colon\mathscr{C}'\to\mathscr{C}^{(d)}$ with $\mathcal{L}'\sim_{\mathbb{Q}}r_{\mathscr{C}'}^*\mathcal{L}^{(d)}$
  satisfying the following (cf.~Notation \ref{rramamra}).
\begin{enumerate}[(i)]
\item\label{cp-i} There exists a canonical morphism $\rho_{\mathscr{C}'}:\mathscr{C}'\to C_{\mathbb{P}^1}$.
\item\label{cp-ii} $\mathscr{C}'_0$ is reduced
.
\item\label{cp-iii}  There exist a positive integer $r>0$ and an ample test configuration $(\mathcal{X},\mathcal{H})$ for $(X,H+rL)$ with a canonical morphism $\Pi:\mathcal{X}\to\mathscr{C}'$.
\item\label{cp-iv} There exists an open subset $U\subset \mathscr{C}'$ such that $\mathrm{codim}_{\mathscr{C}'}(\mathscr{C}'\setminus U)\ge2$ and 
\[
(K_{\mathcal{X}}+(\rho_{\mathcal{X}})_*^{-1}\Delta_{\mathbb{P}^1})|_{\Pi^{-1}(U)}\sim_{\mathbb{Q}}\Pi|_{\Pi^{-1}(U)}^*((K_{\mathscr{C}'}+\mathcal{M}+ (\rho_{\mathscr{C}'})_*^{-1}B_{\mathbb{P}^1})|_U),
\]
where $\rho_{\mathcal{X}}:\mathcal{X}\dashrightarrow X_{\mathbb{P}^1}$ is the canonical birational map and $\mathcal{M}$ is the moduli divisor on $\mathscr{C}'$ defined by $f\times\mathrm{id}_{\mathbb{P}^1}$ and $M_{\mathbb{P}^1}$ (see the paragraph before Definition \ref{de-ambro}).
\end{enumerate}
\end{prop}

\begin{proof}
We first note that to show Proposition \ref{cp}, we may freely replace $(\mathscr{C},\mathcal{L})$ with some blow up or the normalized base change by the $d$-th power morphism of $\mathbb{P}^1$ for some $d\in\mathbb{Z}_{>0}$.
Take a normal test configuration $\mathscr{C}'$ for $C$ with canonical morphisms $r_{\mathscr{C}'}\colon\mathscr{C}'\to\mathscr{C}$ and $\rho_{\mathscr{C}'}\colon\mathscr{C}\to C_{\mathbb{P}^1}$.
Let $\mathcal{L}':=r_{\mathscr{C}'}^*\mathcal{L}$.
Let $\pi:Y\to X$ be the log resolution and set $\Theta:=\pi^{-1}_*\Delta+\mathrm{Ex}(\pi)$ as a snc divisor on $Y$.
 Then $(Y,\Theta)$ is subdlt and log smooth. 
 Let $\mathcal{Y}$ be a $\mathbb{G}_m$-equivariant blow up of $Y_{\mathbb{P}^1}$ that is a smooth and projective test configuration for $Y$ such that the canonical birational map $\Pi_{\mathcal{Y}}\colon\mathcal{Y}\dashrightarrow\mathscr{C}'$ is a morphism (cf.~\cite{Ko2}, \cite{LX}). 
 Let $g\colon\mathcal{Y} \to X_{\mathbb{P}^1}$ and $\rho_{\mathcal{Y}}\colon\mathcal{Y}\to Y_{\mathbb{P}^1}$ be the canonical birational morphisms. 
 Then we can take a $\mathbb{G}_m$-equivariant semistable reduction $\mathcal{Y}'$ of $\mathcal{Y}$ (cf.~\cite{KKMS}, \cite[Theorem 7.17]{KoMo}, \cite[Lemma 5]{LX}) such that $\mathcal{Y}'\to\mathcal{Y}^{(d)}$ is birational for some $d\in\mathbb{Z}_{>0}$ and $\mathcal{Y}'_0+\rho_{\mathcal{Y}'}^*\Theta_{\mathbb{P}^1}$ is a reduced snc divisor, where $\rho_{\mathcal{Y}'}\colon\mathcal{Y}'\to Y_{\mathbb{P}^1}$ is the canonical birational morphism. 
 We may also assume that $\mathscr{C}'^{(d)}_0$ is also reduced.
 Thus, we may replace $\mathcal{Y}$ by $\mathcal{Y}'$ and $\mathscr{C}'$ by $\mathscr{C}'^{(d)}$ and hence we may assume that $\mathcal{Y}_0+\rho_{\mathcal{Y}}^*\Theta_{\mathbb{P}^1}$ is a reduced snc divisor on $\mathcal{Y}$.
 By \cite[III Corollary 10.7]{Ha}, $(\mathcal{Y},\mathcal{Y}_0+\rho_{\mathcal{Y}}^*\Theta_{\mathbb{P}^1})$ is relatively log smooth over some neighborhoods of the generic points of all irreducible components of $\mathscr{C}'_0$. 
 Then, $\mathscr{C}'$ satisfies the conditions (\ref{cp-i}) and (\ref{cp-ii}) now. 

Next, we deal with the rest conditions.
We note that
 $$\Pi_{\mathcal{Y}}:\left(\mathcal{Y},g^{-1}_*\Delta_{\mathbb{P}^1}+\sum_i(1-A_{(X_{\mathbb{P}^1},\Delta_{\mathbb{P}^1})}(E_i))E_i\right)\to \mathscr{C}'$$
  is a sublc-trivial fibration, where $E_i$ runs over all $g$-exceptional prime divisors.
  Let $\Xi:=g^{-1}_*\Delta_{\mathbb{P}^1}+\sum_i(1-A_{(X_{\mathbb{P}^1},\Delta_{\mathbb{P}^1})}(E_i))E_i$.
 Then we have the canonical bundle formula for $\Pi_{\mathcal{Y}}$
\begin{align}
\label{eq--canbdleform-Pi}K_{\mathcal{Y}}+\Xi\sim_{\mathbb{Q}}\Pi_{\mathcal{Y}}^*(K_{\mathscr{C}}+\mathcal{B}+\mathcal{M}),
\end{align}       
where $\mathcal{B}$ (resp.~$\mathcal{M}$) is the discriminant (resp.~moduli) divisor. 
We note that 
\begin{align*}
K_{\mathcal{Y}}+\Xi&\sim_{\mathbb{Q}}g^*(K_{X_{\mathbb{P}^1}}+\Delta_{\mathbb{P}^1})=\Pi_{\mathcal{Y}}^*\rho_{\mathscr{C}'}^*(K_{C_{\mathbb{P}^1}}+M_{\mathbb{P}^1}+ B_{\mathbb{P}^1}),\quad\text{and hence}\\
K_{\mathscr{C}}+\mathcal{M}+ \mathcal{B}&=\rho_{\mathscr{C}'}^*(K_{C_{\mathbb{P}^1}}+M_{\mathbb{P}^1}+ B_{\mathbb{P}^1}).
\end{align*}    
This means that $\mathcal{M}$ is the moduli divisor defined by $f\times\mathrm{id}_{\mathbb{P}^1}$ and $M_{\mathbb{P}^1}$.
On the other hand, we obtain by (\ref{eq--canbdleform-Pi}) that
\[
K_{\mathcal{Y}}+(\rho_{\mathcal{Y}})^{-1}_*\Theta_{\mathbb{P}^1}\sim_{\mathbb{Q}}\Pi_{\mathcal{Y}}^*(K_{\mathscr{C}}+\mathcal{M}+ \mathcal{B})-\sum_i(1-A_{(X_{\mathbb{P}^1},\Delta_{\mathbb{P}^1})}(Q_i))Q_i+\sum_jA_{(X_{\mathbb{P}^1},\Delta_{\mathbb{P}^1})}(F_j)F_j,
\]
where $Q_i$ and $F_j$ run over all $g$-exceptional prime divisors such that $Q_i\subset\mathcal{Y}_0$ and $F_j\not\subset \mathcal{Y}_0$ respectively. 
Since $(\mathcal{Y},\mathcal{Y}_0+(\rho_{\mathcal{Y}})^{-1}_*\Theta_{\mathbb{P}^1})$ is log smooth over some neighborhood of the generic point $\eta_P$ of every irreducible component $P$ of $\mathscr{C}_0$, for any sufficiently small neighborhood $\mathcal{Y}_P\subset \mathcal{Y}$ of $\Pi_{\mathcal{Y}}^{-1}(\eta_P)$,
\[
\mathrm{ord}_P(\mathcal{B})=1-\mathrm{lct}_{\left(\mathcal{Y}_P,\Xi|_{\mathcal{Y}_P}\right)}(\Pi_{\mathcal{Y}}^*P)=\max_G(1-A_{(X_{\mathbb{P}^1},\Delta_{\mathbb{P}^1})}(G)),
\]
where $G$ runs over all prime divisors on $\mathcal{Y}$ whose images are $P$. 
Let $\mathscr{C}'_{\mathrm{reg}}\subset \mathscr{C}'$ and set a $\mathbb{Q}$-divisor on $\Pi_{\mathcal{Y}}^{-1}(\mathscr{C}'_{\mathrm{reg}})$ as
 \begin{align*}
     D&:=\Pi_{\mathcal{Y}}|_{\Pi_{\mathcal{Y}}^{-1}(\mathscr{C}'_{\mathrm{reg}})}^*(\mathcal{B}-(\rho_{\mathscr{C}})_*^{-1}B_{\mathbb{P}^1})|_{\mathscr{C}'_{\mathrm{reg}}}-\sum_i(1-A_{(X_{\mathbb{P}^1},\Delta_{\mathbb{P}^1})}(Q_i))Q_i|_{\Pi_{\mathcal{Y}}^{-1}(\mathscr{C}'_{\mathrm{reg}})}\\
     &+\sum_jA_{(X_{\mathbb{P}^1},\Delta_{\mathbb{P}^1})}(F_j)F_j|_{\Pi_{\mathcal{Y}}^{-1}(\mathscr{C}'_{\mathrm{reg}})},
 \end{align*}
 where $Q_i$ and $F_j$ run over all $g$-exceptional prime divisors such that $Q_i\subset\mathcal{Y}_0$ and $F_j\not\subset \mathcal{Y}_0$ respectively. 
Then, 
\begin{equation}
(K_{\mathcal{Y}}+(\rho_{\mathcal{Y}})^{-1}_*\Theta_{\mathbb{P}^1})|_{\Pi_{\mathcal{Y}}^{-1}(\mathscr{C}'_{\mathrm{reg}})}\sim_{\mathbb{Q}}\Pi_{\mathcal{Y}}|_{\Pi_{\mathcal{Y}}^{-1}(\mathscr{C}'_{\mathrm{reg}})}^*((K_{\mathscr{C}'}+\mathcal{M}+ (\rho_{\mathscr{C}'})_*^{-1}B_{\mathbb{P}^1})|_{\mathscr{C}'_{\mathrm{reg}}})+D,\label{eq--def--D}
\end{equation}
where $\mathrm{codim}_{\mathscr{C}'_{\text{reg}}}\Pi_{\mathcal{Y}}(\mathrm{Supp}\,D_-)\ge2$ and $D_+$ is of insufficient fiber type in the sense of \cite[Definition 2.4]{Ta}, i.e.~if $R$ is an irreducible component of $\Pi_{\mathcal{Y}}(\mathrm{Supp}\,D_+)$ of codimension one in $\mathscr{C}'_{\text{reg}}$ with the generic point $\eta_R$, then $\Pi_{\mathcal{Y}}^{-1}(\eta_R)$ is not contained in the support of $D_+$.
Indeed, for any irreducible component $P$ of $\mathscr{C}'_0$, if $Q$ is a prime divisor on $\mathcal{Y}$ dominating $P$, then we have by (\ref{eq--def--D}) that
\[
\mathrm{ord}_Q(D)=A_{(X_{\mathbb{P}^1},\Delta_{\mathbb{P}^1})}(Q)-\min_GA_{(X_{\mathbb{P}^1},\Delta_{\mathbb{P}^1})}(G)\ge0,
\]
 where $G$ runs over all prime divisors dominating $P$.
 Set $U\subset \mathscr{C}'_{\text{reg}}$ as an open subset such that $\mathrm{codim}_{\mathscr{C}'}(\mathscr{C}'\setminus U)\ge2$ and $\Pi_{\mathcal{Y}}^{-1}(U)\cap D=\Pi_{\mathcal{Y}}^{-1}(U)\cap D_+$.

Note that $(\mathcal{Y},(\rho_{\mathcal{Y}})^{-1}_*\Theta_{\mathbb{P}^1}+\mathcal{Y}_{0})$ is log smooth and dlt. 
By Proposition \ref{model}, there exists a $\mathbb{G}_m$-equivariant birational contraction $\mathcal{Y}\dashrightarrow\mathcal{X}$ such that $(\mathcal{X},\Delta_{\mathcal{X},\epsilon})$ is the lc model of $(\mathcal{Y},(\rho_{\mathcal{Y}})^{-1}_*\Theta_{\mathbb{P}^1}+\epsilon(\pi_{\mathbb{P}^1}\circ\rho_{\mathcal{Y}})^*H_{\mathbb{P}^1})$ for any sufficiently small $\epsilon\in(0, 1)\cap\mathbb{Q}$ over $\mathscr{C}'$. 
Moreover, there is also a $\mathbb{Q}$-factorial good minimal model $(\mathcal{W},\Delta_\mathcal{W})$ of $(\mathcal{Y},(\rho_{\mathcal{Y}})^{-1}_*\Theta_{\mathbb{P}^1})$ and a birational contraction $\xi: \mathcal{W}\dashrightarrow\mathcal{X}$ by Proposition \ref{model}. 
Then, 
\[
(K_{\mathcal{W}}+\Delta_{\mathcal{W}})|_{\Pi_{\mathcal{W}}^{-1}(U)}\sim_{\mathbb{Q}}\Pi_{\mathcal{W}}|_{\Pi_{\mathcal{W}}^{-1}(U)}^*((K_{\mathscr{C}}+\mathcal{M}+ (\rho_{\mathscr{C}'})^{-1}_*B_{\mathbb{P}^1})|_{U})+D'
\]
where $\Pi_{\mathcal{W}}:\mathcal{W}\to \mathscr{C}'$ is the structure morphism and $D'$ is the strict transform of $D|_{\Pi_{\mathcal{Y}}^{-1}(U)}$. 
Now, we claim that $\mathrm{codim}_{U}\Pi_{\mathcal{W}}(\mathrm{Supp}\,D')\ge 2$. 
Indeed, for any irreducible component $P$ of $\mathscr{C}'_0$, take a prime divisor $Q$ on $\mathcal{Y}$ with the generic point $\eta_Q$ dominating $P$ such that $Q\not\subset \mathrm{Supp}\, D_+$. 
It is easy to see that $\eta_Q\not\in\mathbf{B}((K_{\mathcal{W}}+\Delta_{\mathcal{W}})|_{\Pi_{\mathcal{W}}^{-1}(U_P)}/U_P)$ for any  sufficiently small open neighborhood $U_P\subset U$ of the generic point of $P$. 
Thus, $Q$ is not contracted by $\mathcal{Y}\dashrightarrow\mathcal{W}$ by \cite[2.4]{HX} and $(D'_+)|_{\Pi_{\mathcal{W}}^{-1}(U_P)}$ is also of insufficient fiber type over $U_P$. 
Therefore, $$(D'_+)|_{\Pi_{\mathcal{W}}^{-1}(U_P)}=D'|_{\Pi_{\mathcal{W}}^{-1}(U_P)}=0$$ by \cite[2.10]{Lai} and \cite[2.4 (1)]{HX}. 
This implies that $D'=0$.
Take an ample $\mathbb{Q}$-line bundle $\mathcal{H}$ on $\mathcal{X}$ as Proposition \ref{model}. 
Then, it is easy to check that $\Pi\colon\mathcal{X}\to\mathscr{C}'$ and $U$ 
satisfy the conditions (\ref{cp-iii}) and (\ref{cp-iv}).
We complete the proof.
\end{proof}
 
\begin{proof}[Proof of Theorem \ref{can}]
Suppose that $\mathrm{dim}\,X=n$ and $\mathrm{dim}\,C=c$.
For any general point $t\in C$, $H_t$ denotes the restriction of $H$ to the fiber of $f$ over $t$.
Let 
$(\mathscr{C},\mathcal{L})$ be a semiample test configuration for $(C,L)$ such that 
\[
M^\mathrm{NA}_{(B,M)}(\mathscr{C},\mathcal{L})<0.
\]
 Then, we may replace $H$ by $H+rL$, where $r$ is as in Proposition \ref{cp} and take $(\mathcal{X},\mathcal{H})$ and a canonical birational morphism $r_{\mathscr{C}'}\colon\mathscr{C}'\to\mathscr{C}^{(d)}$ as in Proposition \ref{cp}. 
 For any sufficiently small $\epsilon\in\mathbb{Q}_{>0}$,
\[
\frac{-(K_X+\Delta)\cdot (\epsilon H+L)^{n-1}}{ (\epsilon H+L)^n}=-\frac{c(K_C+M+B)\cdot L^{c-1}}{nL^{c}}+O(\epsilon).
\]
 Then
\begin{align*}
&(\epsilon H+L)^n\mathrm{DF}_{\Delta}(\mathcal{X},\epsilon\mathcal{H}+\Pi^*\mathcal{L}')\\
&=(K_{\mathcal{X}/\mathbb{P}^1}+(\rho_{\mathcal{X}})_*^{-1}\Delta_{\mathbb{P}^1})\cdot (\epsilon\mathcal{H}+\Pi^*\mathcal{L}')^n-\frac{n(K_X+\Delta)\cdot (\epsilon H+L)^{n-1}}{(n+1) (\epsilon H+L)^n}(\epsilon\mathcal{H}+\Pi^*\mathcal{L}')^{n+1}\\
&=(K_{\mathcal{X}/\mathbb{P}^1}+(\rho_{\mathcal{X}})_*^{-1}\Delta_{\mathbb{P}^1})\cdot (\epsilon\mathcal{H}+\Pi^*\mathcal{L}')^n\\
&-\frac{n}{n+1}\left(\frac{c(K_C+M+B)\cdot L^{c-1}}{nL^{c}}+O(\epsilon)\right)(\epsilon\mathcal{H}+\Pi^*\mathcal{L}')^{n+1}.
\end{align*}
By the condition (\ref{cp-iv}) of Proposition \ref{cp}, we obtain that 
\[
(K_{\mathcal{X}/\mathbb{P}^1}+(\rho_{\mathcal{X}})_*^{-1}\Delta_{\mathbb{P}^1})\cdot (\epsilon\mathcal{H}+\Pi^*\mathcal{L}')^n=\epsilon^{n-c}(H_t)^{n-c}\binom{n}{c}(K_{\mathscr{C}'/\mathbb{P}^1}+\mathcal{M}+ (\rho_{\mathscr{C}'})_*^{-1}B_{\mathbb{P}^1})\cdot \mathcal{L}'^{c}+O(\epsilon^{n-c+1}),
\]
 where $\mathcal{M}$ is the moduli divisor induced by $f\times\mathrm{id}_{\mathbb{P}^1}$.
 Thus, we conclude that the leading term of $(\epsilon H+L)^n\mathrm{DF}_{\Delta}(\mathcal{X},\epsilon\mathcal{H}+\Pi^*\mathcal{L}')$  coincides with
\begin{align*}
    &\epsilon^{n-c}(H_t)^{n-c}\binom{n}{c}(K_{\mathscr{C}'/\mathbb{P}^1}+\mathcal{M}+ (\rho_{\mathscr{C}'})_*^{-1}B_{\mathbb{P}^1})\cdot \mathcal{L}'^{c}-\binom{n+1}{c+1}\frac{c(K_C+M+B)\cdot L^{c-1}}{nL^{c}}\mathcal{H}^{n-c}\cdot \mathcal{L}'^{c+1}\\
    &=\epsilon^{n-c}\binom{n}{c}(H_t^{n-c})\left(K_{\mathscr{C}'/\mathbb{P}^1}+\mathcal{M}+ (\rho_{\mathscr{C}'})_*^{-1}B_{\mathbb{P}^1})\cdot \mathcal{L}'^{c}-\frac{c(K_C+M+B)\cdot L^{c-1}}{(c+1)L^{c}}\mathcal{L}'^{c+1}\right)
    .
\end{align*}
$\rho_{\mathscr{C}'}^*M_{\mathbb{P}^1}-\mathcal{M}$ is effective by \cite[Lemma 3.39]{KoMo} and \cite[Theorem 1.1]{FG}.
Thus, the $\epsilon^{n-c}$-term of $(\epsilon H+L)^n\mathrm{DF}_{\Delta}(\mathcal{X},\epsilon\mathcal{H}+\Pi^*\mathcal{L}')$ is at most
$
\binom{n}{c}(H^{n-c}\cdot L^c)\mathrm{DF}_{(B,M)}(\mathscr{C}',\mathcal{L}')
$.
Here, we note that
$$\mathrm{DF}_{(B,M)}(\mathscr{C}',\mathcal{L}')=M^\mathrm{NA}_{(B,M)}(\mathscr{C}',\mathcal{L}')=dM^\mathrm{NA}_{(B,M)}(\mathscr{C},\mathcal{L})<0$$ since $\mathscr{C}'_0$ is reduced. Thus, $\mathrm{DF}_{\Delta}(\mathcal{X},\epsilon\mathcal{H}+\Pi^*\mathcal{L}')<0$ for any sufficiently small $\epsilon>0$.
\end{proof}


 
If $C$ is a curve, $C$ is necessarily an Ambro model for any sublc-trivial fibration $f:(X,\Delta)\to C$. Thus, we obtain the following by Theorem \ref{can}:
\begin{cor}
Let $f:(X,\Delta,H)\to (C,L)$ be a polarized lc-trivial fibration such that $C$ is a smooth curve.
Let
$M$ (resp.~$B$) be the moduli (resp.~discriminant) divisor. 
If $(X,\Delta,H)$ is adiabatically K-semistable over $(C,L)$, then $(C,B,M,L)$ is log-twisted K-semistable.
\end{cor}


Next, we state a partial result that uniform adiabatic K-stability of the total space implies K-stability of the log-twisted Fano base of a klt-trivial fibration.
To see this, we prepare the following two lemmas.

\begin{lem}\label{lem-unad}
    Let $f\colon(X,\Delta,H)\to (C,L)$ be a polarized algebraic fiber space subpair and $T$ be a $\mathbb{Q}$-line bundle on $C$.
    Then there exist $\epsilon_0>0$ and $G>0$ such that $(X,L)$ is $\mathrm{J}^{G(\epsilon H+L)+f^*T}$-semistable for any $\epsilon\in\mathbb{Q}\cap(0,\epsilon_0)$. 
\end{lem}

\begin{proof}
    Suppose that $X$ (resp.~$C$) is of dimension $n$ (resp.~$c$).
    By Remark \ref{rem--useful}, J$^{G(\epsilon H+L)+f^*T}$-semistability holds if $G(\epsilon H+L)+f^*T$ and
    \[
    \left(n\frac{f^*T\cdot (f^*L+\epsilon H)^{n-1}}{(f^*L+\epsilon H)^n}+G\right)(f^*L+\epsilon H)-(n-1)f^*T 
    \]
    are ample.
    For sufficiently small $\epsilon>0$, we have that $G(\epsilon H+L)+f^*T=G\epsilon H+f^*(GL+T)$ and
    \begin{align*}
&\left(n\frac{f^*T\cdot (f^*L+\epsilon H)^{n-1}}{(f^*L+\epsilon H)^n}+G\right)(f^*L+\epsilon H)-(n-1)f^*T\\
&=f^*\left(\left(G+c\frac{f^*T\cdot f^*L^{c-1}\cdot H^{n-c}}{f^*L^{c}\cdot H^{n-c}}+O(\epsilon) \right)L-(n-1)T\right)\\
&+\epsilon\left(G+c\frac{f^*T\cdot f^*L^{c-1}\cdot H^{n-c}}{f^*L^{c}\cdot H^{n-c}}+O(\epsilon) \right)H.
\end{align*}
Thus, there exist $\epsilon_0>0$ and $G>0$ such that for any $\epsilon\in\mathbb{Q}\cap(0,\epsilon_0)$, the above two $\mathbb{Q}$-line bundles are ample. 
Then, J$^{G(\epsilon H+L)+f^*T}$-semistability holds for such $G$ and any $\epsilon\in\mathbb{Q}\cap(0,\epsilon_0)$.
\end{proof}
 
\begin{lem}\label{unad}
Let $f\colon(X,\Delta,H)\to (C,L)$ be a polarized lc-trivial fibration and suppose that $(X,\Delta,H)$ is uniformly adiabatically K-stable over $(C,L)$. 
Then there exist three positive constants $\delta_1$, $\delta_2$ and $\epsilon_0$ such that for any rational number $\epsilon\in(0,\epsilon_0)$ and semiample test configuration $(\mathcal{X},\mathcal{M})$
 for $(X,\epsilon H+L)$,
\[
M_{\Delta}^{\mathrm{NA}}(\mathcal{X},\mathcal{M})\ge \delta_1H_{\Delta}^{\mathrm{NA}}(\mathcal{X},\mathcal{M}) +\delta_2 (\mathcal{J}^{\epsilon H+L})^{\mathrm{NA}}(\mathcal{X},\mathcal{M}).
\]
\end{lem}
 
\begin{proof}
Decompose $M_{\Delta}^{\mathrm{NA}}(\mathcal{X},\mathcal{M})=H_{\Delta}^{\mathrm{NA}}(\mathcal{X},\mathcal{M})+(\mathcal{J}^{K_X+\Delta})^{\mathrm{NA}}(\mathcal{X},\mathcal{M})$. By assumption, there exist positive constants $\epsilon'_0$ and $\delta$ such that
$$H_{\Delta}^{\mathrm{NA}}(\mathcal{X},\mathcal{M})\ge-(\mathcal{J}^{K_X+\Delta})^{\mathrm{NA}}(\mathcal{X},\mathcal{M})+\delta (\mathcal{J}^{\epsilon H+L})^{\mathrm{NA}}(\mathcal{X},\mathcal{M})$$
 for any $\epsilon\in\mathbb{Q}\cap(0, \epsilon'_0)$ and semiample test configuration $(\mathcal{X},\mathcal{M})$ for $(X,\epsilon H+f^*L)$. 
 Thus, for any $0<\delta_1<1$,
\begin{align}
(1-\delta_1)H_{\Delta}^{\mathrm{NA}}(\mathcal{X},\mathcal{M})&\ge-(\mathcal{J}^{K_X+\Delta})^{\mathrm{NA}}(\mathcal{X},\mathcal{M})\label{eq--unad1}\\
&+\delta(1-\delta_1) (\mathcal{J}^{\epsilon H+L})^{\mathrm{NA}}(\mathcal{X},\mathcal{M})+\delta_1(\mathcal{J}^{K_X+\Delta})^{\mathrm{NA}}(\mathcal{X},\mathcal{M}).\nonumber
\end{align}
On the other hand, recall that there exists a $\mathbb{Q}$-line bundle $T$ on $C$ such that ${K_X+\Delta}\sim_{\mathbb{Q}}f^*T$. 
By Lemma \ref{lem-unad}, there exist $0<\epsilon_0<\epsilon_0'$ and $G>0$ such that for any $\epsilon\in\mathbb{Q}\cap(0, \epsilon_0)$,
\[
G(I^{\mathrm{NA}}(\mathcal{X},\mathcal{M})-J^{\mathrm{NA}}(\mathcal{X},\mathcal{M}))+(\mathcal{J}^{K_X+\Delta})^{\mathrm{NA}}(\mathcal{X},\mathcal{M})\ge0.
\]
Therefore, there exist positive constants $\delta_1'$ and $\delta_2$ such that for any $0<\epsilon<\epsilon_0$ and $0<\delta_1<\delta_1'$,
\begin{equation}
\delta(1-\delta_1) (\mathcal{J}^{\epsilon H+L})^{\mathrm{NA}}(\mathcal{X},\mathcal{M})+\delta_1(\mathcal{J}^{K_X+\Delta})^{\mathrm{NA}}(\mathcal{X},\mathcal{M})\ge \delta_2(\mathcal{J}^{\epsilon H+L})^{\mathrm{NA}}(\mathcal{X},\mathcal{M}).\label{eq--unad2}
\end{equation}
Fix $0<\delta_1<\min\{\delta_1',1\}$. 
Then we have
\begin{align*}
M_{\Delta}^{\mathrm{NA}}(\mathcal{X},\mathcal{M})&=(1-\delta_1)H_{\Delta}^{\mathrm{NA}}(\mathcal{X},\mathcal{M})+(\mathcal{J}^{K_X+\Delta})^{\mathrm{NA}}(\mathcal{X},\mathcal{M})+\delta_1H_{\Delta}^{\mathrm{NA}}(\mathcal{X},\mathcal{M})\\
&\ge\delta(1-\delta_1) (\mathcal{J}^{\epsilon H+L})^{\mathrm{NA}}(\mathcal{X},\mathcal{M})+\delta_1(\mathcal{J}^{K_X+\Delta})^{\mathrm{NA}}(\mathcal{X},\mathcal{M})+\delta_1H_{\Delta}^{\mathrm{NA}}(\mathcal{X},\mathcal{M})\\
&\ge\delta_1H_{\Delta}^{\mathrm{NA}}(\mathcal{X},\mathcal{M}) +\delta_2 (\mathcal{J}^{\epsilon H+L})^{\mathrm{NA}}(\mathcal{X},\mathcal{M})
\end{align*}
by the equations (\ref{eq--unad1}) and (\ref{eq--unad2}).
We complete the proof.
\end{proof}
 
\begin{thm}\label{stp1}
Let $f\colon(X,\Delta,H)\to (C,L)$ be a polarized klt-trivial fibration, $M$ (resp.~$B$) be the moduli (resp.~discriminant) divisor on $C$. Suppose that $C$ is an Ambro model, $(X,\Delta,H)$ is uniformly adiabatically K-stable over $(C,L)$ and $L=-(K_C+M+B)$ is ample. 

Then $(C,B,M,L)$ is log-twisted K-stable.
\end{thm}
 
\begin{proof}
Since $(X,\Delta,H)$ is uniformly adiabatically K-stable over $(C,L)$, there exist positive constants $\epsilon_0>0$, $\delta_1>0$ and $\delta_2>0$ such that 
\begin{equation}
    M_{\Delta}^{\mathrm{NA}}(\mathcal{X},\mathcal{M})\ge \delta_1H_{\Delta}^{\mathrm{NA}}(\mathcal{X},\mathcal{M}) +\delta_2 (\mathcal{J}^{\epsilon H+L})^{\mathrm{NA}}(\mathcal{X},\mathcal{M})\label{eq-unad-in-thm}
\end{equation}
for any rational number $0<\epsilon<\epsilon_0$ and normal semiample test configuration $(\mathcal{X},\mathcal{M})$ for $(X,\epsilon H+L)$ by Lemma \ref{unad}. 
Fix an arbitrary $\mathbb{Q}$-Cartier $\mathbb{Q}$-divisor $D$ on $C$.
Replacing $\epsilon_0$ with a smaller constant if necessary, we have by Lemma \ref{lem-unad} that
\begin{equation}
(\mathcal{J}^{\epsilon H+L})^{\mathrm{NA}}(\mathcal{X},\mathcal{M})+(\mathcal{J}^{(\frac{a}{\delta_2}f^*D)})^{\mathrm{NA}}(\mathcal{X},\mathcal{M})>0\label{eq--stp1}
\end{equation}
for sufficiently small $a>0$ and $\epsilon\in\mathbb{Q}\cap(0,\epsilon_0)$
.
Thus, we conclude that there exists a positive constant $a_1>0$ such that for any $0<a<a_1$, $0<\epsilon<\epsilon_0$ and normal semiample test configuration $(\mathcal{X},\mathcal{M})$ for $(X,\epsilon H+L)$, $(X,\Delta+\frac{a}{\delta_1}f^*D)$ is klt and (\ref{eq--stp1}) holds.
Note that if $(X,\Delta+\frac{a}{\delta_1}f^*D)$ is klt, then 
\begin{equation}
H_{(\Delta+\frac{a}{\delta_1}f^*D)}^{\mathrm{NA}}(\mathcal{X},\mathcal{M})>0\label{eq--stp2}
\end{equation}
 by Theorem \ref{thm--alpha}. 
 We may assume that there exists a canonical morphism $\rho\colon\mathcal{X}\to X_{\mathbb{P}^1}$. 
 Put $n$ as the dimension of $X$.
Then, we see that
\begin{align*}
M_{\Delta+af^*D}^{\mathrm{NA}}(\mathcal{X},\mathcal{M})&-M_{\Delta}^{\mathrm{NA}}(\mathcal{X},\mathcal{M})\\
&=\frac{a}{(\epsilon H+L)^n}\left((\mathcal{D}-\rho^*(f^*D)_{\mathbb{P}^1})\cdot\mathcal{M}^n-\frac{nf^*D\cdot(\epsilon H+L )^{n-1}}{(n+1)(\epsilon H+L )^{n}}\mathcal{M}^{n+1}\right)\\
&=\delta_1(H_{(\Delta+\frac{a}{\delta_1}f^*D)}^{\mathrm{NA}}(\mathcal{X},\mathcal{M})-H_{\Delta}^{\mathrm{NA}}(\mathcal{X},\mathcal{M}))+(\mathcal{J}^{af^*D})^{\mathrm{NA}}(\mathcal{X},\mathcal{M}),
\end{align*}
where $\mathcal{D}$ is the strict transform of $(f^*D)_{\mathbb{P}^1}$ on $\mathcal{X}$
. 
Combining the above equation with (\ref{eq-unad-in-thm}), (\ref{eq--stp1}) and (\ref{eq--stp2}), we obtain that
\begin{align*}
M_{\Delta+af^*D}^{\mathrm{NA}}(\mathcal{X},\mathcal{M})&=\delta_1(H_{(\Delta+\frac{a}{\delta_1}f^*D)}^{\mathrm{NA}}(\mathcal{X},\mathcal{M})-H_{\Delta}^{\mathrm{NA}}(\mathcal{X},\mathcal{M}))+(\mathcal{J}^{af^*D})^{\mathrm{NA}}(\mathcal{X},\mathcal{M})+M_{\Delta}^{\mathrm{NA}}(\mathcal{X},\mathcal{M})\\
&\ge \delta_1H_{(\Delta+\frac{a}{\delta_1}f^*D)}^{\mathrm{NA}}(\mathcal{X},\mathcal{M})+\delta_2((\mathcal{J}^{\epsilon H+L})^{\mathrm{NA}}(\mathcal{X},\mathcal{M})+(\mathcal{J}^{(\frac{a}{\delta_2}f^*D)})^{\mathrm{NA}}(\mathcal{X},\mathcal{M}))>0.
\end{align*}
for any $a\in\mathbb{Q}\cap(0,a_1)$, $\epsilon\in\mathbb{Q}\cap(0,\epsilon_0)$ and normal semiample test configuration $(\mathcal{X},\mathcal{M})$ for $(X,\epsilon H+L)$. 
 By \cite[Proposition 7.15]{BHJ}, we note that $$\mathrm{DF}_{\Delta+af^*D}(\mathcal{X},\mathcal{M})\ge M_{\Delta+af^*D}^{\mathrm{NA}}(\mathcal{X},\mathcal{M})>0.$$

By Theorem \ref{can}, $(C,B,M,L)$ is at least log-twisted K-semistable.
Note that $(C,B)$ is klt by Theorem \ref{thm--odaka-type}.
Thus, it suffices to show that there exists no nontrivial normal semiample test configuration for $(C,L)$ whose log-twisted Donaldson-Futaki invariant is zero. 
To show this, assume the contrary. 
Note that $M$ is semiample by Lemma \ref{abund}. 
By Theorem \ref{appendix}, we may further assume that there exist an effective $\mathbb{Q}$-Cartier $\mathbb{Q}$-divisor $D\sim_{\mathbb{Q}}L$ on $C$ and a positive constant $a_0>0$ such that 
$(C,B+aD,M,L)$ is K-unstable for any $a\in\mathbb{Q}\cap(0,a_0)$. 
By Theorem \ref{can}, there exist a constant $0<\epsilon_{a}<\epsilon_0$  for any $a\in\mathbb{Q}\cap(0,a_0)$ such that for any $\epsilon\in\mathbb{Q}\cap(0,\epsilon_{a})$, there exists a normal semiample test configuration $(\mathcal{X},\mathcal{M})$ for $(X,\epsilon H+L)$ such that $\mathrm{DF}_{\Delta+af^*D}(\mathcal{X},\mathcal{M})<0$. 
However, by what we have shown in the first paragraph, $\mathrm{DF}_{\Delta+af^*D}(\mathcal{X},\mathcal{M})>0$ for any $a\in\mathbb{Q}\cap(0,\min\{a_0,a_1\})$, $\epsilon\in\mathbb{Q}\cap(0,\epsilon_a)$ and normal semiample test configuration $(\mathcal{X},\mathcal{M})$ for $(X,\epsilon H+L)$.
This is a contradiction.
 Thus, $(C,B,M,L)$ is K-stable and we complete the proof.
\end{proof}

 \subsection{Convergence of the $\delta$-invariant}\label{Weier}
In this subsection, we calculate $\lim_{\epsilon\to 0}\delta_{(X,\Delta)}(\epsilon H+L)$ when $f:(X,\Delta,H)\to (C,L)$ is a polarized algebraic fiber space and $C$ is a curve
. 
We make use of this to obtain a sufficiency condition of adiabatic K-stability of klt-trivial fibrations over curves
.

 First, we bound the multiplicity of $D_{\mathrm{vert}}$, which is the vertical part of $D\in|\epsilon H+L|_{\mathbb{Q}}$.
 
\begin{lem}\label{lem-claim}
Let $f\colon X\to C$ be a contraction to a proper smooth curve $C$. 
Suppose that $(X,f^{-1}(p)_{\mathrm{red}})$ is a projective log smooth pair for any $p\in C$.
Let $H$ be a nef $\mathbb{Q}$-line bundle on $X$ and $L$ be an ample $\mathbb{Q}$-line bundle on $C$.
Then there exists a positive constant $M>0$ such that for any rational number $\epsilon>0$, $D_\epsilon\in|\epsilon H+L|_{\mathbb{Q}}$ and closed point $P$ of $X$, (cf.~Example \ref{ex-ord_p})$$\mathrm{ord}_P(D_{\epsilon,\mathrm{hor}})\le M\epsilon.$$ 
\end{lem}
\begin{proof}
 It suffices to show that there exists a positive constant $M>0$ satisfying the following:
 \begin{itemize}
\item For any irreducible component $G$ of $f^{-1}(f(P))$ that contains $P$, there exists a curve $\gamma\subset f^{-1}(f(P))$ passing through $P$ but not contained in $D_{\epsilon,\mathrm{hor}}$.
Furthermore, $\gamma\cdot H\le M$.
\end{itemize}
Indeed, Lemma \ref{lem-claim} follows from this since
$$\mathrm{ord}_P(D_{\epsilon,\mathrm{hor}})\le\gamma\cdot D_{\epsilon,\mathrm{hor}}=\gamma\cdot(\epsilon H+L)=\epsilon(\gamma\cdot H)\le\epsilon M.$$

 Next, we show the existence of such $M$ and $\gamma$ as we stated above.
 Fix an ample line bundle $A$ on $X$.
  Let $\mathrm{dim}\,X=n$ and $\mu:\tilde{X}\to X$ be the blow up at $P$ with exceptional divisor $E$.
  By Seshadri's theorem \cite[Theorem 1.4.14]{Laz}, there exists a positive integer $m'_1$ such that $\mu^*(mA)-E$ are ample for any $m\ge m'_1$.
  On the other hand, there exists $m_1''\in\mathbb{Z}_{>0}$ such that $(mA-K_X-G)$ is also ample for any $m\ge m_1''$ and for any vertical prime divisor $G$ on $X$.
   Indeed, since there are only finitely many numerical classes of such $G$, it suffices to check the existence of $m''_1$ for finitely many $G$ (cf.~\cite[Corollary 1.2.24]{Laz}). 
   By setting $m_1:=(n-1)m_1'+m_1''$, we obtain that $\mu^*(m_1A-K_X-G)-(n-1)E$ is ample for any vertical prime divisor $G$ and $P\in X$.
   Choose a vertical prime divisor $G$ such that $P\in G$ and set $\tilde{G}=\mu^{-1}_*G$.
   Then 
$$m_1\mu^*A|_{\tilde{G}}-K_{\tilde{G}}-E|_{\tilde{G}}=(\mu^*(m_1A-K_X-G)-(n-1)E)|_{\tilde{G}}$$
    is ample. 
   Therefore, by \cite[Theorem 1.1]{Ko3}, there exists a positive integer $m_2$ depending only on $n$ such that $m_2(m_1\mu^*A-E)|_{\tilde{G}}$ is ample and generated by global sections.
    Then there exist general divisors $D_1,\cdots,D_{n-2}\in |m_2(m_1\mu^*A-E)|_{\tilde{G}}|$ such that $\gamma'=\bigcap_{i=1}^{n-2}D_i\subset \tilde{G}$ is a smooth curve intersecting with $E$ properly by Bertini's theorem and is not contained in $\mu^*(D_{\epsilon,\mathrm{hor}})$.
     Note that
\begin{align*}
\gamma'\cdot \mu^*D_{\epsilon,\mathrm{hor}}&=m_1^{n-2}m_2^{n-2}A^{n-2}\cdot G\cdot H\le m_1^{n-2}m_2^{n-2}\frac{\mathrm{ord}_G(f^{-1}f(G))}{\mathrm{deg}\,L}(A^{n-2}\cdot L\cdot H).
\end{align*}
We note that $T:=\max_G(\mathrm{ord}_G(f^{-1}f(G)))$ is a positive integer, where $G$ runs over all vertical prime divisors.
By setting $\gamma:=\mu_*\gamma'$ and $M:=m_1^{n-2}m_2^{n-2}\frac{T}{\mathrm{deg}\,L}(A^{n-2}\cdot L\cdot H)$, we see the condition holds as we claimed in the first paragraph.
We complete the proof.
\end{proof}

Furthermore, we prepare the following useful lemma, which follows from the same argument of \cite[Proof of Theorem 9.14, Step.~2]{BHJ}.

\begin{lem}\label{lem--alpha-delta-conv-final-part}
Let $X$ be a projective smooth variety.
For any rational number $0<\epsilon<1$, let $\Delta$ be a $\mathbb{Q}$-divisor such that $(X,\Delta)$ is log smooth and $\epsilon$-sublc.
    If $D$ is an effective $\mathbb{Q}$-divisor such that $\mathrm{ord}_P(D)\le\epsilon$ for any closed point $P\in X$, then $(X,\Delta+D)$ is sublc.
\end{lem}

\begin{proof}
By assumption, we see that $(1-\epsilon)\Delta_{\mathrm{red}}\ge \Delta$.
We see by \cite[Corollary 2.31]{KoMo} that $(X,\Delta_{\mathrm{red}})$ is sublc, i.e.~
\[
A_{X}(E)\ge \mathrm{ord}_{E}(\Delta_{\mathrm{red}})\ge \frac{1}{1-\epsilon}\mathrm{ord}_E(\Delta)
\]
for any prime divisor $E$ over $X$.
Since $\mathrm{ord}_P(D)\le\epsilon$ for any closed point $P\in X$, 
$$\epsilon A_X(E)\ge\mathrm{ord}_{E}(D)$$ 
by \cite[Lemma 8.10]{Ko1} for any prime divisor $E$ over $X$.
Thus, 
\[
A_{(X,\Delta+D)}(E)=(1-\epsilon)\left(A_{X}(E)-\frac{1}{1-\epsilon}\mathrm{ord}_E(\Delta)\right)+\epsilon A_X(E)-\mathrm{ord}_E(D)\ge0,
\]
which means that $(X,\Delta+D)$ is sublc.
\end{proof}

\begin{de}\label{dalpha}
Let $f\colon(X,\Delta)\to C$ be an algebraic fiber space subpair over a smooth curve $C$.
Suppose that $(X,\Delta)$ is sublc.
Set $B=\sum_{p\in C}(1-\mathrm{lct}_{(X,\Delta)}(f^*p))p$ as a $\mathbb{Q}$-divisor on $C$ and call this the {\it discriminant} divisor with respect to $f$.
\end{de}

In this situation, we note that $\alpha_{(C,B)}(L)=\frac{1}{\mathrm{deg}\,L}\min_{p\in C}\mathrm{lct}_{(X,\Delta)}(f^{-1}(p))$.
With this in mind, we prove the following result on the ``convergence of $\alpha$-invariant".
 
\begin{thm}\label{ctn}
Let $f\colon(X,\Delta,H)\to (C,L)$ be a polarized algebraic fiber space pair over a smooth curve $C$. 
Suppose that $(X,\Delta)$ is klt.
 Then, 
\[
\lim_{\epsilon\to 0}\alpha_{(X,\Delta)}(L+\epsilon H)=\alpha_{(C,B)}(L),
\]
where $B$ is the discriminant divisor defined as Definition \ref{dalpha}.
\end{thm}
 
\begin{proof}
We may assume that $\mathrm{deg}\,L=1$.
It is easy to see that $$\limsup_{\epsilon\to0}\alpha_{(X,\Delta)}(L+\epsilon H)\le\alpha_{(C,B)}(L)$$
by $\alpha_{(C,B)}(L)=\frac{1}{\mathrm{deg}\,L}\min_{p\in C}\mathrm{lct}_{(X,\Delta)}(f^{-1}(p))$.
 Thus, it suffices to show that
\begin{equation}\label{eq-liminf}
\liminf_{\epsilon\to 0}\alpha_{(X,\Delta)}(L+\epsilon H)\ge\alpha_{(C,B)}(L).
\end{equation}
In other words, we have only to show that for any sufficiently small $\delta>0$, there exists $\epsilon_0>0$ such that for any $\epsilon\in\mathbb{Q}\cap(0, \epsilon_0)$ and $D_\epsilon\in |L+\epsilon H|_{\mathbb{Q}}$, $(X,\Delta+(\alpha_{(C,B)}(L)-\delta) D_\epsilon)$ is lc. 
Take a log resolution $g:X'\to X$ of $(X,\Delta)$ and let $f'=f\circ g$. 
We may further assume that the reduced structure of any fiber of $f'$ is snc. 
Let $\Delta'$ be a $\mathbb{Q}$-divisor on $X'$ such that $g_*\Delta'=\Delta$ and
\[
K_{X'}+\Delta'=g^*(K_X+\Delta).
\]
Fix a very ample divisor $A$ on $X'$ and set $D_\epsilon':=g^*D_\epsilon$.
 Then, we decompose $D'_\epsilon=D'_{\epsilon,\mathrm{vert}}+D'_{\epsilon,\mathrm{hor}}$ (see the first paragraph of \S\ref{Notat}). 

In two paragraphs, we estimate the multiplicities of $D'_{\epsilon,\text{vert}}$.
Put $n$ as the dimension of $X$.
Let
\begin{equation}
    D'_{\epsilon,\mathrm{vert}}=\sum_{i=1}^r a_iF_i+\Theta,\label{eq--coeff-declaration}
\end{equation}
where each $F_i$ is a fiber of $f'$, $a_i\ge 0$ and $\Theta$ is an effective $\mathbb{Q}$-divisor whose support contains no fiber of $f'$. 
In this paragraph, we deal with $\Theta$'s coefficients
.
 We first note that there exist only finitely many singular fibers of $f'$.
 Let $\{f'^{-1}(c_s)\}_{s=1}^l$ be the set of all singular fibers and take
 one of them $f'^{-1}(c_s)$.
 Let $f'^{-1}(c_
 s)=\sum_{j=0}^k n_jC_j$ be the irreducible decomposition and $\sum b_jC_j$ be the restriction of $\Theta$ to the neighborhood of $f'^{-1}(c_s)$. 
 Here, one of $b_j$ is zero and by renumbering $j$, we may assume that $b_0=0$.
 Let $\mathrm{Span}_{\mathbb{R}}\langle C_j \rangle_{j=1}^k$ be an $\mathbb{R}$-vector space generated by $\{C_j\}_{j=1}^k$ and set a bilinear pairing $\langle\cdot,\cdot\rangle$ defined by $\langle E_1 ,E_2\rangle :=A^{n-2}\cdot E_1\cdot E_2$.
 By applying the Zariski lemma (cf.~\cite[Lemma 1]{LX}), $(\mathrm{Span}_{\mathbb{R}}\langle C_j \rangle_{j=1}^k,\sqrt{-\langle\cdot ,\cdot\rangle})$ is a finite dimensional normed space and hence the norm $\sqrt{-\langle\cdot ,\cdot\rangle}$ is equivalent to the supremum norm. 
 Therefore, there exists a constant $\beta>0$ such that 
 \begin{equation}
     -\left\langle\sum_{j=1}^kd_jC_j,\sum_{j=1}^k d_jC_j\right\rangle\ge \beta (\max\{ d_j\})^2\label{eq-riesz-1}
 \end{equation}
 for any $(d_j)_{j=1}^k\in\mathbb{R}^{k}$. 
 Since $D'_\epsilon$ is nef,
\begin{equation}
    \sum_{j=1}^k b_jC_j\cdot \left(\sum_{j=1}^k b_jC_j +D'_{\epsilon,\mathrm{hor}}\right)\cdot A^{n-2}=\sum_{j=1}^k b_jC_j\cdot g^*(\epsilon H+L)\cdot A^{n-2}\ge0.\label{eq-hodge-inequality-d_e}
\end{equation}
Note that if we set $m=\min_jn_j>0$ then $\frac{\max\{ b_j\}}{m}f'^{-1}(c_s)-\sum_{j=1}^k b_jC_j$ is effective. Thus, we obtain
\begin{align*}
\frac{\max\{ b_j\}}{m}\epsilon g^*H\cdot g^*L&\cdot A^{n-2}=\frac{\max\{ b_j\}}{m}g^*f^{-1}(c_s)\cdot g^*D_\epsilon\cdot A^{n-2}=\frac{\max\{ b_j\}}{m}f'^{-1}(c_s)\cdot D'_{\epsilon,\mathrm{hor}}\cdot A^{n-2}\\
&\ge \left(\sum_{j=1}^k b_jC_j\cdot D'_{\epsilon,\mathrm{hor}}\right)\cdot A^{n-2}\ge-\left\langle\sum b_jC_j,\sum b_jC_j\right\rangle\ge\beta (\max\{ b_j\})^2.
\end{align*}
Here, we apply (\ref{eq-riesz-1}) and (\ref{eq-hodge-inequality-d_e}) in the second line of the above inequality.
Therefore, 
$$\max\{ b_j\}\le \frac{\epsilon}{m\beta}g^*H\cdot f'^*L\cdot A^{n-2}.$$
 The constants $\beta$ and $m$ depend only on finitely many choices of irreducible components of $f'^{-1}(c_s)$.
 Hence, there exists a constant $M'>0$ such that $M'\epsilon\sum_{s=1}^lf'^{-1}(c_s)-\Theta$ is effective for any $D_{\epsilon}$.
 
 Next, we bound the coefficients $a_i$ in $D'_{\epsilon,\mathrm{vert}}=\sum_{i=1}^r a_iF_i+\Theta$ in (\ref{eq--coeff-declaration}).
 We have that 
 \[
 \left(\sum_{i=1}^r a_iF_i\right)\cdot A^{n-1}\le A^{n-1}\cdot g^*(\epsilon H+L) 
 \]
 and hence $\sum_{i=1}^ra_i\le1+\epsilon(\frac{A^{n-1}\cdot g^*H}{A^{n-1}\cdot f'^*L})$.
 By the fact that there are only finitely many singular fibers of $f'$, set $T:=\max_{c,Q}\mathrm{ord}_Q{f'^{-1}(c)}>0$ where $c\in C$ runs over all closed points and $Q$ runs over all irreducible components of $f'^{-1}(c)$. 
 Then we have that
 \[
\sum_{i=1}^r a_iF_i=\sum_{i=1}^r a'_iF_i+B',
 \]
 where $a'_i>0$ such that $\sum_{i=1}^r a'_i\le1$ and $B'$ is an effective snc $\mathbb{Q}$-divisor such that $\mathrm{ord}_P(B')\le M''\epsilon$ for any closed point $P\in X'$, where $M''=nT(\frac{A^{n-1}\cdot g^*H}{A^{n-1}\cdot f'^*L})$.
 We note that in particular, it holds that $0<a'_i\le1$.
 
 Thus, we have by Lemma \ref{lem-claim}
 and by the above arguments that there exists a positive constant $M>0$ such that for any sufficiently small rational number $\epsilon>0$ and $D_{\epsilon}\in|\epsilon H+L|_{\mathbb{Q}}$,
 \[
 g^* D_{\epsilon}= \sum_{i=1}^r d_iF_i+\Xi,
 \]
 where $0<d_i\le1$, $F_i$ is a fiber of $f'$ and $\Xi$ is a $\mathbb{Q}$-divisor such that $\mathrm{ord}_P(\Xi)\le M\epsilon$ for any closed point $P\in X'$.
 By the definition of $\alpha_{(C,B)}(L)$ and the property of $X'$, $(X',\Delta'+(\alpha_{(C,B)}(L)-\delta)\sum_{i=1}^r d_iF_i)$ is log smooth and $\delta$-sublc if we choose $\delta$ small enough. 
 Thus, if $0<\epsilon<\epsilon_0:=\min\{1,\frac{\delta
 }{M(\alpha_{(C,B)}(L)-\delta)}\}$, then we have that $(\alpha_{(C,B)}(L)-\delta)\mathrm{ord}_P(\Xi)\le \delta$ for any closed point $P\in X'$.
 By Lemma \ref{lem--alpha-delta-conv-final-part}, 
 $(X',\Delta'+(\alpha_{(C,B)}(L)-\delta)g^* D_{\epsilon})$ is sublc.
 This means that for any $\epsilon\in(0,\epsilon_0)\cap\mathbb{Q}$ and $D_\epsilon\in |L+\epsilon H|_{\mathbb{Q}}$, $(X,\Delta+(\alpha_{(C,B)}(L)-\delta) D_\epsilon)$ is lc. 
 Hence, we have (\ref{eq-liminf})
 and complete the proof.
\end{proof}

By applying Theorem \ref{ctn}, we obtain the following theorem for the $\delta$-invariant.
\begin{thm}\label{ctnd}
Let $f\colon(X,\Delta,H)\to (C,L)$ be a polarized algebraic fiber space pair. Suppose that $(X,\Delta)$ is a klt pair, $C$ is a smooth curve and $B$ is the discriminant divisor.
Then
\[
\lim_{\epsilon\to 0}\delta_{(X,\Delta)}(\epsilon H+L)=\delta_{(C,B)}(L).
\]
\end{thm}

\begin{proof}
By Example \ref{ex-delta-alpha-curve}, $\delta_{(C,B)}(L)=2\alpha_{(C,B)}(L)$ and 
it is enough to show that
\[
\lim_{\epsilon\to 0}\delta_{(X,\Delta)}(\epsilon H+L)=2\alpha_{(C,B)}(L)
\]
instead.
We may assume that $L=f^*c$ for some general point $c\in C$ since the $\delta_{(X,\Delta)}(\epsilon H+ L)$ depends only on the numerical equivalence class of $L$.
First, we see that
\begin{equation}
\limsup_{\epsilon\to0}\delta_{(X,\Delta)}(\epsilon H+L)\le2\alpha_{(C,B)}(L).\label{eq-first}
\end{equation}
Indeed, recall that $\alpha_{(C,B)}(L)=\mathrm{lct}_{(X,\Delta)}(F)$ for some fiber $F=f^{-1}(p)$.
For any $\delta>0$, pick a prime divisor $E$ over $X$ such that $A_{(X,\Delta+(\alpha_{(C,B)}(L)+\delta)F)}(E)<0$.
We see by definition that
\[
S_{\epsilon H+L}(E)\ge\frac{\mathrm{ord}_E(F)}{(\epsilon H+L)^n}\int^{\infty}_0\mathrm{vol}(\epsilon H+L-xF)dx.
\]
It follows that
\[
\int^{\infty}_0\mathrm{vol}(\epsilon H+L-xF)dx=\int^{\infty}_0\mathrm{vol}(\epsilon H+(1-x)L)dx
\]
from the fact that the volume depends only on the numerical class (cf.~\cite[Theorem 1.4.40]{Laz}).
On the other hand, we have
\begin{align*}
\mathrm{vol}(\epsilon H+(1-x)L)=n(1-x)\epsilon^{n-1}H^{n-1}\cdot L+\epsilon^n H^n
\end{align*}
for any $x\in\mathbb{Q}\cap(0,1)$.
Take $\tau\in\mathbb{Q}_{>0}$ such that $H-\tau L$ is not pseudoeffective.
Then 
\begin{align*}
    \frac{1}{2}\epsilon^{n-1}(H^{n-1}\cdot L)+\epsilon^nH^n&=\int^1_{0}\mathrm{vol}(\epsilon H+(1-x)L)dx\le\int^\infty_{0}\mathrm{vol}(\epsilon H+(1-x)L)dx\\
    &=\int^{1+\epsilon\tau}_{0}\mathrm{vol}(\epsilon H+(1-x)L)dx\le\frac{1}{2}\epsilon^{n-1}(H^{n-1}\cdot L)+(\epsilon^n+\epsilon^{n+1}\tau)H^n
\end{align*}
and hence we have that
\begin{equation}\label{eq--infinity--integral}
    \lim_{\epsilon\to0}\frac{\int^{\infty}_0\mathrm{vol}(\epsilon H+L-xF)dx}{(\epsilon H+L)^n}=\lim_{\epsilon\to0}\frac{\int^{\infty}_0\mathrm{vol}(\epsilon H+(1-x)L)dx}{(\epsilon H+L)^n}=\frac{1}{2}.
\end{equation}
Thus, we have  
\[
A_{(X,\Delta)}(E)<2(\alpha_{(C,B)}(L)+\delta)S_{\epsilon H+L}(E)
\]
for any sufficiently small rational number $\epsilon>0$. 
This means that 
\[
\limsup_{\epsilon\to0}\delta_{(X,\Delta)}(\epsilon H+L)\le2(\alpha_{(C,B)}(L)+\delta)
\]
for any $\delta>0$ and
that (\ref{eq-first}) holds.

From now on, we show
\begin{equation}\label{eq-dliminf}
\liminf_{\epsilon\to0}\delta_{(X,\Delta)}(\epsilon H+L)\ge2\alpha_{(X,\Delta)}(L).
\end{equation}
    First, we show the following claim.
    \begin{claim}
For any sufficiently small $\epsilon\in\mathbb{Q}_{>0}$, let \[
C_{\epsilon}:=\frac{\int^{\infty}_0\mathrm{vol}(\epsilon H+(1-x)L)dx}{(\epsilon H+L)^n}.
\]
Then for any prime divisor $E$ over $X$, we obtain that
\[
S_{\epsilon H+L}(\mathrm{ord}_E)\le T_{\epsilon H+(1-C_\epsilon)L}(\mathrm{ord}_E).
\]
    \end{claim}
    \begin{proof}[Proof of Claim]
    Take a irreducible and normal fiber $F$ of $f$ that does not contain the generic point of the center of $E$.
For any $m\in\mathbb{Z}_{>0}$ such that $m\epsilon H+mL$ is a Cartier divisor, we take a basis $\{s_i\}$ of $H^0(X,m\epsilon H+mL)$ compatible with $E$ and $F$ by Lemma \ref{lem--compatible--basis}. 
Let $D$ be the $m$-basis type divisor of $\epsilon H+L$ defined by $\{s_i\}$.
Then we have by Lemma \ref{lem--compatible--basis} that
\[
\mathrm{ord}_F(D)=S^m_{\epsilon H+L}(\mathrm{ord}_F).
\]
This means that $D-S^m_{\epsilon H+L}(\mathrm{ord}_F)F$ is effective and 
\[
S^{m}_{\epsilon H+L}(\mathrm{ord}_E)=\mathrm{ord}_E(D)=\mathrm{ord}_E(D-S^m_{\epsilon H+L}(\mathrm{ord}_F)F).
\]
We note that $\lim_{m\to\infty}S^m_{\epsilon H+L}(\mathrm{ord}_F)=S_{\epsilon H+L}(\mathrm{ord}_F)=C_{\epsilon}$.
By (\ref{eq--infinity--integral}), we have $\lim_{\epsilon\to 0}C_{\epsilon}=\frac{1}{2}$. 
Thus, we may assume that $C_\epsilon<1$ and $S^m_{\epsilon H+L}(\mathrm{ord}_F)<1$ by replacing $m$ with a larger integer and $\epsilon$ with a smaller positive rational number.
Then, we obtain that
\[
S^{m}_{\epsilon H+L}(\mathrm{ord}_E)\le T_{\epsilon H+L-S^m_{\epsilon H+L}(\mathrm{ord}_F)F}(\mathrm{ord}_E).
\]
It is not hard to see that \[T_{\epsilon H+L-S^m_{\epsilon H+L}(\mathrm{ord}_F)F}(\mathrm{ord}_E)=T_{\epsilon H+(1-S^m_{\epsilon H+L}(\mathrm{ord}_F))L}(\mathrm{ord}_E)\]
by $L\equiv F$.
If $C_{\epsilon}\ge S^m_{\epsilon H+L}(\mathrm{ord}_F)$, we obtain that
\begin{align*}
T_{\epsilon H+(1-C_\epsilon)L}(\mathrm{ord}_E)&=\frac{1-C_\epsilon}{1-S^m_{\epsilon H+L}(\mathrm{ord}_F)}T_{\frac{\epsilon(1-S^m_{\epsilon H+L}(\mathrm{ord}_F))}{1-C_\epsilon} H+(1-S^m_{\epsilon H+L}(\mathrm{ord}_F))L}(\mathrm{ord}_E)\\
&\ge \frac{1-C_\epsilon}{1-S^m_{\epsilon H+L}(\mathrm{ord}_F)}T_{\epsilon H+(1-S^m_{\epsilon H+L}(\mathrm{ord}_F))L}(\mathrm{ord}_E).
\end{align*}
Otherwise, we have $T_{\epsilon H+(1-C_\epsilon)L}(\mathrm{ord}_E)\ge T_{\epsilon H+(1-S^m_{\epsilon H+L}(\mathrm{ord}_F))L}(\mathrm{ord}_E)$.
Thus, we obtain that
\begin{align*}
S_{\epsilon H+L}(\mathrm{ord}_E)&=\lim_{m\to\infty}S^m_{\epsilon H+L}(\mathrm{ord}_E)\le \lim_{m\to\infty}\min\left\{1,\frac{1-C_\epsilon}{1-S^m_{\epsilon H+L}(\mathrm{ord}_F)}\right\}T_{\epsilon H+(1-C_\epsilon)L}(\mathrm{ord}_E)\\
&=T_{\epsilon H+(1-C_\epsilon)L}(\mathrm{ord}_E).
\end{align*}
  We complete the proof of Claim.  \end{proof}
  Note that $T_{\epsilon H+(1-C_\epsilon)L}(\mathrm{ord}_E)=(1-C_{\epsilon})T_{\frac{\epsilon}{1-C_\epsilon}H+L}(\mathrm{ord}_E)$.
  Recall that $\lim_{\epsilon\to 0}C_{\epsilon}=\frac{1}{2}$.
  By Claim and Theorem \ref{ctn}, we obtain that 
  \begin{align*}
\liminf_{\epsilon\to 0}\delta_{(X,\Delta)}(\epsilon H+L) &=\liminf_{\epsilon\to0}\inf_{E}\frac{A_{(X,\Delta)}(E)}{S_{\epsilon H+L}(\mathrm{ord}_E)}     \ge \liminf_{\epsilon\to0}(1-C_{\epsilon})^{-1}\inf_{E}\frac{A_{(X,\Delta)}(E)}{T_{\frac{\epsilon}{1-C_\epsilon}H+L}(\mathrm{ord}_E)} \\
&=\liminf_{\epsilon\to0}(1-C_{\epsilon})^{-1}\alpha_{(X,\Delta)}\left(\frac{\epsilon}{1-C_\epsilon}H+L\right)=2\alpha_{(C,B)}(L),
  \end{align*}
  where $E$ runs over all prime divisors over $X$.
  This is nothing but (\ref{eq-dliminf}).
  Then
(\ref{eq-first}) and (\ref{eq-dliminf}) complete the proof.
\end{proof}

\begin{rem}
There is a fact that the alpha invariant is continuous on the K\"{a}hler cone \cite[Theorem 3.5]{Der}. 
For the delta invariant, a similar result known by Zhang (cf.~\cite{Z2}).
Due to Theorems \ref{ctn} and \ref{ctnd}, we can extend these invariants continuously to the fiber class.
\end{rem}

Now, we are ready to show Theorem \ref{dd}.
First, we prepare the following lemma on J-stability.
  
\begin{lem}\label{Jcomp}
Let $f:(X,H)\to (C,L)$ be a polarized algebraic fiber space.
Suppose that $C$ is a smooth curve and $H$ is ample. 
Then for any sufficiently small rational number $\epsilon>0$ and semiample test configuration $(\mathcal{X},\mathcal{M})$ for $(X,\epsilon H+L)$,
\[
(\mathcal{J}^{-f^*L})^{\mathrm{NA}}(\mathcal{X},\mathcal{M})> -(I^\mathrm{NA}(\mathcal{X},\mathcal{M})-J^\mathrm{NA}(\mathcal{X},\mathcal{M})).
\]
\end{lem}
 
\begin{proof}
It suffices to show that $(X,\epsilon H+L)$ is uniformly J$^{H}$-stable since $-L+(\epsilon H+L)= \epsilon H$. By the ampleness of
\[
n\frac{H\cdot (\epsilon H+L)^{n-1}}{(\epsilon H+L)^n}(\epsilon H+L)-(n-1)H=\frac{\epsilon^2(H^n)H+n((n-1)H^{n-1}\cdot L+\epsilon H^n)L}{n\epsilon H^{n-1}\cdot L+\epsilon^2 H^n}
\]
and by Remark \ref{rem--useful}, $(X,\epsilon H+L)$ is uniformly J$^{H}$-stable.
\end{proof}

 The following gives a sufficient condition for uniform adiabatic K-stability.
 
\begin{thm}\label{ctn3}
Let $f:(X,\Delta,H)\to(\mathbb{P}^1,\mathcal{O}(1))$ be a polarized klt-trivial fibration, let $M$ be the moduli divisor and let $B$ be the discriminant divisor. 
Suppose that $-(K_{\mathbb{P}^1}+M+B)$ is ample. 
If $(\mathbb{P}^1,B,M,\mathcal{O}(1))$ is uniformly K-stable, then
$(X,\Delta,H)$ is uniformly adiabatically K-stable over $(\mathbb{P}^1,\mathcal{O}(1))$. Moreover, if $\Delta=0$, $X$ is smooth and $(\mathbb{P}^1,B,M,\mathcal{O}(1))$ is uniformly adiabatically K-stable, then $(X,H)$ has cscK metrics adiabatically.
\end{thm}
\begin{proof}
Let $\mathrm{deg}\, (B+M)=s\ge0$ and then we have that
\[
\delta_{(C,B)}(\mathcal{O}(1))=(2-s)\delta_{(C,B)}(-(K_C+M+B))>2-s
\]
by Corollary \ref{appc}.
Therefore, there exist positive constants $\epsilon_0$ and $\delta$ such that 
\[
\delta_{(X,\Delta)}(\epsilon H+\mathcal{O}(1))\ge2-s+\delta
\]
for any $\epsilon\in(0,\epsilon_0)\cap\mathbb{Q}$ by Theorem \ref{ctnd}.
On the other hand, 
\[
K_X+\Delta\sim_{\mathbb{Q}}f^*(K_{\mathbb{P}^1}+M+B)\sim_{\mathbb{Q}} -(2-s)f^*\mathcal{O}(1).
\]
 Therefore,
\[
(\mathcal{J}^{K_X+\Delta})^\mathrm{NA}(\mathcal{X},\mathcal{M})\ge -(2-s)(I^\mathrm{NA}(\mathcal{X},\mathcal{M})-J^\mathrm{NA}(\mathcal{X},\mathcal{M}))
\]
for any $\epsilon\in(0,\epsilon_0)\cap\mathbb{Q}$ and semiample test configuration $(\mathcal{X},\mathcal{M})$ for $(X,\epsilon H+\mathcal{O}(1))$ by Lemma \ref{Jcomp}. Furthermore, we have by Theorem \ref{bhjz} that
\begin{align*}
    M_\Delta^{\mathrm{NA}}(\mathcal{X},\mathcal{M})&=H^{\mathrm{NA}}_\Delta(\mathcal{X},\mathcal{M})+(\mathcal{J}^{K_X+\Delta})^\mathrm{NA}(\mathcal{X},\mathcal{M})\\
    &\ge\delta(I^\mathrm{NA}(\mathcal{X},\mathcal{M})-J^\mathrm{NA}(\mathcal{X},\mathcal{M})).
\end{align*}
This means that $(X,\Delta,H)$ is uniformly adiabatically K-stable over $(C,L)$.
We complete the proof of the first assertion.
 
On the other hand, we can show that the K-energy is coercive by \cite[Proposition 3.5]{Z} and \cite[Theorem 1.1]{C} as the proof of the first assetion. Thus, the second assertion follows from \cite[Theorem 4.1]{Ch}.
\end{proof}

\begin{proof}[Proof of Theorem \ref{dd}]
The assertion for the case when $K_X+\Delta$ is nef follows from Proposition \ref{bhj9394} and Corollary \ref{cor--minimal-stable}.
On the other hand, the assertion for the case when $K_X+\Delta$
 is not nef follows from Theorems \ref{stp1} and \ref{ctn3}.
 \end{proof}

\subsection{Applications to Rational Elliptic Surfaces}\label{RES}

We apply Theorems \ref{dd} and \ref{ff} to rational elliptic surfaces.
For this, we first classify K-stable log-twisted Fano curves $(\mathbb{P}^1,B,M,L)$ associated with rational elliptic surfaces.
\begin{prop}\label{P^1}
Let $f:X\to\mathbb{P}^1$ be a rational elliptic surface of index $m(X)$, $B=\sum_{Q\in\mathbb{P}^1}s_QQ$ be the discriminant divisor and $M$ be the moduli divisor (cf.~Definition \ref{de-mult}).

Then $(\mathbb{P}^1,B,M,\mathcal{O}(1))$ is uniformly K-stable (resp, K-semistable, K-unstable) if and only if $\max_Qs_Q< \frac{2m(X)-1}{2m(X)}$ (resp., $\max_Qs_Q\le \frac{2m(X)-1}{2m(X)}$, $\max_Qs_Q> \frac{2m(X)-1}{2m(X)}$).
\end{prop}
\begin{proof}
It immediately follows from Corollary \ref{appc}.
\end{proof}
We obtain the following by the Kodaira canonical bundle formula.
\begin{cor}\label{II}
Let $f:X\to \mathbb{P}^1$ be a rational elliptic surface.
Then the log twisted pair $(\mathbb{P}^1,B,M,\mathcal{O}(1))$ induced by $f$ is K-unstable if and only if one of the following holds.
\begin{itemize}
 \item $m(X)=1$ and $X$ admits one of $II^*$, $III^*$ or $IV^*$-type fibers,
 \item $m(X)=2$ and $X$ admits at least one $II^*$-type fiber.
 \end{itemize}
 On the other hand, $(\mathbb{P}^1,B,M,\mathcal{O}(1))$ is K-stable if and only if one of the following holds.
 \begin{itemize}
     \item $m(X)=1$ and $X$ has at most reduced fibers,
\item $m(X)=2$ and $X$ has no $II^*$ or $III^*$-type fiber,
\item $m(X)=3$ and $X$ has no $II^*$-type fiber,
\item $m(X)\ge 4$.
 \end{itemize}
\end{cor}
\begin{proof}
This immediately follows from Proposition \ref{P^1}.
\end{proof}
Therefore, we obtain the following by Theorems \ref{dd}, \ref{ff} and Corollary \ref{II}.
\begin{cor}\label{K-fib}
Let $f\colon X\to \mathbb{P}^1$ be a smooth rational elliptic surface with an ample $\mathbb{Q}$-Cartier $\mathbb{Q}$-divisor $H$ on $X$. Then the following hold.
\begin{itemize}
\item If $m(X)=1$ and $X$ admits one of $II^*$, $III^*$ or $IV^*$-type fibers, then $(X,H)$ is adiabatically K-unstable over $(\mathbb{P}^1,\mathcal{O}(1))$.
\item If $m(X)=2$ and $X$ admits at least one $II^*$-type fiber, then $(X,H)$ is adiabatically K-unstable over $(\mathbb{P}^1,\mathcal{O}(1))$.
\item If $m(X)=1$, then $X$ has at most reduced fibers if and only if $(X,H)$ is uniformly adiabatically K-stable. If $X$ is further smooth, $(X,H)$ has cscK metrics adiabatically.
\item If $m(X)=2$, then $X$ has no $II^*$ or $III^*$-type fiber if and only if $(X,H)$ is uniformly adiabatically K-stable. If $X$ is further smooth, $(X,H)$ has cscK metrics adiabatically.
\item If $m(X)=3$, then $X$ has no $II^*$-type fiber if and only if $(X,H)$ is uniformly adiabatically K-stable. If $X$ is further smooth, $(X,H)$ has cscK metrics adiabatically.
\item If $m(X)\ge 4$, then $(X,H)$ is uniformly adiabatically K-stable. If $X$ is further smooth, $(X,H)$ has cscK metrics adiabatically.
\end{itemize}
\end{cor}

\begin{rem}\label{rem-final}
It is well-known that there exists a smooth rational elliptic surface $f\colon X\to\mathbb{P}^1$ with a $II^*$-type fiber of $m(X)=1$.
Let $B$ (resp.~$M$) be the discriminant (resp.~moduli) divisor on $\mathbb{P}^1$.
Note that $\mathrm{deg}(B+M)=1$.
By Corollary \ref{K-fib} below, we have that $f\colon (X,H)\to (\mathbb{P}^1,\mathcal{O}(1))$ is adiabatically K-unstable.
However, the twisted pair $(\mathbb{P}^1,0,B+M,\mathcal{O}(1))$ is uniformly K-stable.
Indeed, $(\mathbb{P}^1,0,T,\mathcal{O}(1))$ is uniformly K-stable for any $\mathbb{Q}$-line bundle of $\mathrm{deg}\,T>0$ by Corollary \ref{appc}.
 Hence, it is better to consider log-twisted K-stability of the base scheme to detect adiabatic K-stability of an arbitrary klt-trivial fibration at least when the base is a curve (cf., Theorem \ref{can}, \cite[Corollary 4.3]{DR2}).
 \end{rem}

\appendix
\def\thesection{\Alph{section}}

\section{Auxiliary results on log-twisted Fano pairs}\label{appendices}
  In this appendix, we will prove some results to apply in \S \ref{sec-3}.
  The main results are Theorems \ref{bhjz} and \ref{appendix}.

\subsection{Ding stability of log-twisted Fano pairs and $\delta$-invariant}\label{sec-logtwisteddelta}
 
In this subsection, we show that the log-twisted $\beta$-invariant classifies Ding-stable log-twisted Fano pairs $(X,B,T,L)$ as \cite{Fjt}.
 First, we introduce the log-twisted Ding invariant of a test configuration as follows.

\begin{de}
    Let $(X,B)$ be an $n$-dimensional subklt pair such that $L:=-(K_X+B)$ is ample and $(\mathcal{X},\mathcal{L})$ be a semiample test configuration for $(X,L)$.
    We set the {\it log Ding invariant} as 
    \[
\mathrm{Ding}_{B}(\mathcal{X},\mathcal{L})=-\frac{\mathcal{L}^{n+1}}{(n+1)L^n}-1+\mathrm{lct}_{(\mathcal{X},\mathcal{B}+D)}(\mathcal{X}_0),
\]
where $\mathcal{B}$ is the strict transform of $B_{\mathbb{P}^1}$ on $\mathcal{X}$ and $D$ is a $\mathbb{Q}$-divisor on $\mathcal{X}$ defined by
\begin{align*}
\mathrm{Supp}\, D&\subset\mathcal{X}_0,\quad\mathrm{and}\\
D&\sim_{\mathbb{Q}} -(K_{\mathcal{X}/\mathbb{P}^1}+\mathcal{B})-\mathcal{L}.
\end{align*}
Note that $D$ is uniquely determined.
Thus, we simply denote that $D=-(K_{\mathcal{X}/\mathbb{P}^1}+\mathcal{B})-\mathcal{L}$. 

We set {\it the log $\beta$-invariant} of a prime divisor $F$ over $X$ with respect to $(X,B)$ as
\[
\beta_{(X,B)}(F)=A_{(X,B)}(F)\mathrm{vol}(L)-\int^{T_L(F)}_0\mathrm{vol}(L-xF)dx.
\] 
\end{de}

\begin{de}\label{detw}
Let $(X,B,T,L=-K_X-B-T)$ be an $n$-dimensional log-twisted Fano pair such that $(\mathcal{X},\mathcal{L})$ be a normal semiample test configuration for $(X,L)$. Let also $\rho:\mathcal{X}\dashrightarrow X_{\mathbb{P}^1}$ be the {\it canonical birational map}, i.e., $\rho$ is the identity over $\mathbb{P}^1\setminus \{0\}$.
Take a general $\mathbb{Q}$-divisor $T_1\sim_{\mathbb{Q}}T$.
We may assume that $(X,B+T_1)$ is subklt. 
Then the {\it log-twisted Ding invariant} is defined as
\[
\mathrm{Ding}_{(B,T)}(\mathcal{X},\mathcal{L})=-\frac{\mathcal{L}^{n+1}}{(n+1)L^n}-1+\mathrm{lct}_{(\mathcal{X},\mathcal{B}+D+\rho_*^{-1}T_{1,\mathbb{P}^1})}(\mathcal{X}_0),
\]
where $\mathcal{B}$ is the strict transform of $B_{\mathbb{P}^1}$ on $\mathcal{X}$ and $D$ is a $\mathbb{Q}$-divisor on $\mathcal{X}$ defined by
\begin{align*}
\mathrm{Supp}\, D&\subset\mathcal{X}_0,\quad\mathrm{and}\\
D&\sim_{\mathbb{Q}} -(K_{\mathcal{X}/\mathbb{P}^1}+\mathcal{B}+\rho_*^{-1}T_{1,\mathbb{P}^1})-\mathcal{L}.
\end{align*}
Note that $D$ is uniquely determined and independent from the choice of general $T_1$. 
Indeed, we see that for another sufficiently general $\mathbb{Q}$-divisor $T_2\sim_{\mathbb{Q}}T$, $\rho_*^{-1}T_{1,\mathbb{P}^1}\sim\rho_*^{-1}T_{2,\mathbb{P}^1}$  as \cite[Lemma 2.12]{BLZ}.
Thus, we simply denote that $D=-(K_{\mathcal{X}/\mathbb{P}^1}+\mathcal{B}+\rho_*^{-1}T_{1,\mathbb{P}^1})-\mathcal{L}$. 
On the other hand, let $F$ be a prime divisor over $X$. 
Then {\it the log-twisted $\beta$-invariant} of $F$ with respect to $(X,B,T)$ is defined as
\[
\beta_{(X,B,T)}(F)=A_{(X,B)}(F)\mathrm{vol}(L)-\int^{T_L(F)}_0\mathrm{vol}(L-xF)dx=\mathrm{vol}(L)(A_{(X,B)}(F)-S_L(F)).
\]
We also set
\[
j_{L}(F)=T_L(F)\mathrm{vol}(L)-\int^{T_L(F)}_0\mathrm{vol}(L-xF)dx.
\]
\end{de}

 The following lemma shows that the invariants of Definition \ref{detw} are independent of the choice of a general divisor.
\begin{lem}\label{dtw2}
Let $(X,B,T,L=-K_X-B-T)$ be a log-twisted Fano pair and $(\mathcal{X},\mathcal{L})$ a normal semiample test configuration. 
Then, $\mathrm{Ding}_{B+T_1}(\mathcal{X},\mathcal{L})$ is independent of the choice of a general $\mathbb{Q}$-divisor $T_1\sim_{\mathbb{Q}}T$. 
Moreover, for any prime divisor $F$ over $X$ and general $\mathbb{Q}$-divisor $T_1\sim_{\mathbb{Q}}T$, 
\begin{align*}
\beta_{(X,B,T)}(F)=\beta_{(X,B+T_1)}(F).
\end{align*}
\end{lem}
 
\begin{proof}
We first deal with the first assertion.
It is easy to see as \cite[Proposition 2.5 (2)]{Fjt} that if $g:\mathcal{Z}\to \mathcal{X}$ is a canonical morphism of test configurations for $X$ and $\mathcal{Z}$ dominates $X_{\mathbb{P}^1}$ and $T_3\sim_{\mathbb{Q}}T$ satisfies that $(X,B+T_3)$ is subklt, then  $$\mathrm{Ding}_{B+T_3}(\mathcal{X},\mathcal{L})=\mathrm{Ding}_{B+T_3}(\mathcal{Z},g^*\mathcal{L}).$$
Fix such $\mathcal{Z}$.
Take $T_1\sim_{\mathbb{Q}}T$ general enough such that $\rho_{\mathcal{Z}}^*T_{1,\mathbb{P}^1}$ contains no component of $\mathcal{Z}_0$, where $\rho_{\mathcal{Z}}:\mathcal{Z}\to X_{\mathbb{P}^1}$ is the canonical morphism.
This means that $\rho_{\mathcal{Z}}^*T_{1,\mathbb{P}^1}=(\rho_{\mathcal{Z}})^{-1}_*T_{1,\mathbb{P}^1}$.
Let $\mathcal{B}_{\mathcal{Z}}$ be the strict transform of $B\times\mathbb{P}^1$ and $D_{\mathcal{Z}}$ be the $\mathbb{Q}$-divisor on $\mathcal{Z}$ defined by 
\[
D_{\mathcal{Z}}=-(K_{\mathcal{Z}/\mathbb{P}^1}+\mathcal{B}_{\mathcal{Z}}+\rho_{\mathcal{Z}}^*T_{1,\mathbb{P}^1})-g^*\mathcal{L}.
\]
Let $D=-(K_{\mathcal{X}/\mathbb{P}^1}+\mathcal{B}+\rho_*^{-1}T_{1,\mathbb{P}^1})-\mathcal{L}$.
Then we see that $K_\mathcal{Z}+\mathcal{B}_{\mathcal{Z}}+D_{\mathcal{Z}}+\rho_{\mathcal{Z}}^*T_{1,\mathbb{P}^1}=g^*(K_\mathcal{X}+\mathcal{B}+D+\rho_*^{-1}T_{1,\mathbb{P}^1})$.
Thus,
\[
\mathrm{lct}_{(\mathcal{X},\mathcal{B}+D+\rho_*^{-1}T_{1,\mathbb{P}^1})}(\mathcal{X}_0)=\mathrm{lct}_{(\mathcal{Z},\mathcal{B}_{\mathcal{Z}}+D_{\mathcal{Z}}+\rho_{\mathcal{Z}}^*T_{1,\mathbb{P}^1})}(\mathcal{Z}_0).
\]
Here, we see that the right hand side is independent of the choice of general $T_1$.
Indeed, take a log resolution $\sigma:\mathcal{Y}\to\mathcal{Z}$ of $(\mathcal{Z},\mathcal{B}_{\mathcal{Z}}+D_{\mathcal{Z}})$. 
Let $\rho_{\mathcal{Y}}=\rho_{\mathcal{Z}}\circ\sigma$ and $\mathcal{B}_Y$ be the $\mathbb{Q}$-divisor such that
\[
K_{\mathcal{Y}}+\mathcal{B}_Y=\sigma^*(K_{\mathcal{Z}}+\mathcal{B}+\mathcal{D}_{\mathcal{Z}}).
\]
Thus we see by the Bertini theorem that
\[
\mathrm{lct}_{(\mathcal{X},\mathcal{B}+D+\rho^{-1}_*T_{1,\mathbb{P}^1})}(\mathcal{X}_0)=\mathrm{lct}_{(\mathcal{Z},\mathcal{B}_{\mathcal{Z}}+D_{\mathcal{Z}}+\rho_{\mathcal{Z}}^*T_{1,\mathbb{P}^1})}(\mathcal{Z}_0)=\mathrm{lct}_{(\mathcal{Y},\mathcal{B}_Y+\rho_{\mathcal{Y}}^*T_{1,\mathbb{P}^1})}(\mathcal{Y}_0)
\]
is independent from general $T_1$ such that $(\rho_{\mathcal{Y}})^*T_{1,\mathbb{P}^1}+\mathcal{Y}_0+\mathcal{B}_Y$ is an snc divisor, and $(\rho_{\mathcal{Y}})^*T_{1,\mathbb{P}^1}$ and $\mathcal{Y}_0+\mathcal{B}_Y$ have no common component.

For the second assertion, we have only to choose $T_1$ so that the support of $T_1$ does not contain the center of $F$.
\end{proof}

\begin{de}
Let $(X,B,T,L=-(K_X+B+T))$ be a log-twisted Fano pair
. Then, we say that $(X,B,T,L)$ is {\it Ding-semistable} if $\mathrm{Ding}_{(B,T)}(\mathcal{X},\mathcal{L})\ge0$ for any normal semiample test configuration $(\mathcal{X},\mathcal{L})$.
We also say that $(X,B,T,L)$ is {\it uniformly Ding-stable} if there exists $\epsilon>0$ such that $\mathrm{Ding}_{(B,T)}(\mathcal{X},\mathcal{L})\ge\epsilon(\mathcal{J}^{L})^{\mathrm{NA}}(\mathcal{X},\mathcal{L})$ for any normal semiample test configuration $(\mathcal{X},\mathcal{L})$.
\end{de}

Next, we check that \cite[Proposition 2.5]{Fjt} also holds for the log-twisted Ding invariant.

\begin{prop}\label{fjtprop}
Let $(X,B,T,L=-(K_X+B+T))$ be a log-twisted Fano pair
. Then the following hold.
\begin{enumerate}
\item Let $(\mathcal{Y},\mathcal{L}_{\mathcal{Y}})$ and $(\mathcal{X},\mathcal{L})$ be normal semiample test configurations for $(X,L)$. 
Suppose that there exists a canonical birational morphism $\pi\colon\mathcal{Y}\to\mathcal{X}$ such that $\mathcal{L}_{\mathcal{Y}}=\pi^*\mathcal{L}$ (cf. Definition \ref{nadef}). Then, 
\begin{align*}
\mathrm{Ding}_{(B,T)}(\mathcal{Y},\mathcal{L}_{\mathcal{Y}})=\mathrm{Ding}_{(B,T)}(\mathcal{X},\mathcal{L}).
\end{align*}
\item Let $(\mathcal{X},\mathcal{L})$ be a normal semiample test configuration for $(X,L)$ and $(\mathcal{X}^{(d)},\mathcal{L}^{(d)})$ be the normalized base change of $(\mathcal{X},\mathcal{L})$ via the $d$-th power map of $\mathbb{P}^1$ for any $d\in\mathbb{Z}_{>0}$ (cf.~Notation \ref{rramamra}). Then,
\[
\mathrm{Ding}_{(B,T)}(\mathcal{X}^{(d)},\mathcal{L}^{(d)})=d\,\mathrm{Ding}_{(B,T)}(\mathcal{X},\mathcal{L}).
\]
\item Let $(\mathcal{X},\mathcal{L})$ be a normal semiample test configuration for $(X,L)$. Then, $\mathrm{Ding}_{(B,T)}(\mathcal{X},\mathcal{L})\le\mathrm{DF}_{(B,T)}(\mathcal{X},\mathcal{L})$ and equality holds if and only if $(\mathcal{X},\rho^{-1}_*T_{1,\mathbb{P}^1}+\mathcal{B}+\mathcal{X}_0)$ is sublc where $\mathcal{B}$ is the strict transform of $B\times\mathbb{P}^1$, $\rho:\mathcal{X}\dashrightarrow X_{\mathbb{P}^1}$ is the canonical birational map, and $\mathcal{L}\sim_{\mathbb{P}^1,\mathbb{Q}}-(K_{\mathcal{X}}+\rho^{-1}_*T_{1,\mathbb{P}^1}+\mathcal{B})$ for general $T_1\sim_{\mathbb{Q}}T$.
\end{enumerate}
\end{prop}
\begin{proof}
 It is easy to check that the proofs of \cite[Proposition 2.5 (3)]{Fjt} and \cite[Theorem 3.2]{Fjt2} also work for subpairs (cf.~\cite[\S7.8]{BHJ}).
 Thus the assertions follow from Lemma \ref{dtw2}.
 \end{proof}

We see that the (uniform) Ding-(semi)stability implies the (uniform) K-(semi)stability of log-twisted Fano pairs by Proposition \ref{fjtprop}. 
The main purpose in this subsection is to show the following theorem.
\begin{thm}\label{bhjz}
Let $(X,\Delta,L)$ be a polarized klt pair and $T=-(K_X+\Delta+L)$. Then 
\[
\delta_{(X,\Delta)}(L)\ge1
\]
if and only if the log twisted Fano pair $(X,\Delta,T,L)$ is Ding-semistable. Moreover,
\[
H_{\Delta}^\mathrm{NA}(\mathcal{X},\mathcal{L})\ge\delta_{(X,\Delta)}(L)(I^\mathrm{NA}(\mathcal{X},\mathcal{L})-J^\mathrm{NA}(\mathcal{X},\mathcal{L}))
\]
for any semiample test configuration $(\mathcal{X},\mathcal{L})$ for $(X,L)$.
\end{thm}
  
  To show this, we first deduce that (uniform) Ding-stability implies (uniform) positivity of $\beta$-invariant.
  
  \begin{thm}[cf., {\cite[Theorem 6.12]{Fjt}}]\label{6.12}
Let $(X,B,T,L=-K_X-B-T)$ be an $n$-dimensional log-twisted Fano pair and $\delta\in[0,1)$. If $\mathrm{Ding}_{(B,T)}(\mathcal{X},\mathcal{L})\ge \delta \,J^{\mathrm{NA}}(\mathcal{X},\mathcal{L})$ for any normal semiample test configuration $(\mathcal{X},\mathcal{L})$, then for any prime divisor $F$ over $X$, $\beta_{(X,B,T)}(F)\ge \delta\,j_{(X,B,T)}(F)$. 

In particular, if $(X,B,T,L)$ is Ding-semistable, then $\delta_{(X,B)}(L)\ge1$.
\end{thm}
 
\begin{proof}
This theorem essentially follows from {\cite[Theorem 6.12]{Fjt}}. 
As in the proof of \cite[Theorem 4.1]{Fjt}, take a log resolution $\sigma:Y\to X$ such that $F$ is a smooth divisor on $Y$ and $r_0\in\mathbb{Z}_{>0}$ such that $r_0(K_X+B+T)$ is $\mathbb{Z}$-Cartier. The argument of the proof of [{\it loc.cit}, 4.1] shows the following.
Let $I_{(k,x)}$ be the image of
$$H^0(Y,-kr_0\sigma^*(K_X+B+T)-\lceil xF\rceil)\otimes\mathcal{O}_X(kr_0\sigma^*(K_X+B+T))\to\mathcal{O}_X$$
for any $k\in\mathbb{Z}_{>0}$ and $\mathscr{I}_r:=\sum_{j=0}^{r(e_+-e_-)} I_{(r,re_+-j)}t^j$ for sufficiently large $r\ge r_0$.
Here, we take $e_-\in\mathbb{Z}_{<0}$ and $e_+\in\mathbb{Z}$ such that $e_+>rT_L(F)$.
We set $\rho_r\colon\mathcal{X}_r\to X_{\mathbb{P}^1}$ as the blow up along $\mathscr{I}_r$, $\mathcal{L}_r:=\rho_r^*L_{\mathbb{P}^1}-\rho_r^{-1}\mathscr{I}_r^{\frac{1}{rr_0}}$ as a semiample $\mathbb{Q}$-line bundle and
\[
d_{r,\delta}:=1+E^{\mathrm{NA}}(\mathcal{X}_r,\mathcal{L}_r)+\delta J^{\mathrm{NA}}(\mathcal{X}_r,\mathcal{L}_r).
\]
Then, we have that the pair
\begin{align*}
\left(X\times\mathbb{P}^1,\mathscr{I}_r^{\frac{1}{rr_0}}\cdot\mathcal{O}(B_{\mathbb{P}^1}+T_{r,\mathbb{P}^1})\cdot(t^{d_{r,\delta}})\right)
\end{align*}
is sublc (cf.~\cite[Definition 2.4]{Fjt2}), where $T_{r}\sim_{\mathbb{Q}}T$ is general.
Thanks to \cite[Claim 3]{Fjt}, 
\begin{equation}\label{eq-thanks-fujita}
d_{\infty,\delta}:=\lim_{r\to\infty}d_{r,\delta}=1-\frac{e_+}{r_0}+\delta T_L(F)+\frac{1-\delta}{L^n}\int^{T_L(F)}_{0}\mathrm{vol}(L-xF)dx.
\end{equation}
On the other hand, the pair
\[
\left(Y\times\mathbb{P}^1,\mathcal{O}_{Y\times\mathbb{P}^1}(-\sum_{i=0}^k(1-A_{(X,B)}(F_i))(F_i\times\mathbb{P}^1)+\sigma^{-1}_*B_{\mathbb{P}^1}+\sigma^*T_{r,\mathbb{P}^1})\cdot ( \sigma^{-1}\mathscr{I}_r)^{\frac{1}{rr_0}}\cdot (t)^{d_{r,\delta}}\right)
\]
is sublc
, where $F_0=F$ and $\{F_i\}$ is the set of $\sigma$-exceptional prime divisors. 
Here, by abuse of notations, we denote $\sigma\times\mathrm{id}_{\mathbb{P}^1}$ by $\sigma$. We can choose $T_r$ so general that $A_{(X,B)}(F_i)=A_{(X,B+T_r)}(F_i)$ for any $i$, $\sigma$ is also a log resolution of $(X,B+T_r)$ and $\sigma^{-1}_*T_r=\sigma^*T_r$.
On the other hand, we have
\[
\mathscr{I}_r\cdot \mathcal{O}_{Y\times\mathbb{P}^1}\subset (\mathcal{O}_{Y\times\mathbb{P}^1}(-F)+t)^{re_+}.
\]
Let $\mathcal{D}:=-\sum_{i=0}^k(1-A_{(X,B)}(F_i))(F_i\times\mathbb{P}^1)+\sigma^{-1}_*B_{\mathbb{P}^1}$.
Therefore, the pair
\[
\left(Y\times\mathbb{P}^1, ( \mathcal{O}_{Y\times\mathbb{P}^1}(-F)+t)^{\frac{e_+}{r_0}}\cdot\mathcal{O}_{Y\times\mathbb{P}^1}(\mathcal{D}+\sigma^*T_{r,\mathbb{P}^1})\cdot (t)^{d_{r,\delta}}\right)
\]
is sublc. By \cite[Corollary 2.33]{KoMo}, we can choose $T'\sim_{\mathbb{Q}}T$ so general that for any $r\in\mathbb{Z}_{>0}$
\[
\left(Y\times\mathbb{P}^1, ( \mathcal{O}_{Y\times\mathbb{P}^1}(-F)+t)^{\frac{e_+}{r_0}}\cdot\mathcal{O}_{Y\times\mathbb{P}^1}(\mathcal{D}+\sigma^*T'_{\mathbb{P}^1})\cdot (t)^{d_{r,\delta}}\right)
\]
is also sublc, $\sigma$ is also a log resolution of $(X,B+T')$, $\sigma^*T'_{\mathbb{P}^1}=\sigma^{-1}_*T'_{\mathbb{P}^1}$ and $A_{(X,B+T')}(F_i)=A_{(X,B)}(F_i)$. Hence,
\begin{equation}\label{eq-subpair-too-long}
\left(Y\times\mathbb{P}^1, ( \mathcal{O}_{Y\times\mathbb{P}^1}(-F)+t)^{\frac{e_+}{r_0}}\cdot\mathcal{O}_{Y\times\mathbb{P}^1}(\mathcal{D}+\sigma^*T'_{\mathbb{P}^1})\cdot (t)^{d_{\infty,\delta}}\right)
\end{equation}
is also sublc.
Let $E_F$ be the exceptional divisor of the blow up $\mathcal{Y}\to Y_{\mathbb{P}^1}$ along $F\times\{0\}$. 
Similarly to the proof of \cite[Theorem 4.1]{Fjt}, we have that 
$A_{(X,B+T')}(F)-\frac{e_+}{r_0}-d_{\infty,\delta}+1\ge0$ by calculating the log discrepancy of the pair (\ref{eq-subpair-too-long}) with respect to $E_F$.
Then the assertion follows from (\ref{eq-thanks-fujita}).
\end{proof}

To prove the converse, recall the following explicit formula of the log Ding-invariant by \cite{BHJ}.

 \begin{prop}[{\cite[Proposition 7.29]{BHJ}}]\label{bhjmod}
Let $(X,B,L)$ be a polarized klt pair and $T$ be a $\mathbb{Q}$-line bundle on $X$ such that $T=-(K_X+B+L)$.
Let $(\mathcal{X},\mathcal{L})$ be a semiample test configuration for $(X,L)$ dominating $X_{\mathbb{P}^1} $ such that $(\mathcal{X},\mathcal{B}+\rho^*T_{1,\mathbb{P}^1}+\mathcal{X}_{0,\mathrm{red}})$ is sublc for general $T_1\sim_{\mathbb{Q}}T$, where $\rho:\mathcal{X}\to X_{\mathbb{P}^1}$ is the canonical morphism and $\mathcal{B}=\rho^{-1}_*B_{\mathbb{P}^1}$. Let $\{E_i\}_{i=1}^r$ be the set of irreducible components of $\mathcal{X}_0$ and $G$ be a $\mathbb{Q}$-divisor defined by $\mathrm{Supp}\, G\subset \mathcal{X}_0$ and $G\sim_{\mathbb{Q}}\mathcal{L}-L_{\mathbb{P}^1}$.
Then 
\[
\mathrm{Ding}_{(B,T)}(\mathcal{X},\mathcal{L})=\mathrm{Ding}_{B+T_1}(\mathcal{X},\mathcal{L})=\min_{E_i}(A_{(X,B)}(v_{E_i})+b_{E_i}^{-1}\mathrm{ord}_{E_i}(G)-E^{\mathrm{NA}}(\mathcal{X},\mathcal{L}).
\]
\end{prop}

\begin{proof}
The first equality follows from Lemma \ref{dtw2}. 
Let $D=-\mathcal{L}-K_{\mathcal{X}/\mathbb{P}^1}-\mathcal{B}-\rho^*T_{1,\mathbb{P}^1}$ and $D^{\mathrm{log}}=D-\mathcal{X}_{0,\mathrm{red}}+\mathcal{X}_0$. 
We see by \cite[Proposition 7.29]{BHJ} that 
\begin{align*}
\mathrm{lct}_{(\mathcal{X},\mathcal{B}+\rho^*T_{1,\mathbb{P}^1}+\mathcal{X}_{0,\mathrm{red}}+D^{\mathrm{log}})}(\mathcal{X}_0)=\min_{E_i}(A_{(X,B+T_1)}(v_{E_i})+b_{E_i}^{-1}\mathrm{ord}_{E_i}(G))
\end{align*}
for any $T_1\sim_{\mathbb{Q}}T$ such that $(X,B+T_1)$ is subklt and $(\mathcal{X},\mathcal{B}+\rho^*T_{1,\mathbb{P}^1}+\mathcal{X}_{0,\mathrm{red}})$ is sublc.
Thus, we obtain the assertion if we further choose $T_1$ so general that $\rho^{-1}_*T_{1,\mathbb{P}^1}=\rho^*T_{1,\mathbb{P}^1}$ and $A_{(X,B+T_1)}(v_{E_i})=A_{(X,B)}(v_{E_i})$ for any ${E_i}$.
 %
\end{proof}

 To show Theorem \ref{bhjz}, we remark that the following holds.
 
 \begin{prop}[{\cite[Proposition 2]{LX}}]\label{lxlcmod}
 Let $(X,B)$ be a projective klt pair and $\mathcal{X}$ be a normal test configuration for $X$. Then there exist $d\in \mathbb{Z}_{>0}$ and a normal test configuration $\mathcal{Y}$ for $X$ such that 
 \begin{enumerate}
 \item There exists a canonical morphism $h\colon\mathcal{Y}\to\mathcal{X}^{(d)}$, where $\mathcal{X}^{(d)}$ as Notation \ref{rramamra},
 \item $(\mathcal{Y},\mathcal{B}_{\mathcal{Y}}+\mathcal{Y}_0)$ is lc, where $\mathcal{B}_{\mathcal{Y}}$ is the strict transform of $B\times\mathbb{P}^1$,
 \item $K_\mathcal{Y}+\mathcal{B}_{\mathcal{Y}}+\mathcal{Y}_0$ is $h$-ample.
 \end{enumerate}
  \end{prop}
  
We call such $\mathcal{Y}$ the lc modification of $(\mathcal{X}^{(d)},\mathcal{B}^{(d)})$.
The proof of this proposition is similar to \cite[Proposition 2]{LX} and is left to the reader. Then, we apply these propositions to prove Theorem \ref{bhjz}.
 
\begin{proof}[Proof of Theorem \ref{bhjz}]
The second assertion follows from the first assertion and Proposition \ref{fjtprop}. 
Indeed, we have that for any normal semiample test configuration $(\mathcal{X},\mathcal{L})$ for $(X,L)$,
\begin{align*}
M^{\mathrm{NA}}_{(\Delta,T)}(\mathcal{X},\mathcal{L})&=H^{\mathrm{NA}}_{\Delta}(\mathcal{X},\mathcal{L})-(I^\mathrm{NA}(\mathcal{X},\mathcal{L})-J^\mathrm{NA}(\mathcal{X},\mathcal{L})),\quad\text{and}\\
\mathrm{Ding}_{(\Delta,T)}(\mathcal{X},\mathcal{L})&\le M^{\mathrm{NA}}_{(\Delta,T)}(\mathcal{X},\mathcal{L})
\end{align*}
by Proposition \ref{fjtprop} (3).
If the first assertion holds, then we obtain that $\delta_{(X,\Delta)}(L)\ge 1$ implies
\[
H^{\mathrm{NA}}_{\Delta}(\mathcal{X},\mathcal{L})\ge I^\mathrm{NA}(\mathcal{X},\mathcal{L})-J^\mathrm{NA}(\mathcal{X},\mathcal{L}).
\]
Put $n=\mathrm{dim}\,X$ and we note that for any $r\in\mathbb{Q}_{>0}$, 
\begin{align*}
\delta_{(X,\Delta)}(rL)&=r^{-1}\delta_{(X,\Delta)}(L)\\
    H_{\Delta}^\mathrm{NA}(\mathcal{X},r\mathcal{L})&=r^nH_{\Delta}^\mathrm{NA}(\mathcal{X},\mathcal{L})\\
    I^\mathrm{NA}(\mathcal{X},r\mathcal{L})-J^\mathrm{NA}(\mathcal{X},r\mathcal{L})&=r^{n+1}(I^\mathrm{NA}(\mathcal{X},\mathcal{L})-J^\mathrm{NA}(\mathcal{X},\mathcal{L})).
\end{align*}
Thus, for any $\epsilon>0$ such that $\delta_{(X,\Delta)}(L)-\epsilon\in\mathbb{Q}_{>0}$, we obtain that 
\begin{align*}
    H^{\mathrm{NA}}_{\Delta}(\mathcal{X},\mathcal{L})&=(\delta_{(X,\Delta)}(L)-\epsilon)^{n+1}H^{\mathrm{NA}}_{\Delta}(\mathcal{X},(\delta_{(X,\Delta)}(L)-\epsilon)\mathcal{L})\\
    &\ge(\delta_{(X,\Delta)}(L)-\epsilon)^{n+1}(I^\mathrm{NA}(\mathcal{X},(\delta_{(X,\Delta)}(L)-\epsilon)\mathcal{L})-J^\mathrm{NA}(\mathcal{X},(\delta_{(X,\Delta)}(L)-\epsilon)\mathcal{L}))\\
    &\ge(\delta_{(X,\Delta)}(L)-\epsilon)(I^\mathrm{NA}(\mathcal{X},\mathcal{L})-J^\mathrm{NA}(\mathcal{X},\mathcal{L})).
\end{align*}
Hence, we obtain that $H^{\mathrm{NA}}_{\Delta}(\mathcal{X},\mathcal{L})\ge \delta_{(X,\Delta)}(L)(I^\mathrm{NA}(\mathcal{X},\mathcal{L})-J^\mathrm{NA}(\mathcal{X},\mathcal{L}))$.

Next, we deal with the first assertion.
It suffices to show the opposite direction to Theorem \ref{6.12}.
 By Proposition \ref{fjtprop} (1), we may assume that $\mathcal{X}$ dominates $X_{\mathbb{P}^1}$. 
 Furthermore, by Proposition \ref{fjtprop} (2) and Proposition \ref{lxlcmod}, we may assume that $(\mathcal{X},\Delta_{\mathcal{X}}+\mathcal{X}_{0,\mathrm{red}})$ is lc where $\Delta_{\mathcal{X}}$ is the strict transform of $\Delta\times\mathbb{P}^1$.
 Hence, we may also assume that $(\mathcal{X},\Delta_{\mathcal{X}}+\rho^*T_{1,\mathbb{P}^1}+\mathcal{X}_{0,\mathrm{red}})$ is sublc where $T_1\sim_{\mathbb{Q}}T$ is general.
 Let $G$ be a $\mathbb{Q}$-divisor defined by $G\sim_{\mathbb{Q}}\mathcal{L}-\rho^*L_{\mathbb{P}^1}$ and $\mathrm{Supp}\,G\subset \mathcal{X}_0$.
 Then, we see that
\begin{equation}\label{eq-ding-twist}
\mathrm{Ding}_{(\Delta,T)}(\mathcal{X},\mathcal{L})=\mathrm{Ding}_{(\Delta+T_1)}(\mathcal{X},\mathcal{L})=\min_E(A_{(X,\Delta)}(v_E)+b_E^{-1}\mathrm{ord}_E(G))-E^{\mathrm{NA}}(\mathcal{X},\mathcal{L})
\end{equation}
by Proposition \ref{bhjmod}, where $E$ runs over all irreducible components of $\mathcal{X}_0$.
Take $r\in\mathbb{Z}_{>0}$, we may assume that $r\mathcal{L}$ is a line bundle. 
Then, we set 
\[
F^{\lambda}H^0(X,mrL):=\{s\in H^0(X,mrL)\,|\,t^{-\lambda}\bar{s}\in H^0(\mathcal{X},mr\mathcal{L})\}
\]
for $\lambda\in\mathbb{Z}$ and $m\in\mathbb{Z}_{\ge0}$, where $\bar{s}\in H^0(X_{\mathbb{P}^1},mrL_{\mathbb{P}^1})$ is a $\mathbb{G}_m$-equivariant section corresponding to $s$.
We have that
$$F^{\lambda}H^0(X,mL)=\bigcap_E\{s\in H^0(X,mL)|v_E(s)+m b_E^{-1}\mathrm{ord}_E(G)\ge \lambda \}$$ 
for sufficiently divisible $m$ and all $\lambda\in\mathbb{Z}$ by \cite[Lemma 5.17]{BHJ}. 
Set for any $\lambda\in\mathbb{R}$ $$R^{(\lambda)}=\bigoplus_{m\in\mathbb{Z}_{\ge0}}F^{\lceil mr\lambda\rceil}H^0(X,mrL).$$
It is well-known that $R^{(\lambda)}$ is a graded algebra.
On the other hand, set $\lambda_{(\max)}=\max_Eb_E^{-1}\mathrm{ord}_E(G)$ and $\lambda_{(\mathrm{min})}=\min_Eb_E^{-1}\mathrm{ord}_E(G)$, where $E$ runs over all irreducible components of $\mathcal{X}_0$.

 Next, assume that $\delta_{(X,\Delta)}(L)\ge1$. 
 To see $\mathrm{Ding}_{(\Delta,T)}(\mathcal{X},\mathcal{L})\ge0$, it suffices to show 
 \begin{equation}
 A_{(X,\Delta)}(v_E)+b_E^{-1}\mathrm{ord}_E(G)-E^{\mathrm{NA}}(\mathcal{X},\mathcal{L})\ge0 \label{eq--delta--enough--to--show}
 \end{equation}
     for any irreducible component $E$ of $\mathcal{X}_0$ by (\ref{eq-ding-twist}).
 Note that
 \[
 E^{\mathrm{NA}}(\mathcal{X},\mathcal{L})= -(L^n)^{-1}\int_{\mathbb{R}}\lambda\frac{d}{d\lambda}\mathrm{vol}(R^{(\lambda)})d\lambda
 \]
  by \cite[Th.~5.3, Prop.~5.9 and Lem.~7.3]{BHJ}.
  We note that the support of the distribution $\frac{d}{d\lambda}\mathrm{vol}(R^{(\lambda)})$ coincides with $[\lambda_{(\min)},\lambda_{(\max)}]$ by \cite[Theorem 5.16]{BHJ}. 
 Since $\delta_{(X,\Delta)}(L)\ge1$, we obtain that $A_{(X,\Delta)}(v_E)\ge S_L(v_E)$ and 
\begin{align*}
&A_{(X,\Delta)}(v_E)+b_E^{-1}\mathrm{ord}_E(G)-E^{\mathrm{NA}}(\mathcal{X},\mathcal{L})\ge S_L(v_E)+b_E^{-1}\mathrm{ord}_E(G)+(L^n)^{-1}\int_{\mathbb{R}}\lambda\frac{d}{d\lambda}\mathrm{vol}(R^{(\lambda)})d\lambda\\
&\ge\int_{x\ge0}(L^n)^{-1}\mathrm{vol}(L-xv_E)dx+b_E^{-1}\mathrm{ord}_E(G)-\lambda_{(\mathrm{min})}-\int_{\lambda\ge\lambda_{(\mathrm{min})}}(L^n)^{-1}\mathrm{vol}(R^{(\lambda)})d\lambda\\
&\ge \int_{x\ge0}(L^n)^{-1}\mathrm{vol}(L-xv_E)dx-\int_{\lambda\ge b_E^{-1}\mathrm{ord}_E(G)}(L^n)^{-1}\mathrm{vol}(R^{(\lambda)})d\lambda\\
&= \int_{x\ge0}(L^n)^{-1}\mathrm{vol}(L-xv_E)dx-\int_{x\ge0}(L^n)^{-1}\mathrm{vol}(R^{(x+ b_E^{-1}\mathrm{ord}_E(G))})dx
\end{align*}
for any irreducible component $E$ of $\mathcal{X}_0$ by $(L^n)^{-1}\mathrm{vol}(R^{(\lambda)})\le1$. For sufficiently divisible $m$ and all $x\in\mathbb{Z}$, we have
\begin{align*}
F^{(x+ mb_E^{-1}\mathrm{ord}_E(G))}H^0(X,mL)&=\bigcap_F\{s\in H^0(X,mL)|v_F(s)+m b_F^{-1}\mathrm{ord}_F(G)\ge x+m b_E^{-1}\mathrm{ord}_E(G) \}\\
&\subset \{s\in H^0(X,mL)|v_E(s)\ge x \}
\end{align*}
where $F$ runs over all irreducible components of $\mathcal{X}_0$. Thus, we have
\[
\int_{x\ge0}(L^n)^{-1}\mathrm{vol}(L-xv_E)dx-\int_{x\ge0}(L^n)^{-1}\mathrm{vol}(R^{(x+ b_E^{-1}\mathrm{ord}_E(G))})dx\ge0.
\]
Therefore, we obtain (\ref{eq--delta--enough--to--show}) and that $(X,\Delta,T,L)$ is log-twisted Ding-semistable.
We are done.
\end{proof}



  \subsection{Log-twisted K-semistable Fano pair and destabilizing boundary}\label{sec-stp1}
  
  Finally, we prove that Ding-stability and K-stability coincide for log-twisted Fano pairs when $T$ is semiample. 
  The following is the main result in this subsection.

  \begin{thm}\label{appendix}
Let $(X,B,T,L=-K_X-B-T)$ be a K-semistable but not K-stable log-twisted Fano pair such that $T$ is semiample.
Then there exists a $\mathbb{Q}$-Cartier $\mathbb{Q}$-divisor $D\in |L|_{\mathbb{Q}}$ such that $(X,B+\epsilon D,T,(1-\epsilon) L)$ is a K-unstable log-twisted Fano pair for any sufficiently small rational number $0<\epsilon<1$.
\end{thm}


We note that the logarithmic case of Theorem \ref{appendix} had been shown in \cite[Proposition 3.3]{BLX}.
To show Theorem \ref{appendix}, we first generalize the results of \cite{LX}, \cite{Fjt} and \cite[Theorem 3.12]{BLZ}.

\begin{thm}[cf., {\cite[Theorem 2]{LX}}, {\cite[Theorem 6.7]{Fjt}}]\label{thm-fjt-6.7-lc-modif}
Let $(X,B,T,L=-K_X-B-T)$ be a log-twisted Fano pair such that $(X,B)$ is klt, $T$ is semiample and $(\mathcal{X},\mathcal{L})$ be a normal ample test configuration. Then there exist $d\in\mathbb{Z}_{>0}$ and an ample test configuration $(\mathcal{X}^\mathrm{lc},\mathcal{L}^\mathrm{lc})$ dominating $(\mathcal{X}^{(d)},\mathcal{L}^{(d)})$ via $\pi$ that is the normalization of the base change of the $d$-th power map of $\mathbb{P}^1$ such that:
\begin{itemize}
\item[(1)] 
$(\mathcal{X}^\mathrm{lc},\mathcal{B}^\mathrm{lc}+(\rho^\mathrm{lc})_*^{-1}T_{1,\mathbb{P}^1}+\mathcal{X}^\mathrm{lc}_0)$ is lc for any general effective divisor $T_1\sim_{\mathbb{Q}}T$, where $\mathcal{B}^{\mathrm{lc}}$ is the strict transform of $B_{\mathbb{P}^1}$ and $\rho^{\mathrm{lc}}\colon\mathcal{X}^{\mathrm{lc}}\dashrightarrow X_{\mathbb{P}^1}$ is the canonical birational map.
\item[(2)] $\mathrm{DF}_{(B,T)}(\mathcal{X}^\mathrm{lc},\mathcal{L}^\mathrm{lc})\le d\, \mathrm{DF}_{(B,T)}(\mathcal{X},\mathcal{L})$ and
\[
\mathrm{Ding}_{(B,T)}(\mathcal{X}^\mathrm{lc},\mathcal{L}^\mathrm{lc})-\delta\, J^{\mathrm{NA}}(\mathcal{X}^\mathrm{lc},\mathcal{L}^\mathrm{lc})\le d\, (\mathrm{Ding}_{(B,T)}(\mathcal{X},\mathcal{L})-\delta\, J^{\mathrm{NA}}(\mathcal{X},\mathcal{L})),
\]
 for any $\delta\in [0,1)$.
\end{itemize}
\end{thm}
 
\begin{proof}
By the semistable reduction theorem \cite{KKMS}, there exist a log smooth pair $(\mathcal{Y},\mathrm{Ex}(g)+g^{-1}_*(\mathcal{B}^{(d)})+\mathcal{Y}_0)$ and a birational morphism $g:\mathcal{Y}\to \mathcal{X}^{(d)}$ such that the central fiber $\mathcal{Y}_0$ is reduced (cf., \cite[Theorem 7.17]{KoMo}) for some $d\in\mathbb{Z}_{>0}$. 
We may assume that $(\mathcal{Y},\mathrm{Ex}(g)+g^{-1}_*(\mathcal{B}^{(d)})+\mathcal{Y}_0)$ admits a $\mathbb{G}_m$-action such that $g$ is $\mathbb{G}_m$-equivariant and there exists a morphism $\rho_{\mathcal{Y}}:\mathcal{Y}\to X_{\mathbb{P}^1}$ such that $g$ and $\rho_{\mathcal{Y}}$ coincide over $\mathbb{P}^1\setminus\{0\}$.
Then take a general effective divisor $T_1\sim_{\mathbb{Q}}T$ such that $(\mathcal{Y},g_*^{-1}(\mathcal{B}^{(d)}+\mathcal{X}^{(d)}_0)+\rho_{\mathcal{Y}}^*T_{1,\mathbb{P}^1}+E)$ is log smooth and dlt, $(\rho_{\mathcal{Y}})^{-1}_*T_{1,\mathbb{P}^1}=\rho_{\mathcal{Y}}^*T_{1,\mathbb{P}^1}$ and 
\begin{align*}
d\, \mathrm{DF}_{(B,T)}(\mathcal{X},\mathcal{L})=\mathrm{DF}_{(B,T)}(\mathcal{X}^{(d)},\mathcal{L}^{(d)})&=\mathrm{DF}_{B+T_1}(\mathcal{X}^{(d)},\mathcal{L}^{(d)})=d\,\mathrm{DF}_{B+T_1}(\mathcal{X},\mathcal{L})\\
d\, \mathrm{Ding}_{(B,T)}(\mathcal{X},\mathcal{L})=\mathrm{Ding}_{(B,T)}(\mathcal{X}^{(d)},\mathcal{L}^{(d)})&=\mathrm{Ding}_{B+T_1}(\mathcal{X}^{(d)},\mathcal{L}^{(d)})=d\,\mathrm{Ding}_{B+T_1}(\mathcal{X},\mathcal{L}).
\end{align*}
We take the lc modification $h:\mathcal{X}^\mathrm{lc}\to (\mathcal{X}^{(d)},\mathcal{B}^{(d)}+(\rho^{(d)})^{-1}_*T_{1,\mathbb{P}^1}+\mathcal{X}^{(d)}_0)$ by Proposition \ref{lxlcmod}, where $\rho^{(d)}\colon\mathcal{X}^{(d)}\dashrightarrow X_{\mathbb{P}^1}$ is the canonical birational map. 
Let $\pi:\mathcal{Y}\dashrightarrow\mathcal{X}^{\mathrm{lc}}$ be the canonical birational contraction and $F$ be the summation of $h$-exceptional prime divisors such that
\[
-F\sim_{\mathcal{X}^{(d)},\mathbb{Q}}K_{\mathcal{X}^{\mathrm{lc}}}+\pi_*(g_*^{-1}(\mathcal{B}^{(d)}+\mathcal{X}^{(d)}_0)+\rho_{\mathcal{Y}}^*T_{1,\mathbb{P}^1}+E).
\]
Let $\mathcal{B}^\mathrm{lc}=\pi_*(g_*^{-1}\mathcal{B}^{(d)}+E)$ and $\mathcal{L}^\mathrm{lc}$=$h^*\mathcal{L}^{(d)}-\epsilon F$ for sufficiently small $\epsilon>0$.
Then, it follows from the proof of \cite[Theorem 6.7]{Fjt} that
\begin{align*}
\mathrm{DF}_{B+T_1}(\mathcal{X}^\mathrm{lc},\mathcal{L}^\mathrm{lc})&\le \mathrm{DF}_{B+T_1}(\mathcal{X}^{(d)},\mathcal{L}^{(d)})\quad\text{and}\\
\mathrm{Ding}_{B+T_1}(\mathcal{X}^\mathrm{lc},\mathcal{L}^\mathrm{lc})-\delta\, J^{\mathrm{NA}}(\mathcal{X}^\mathrm{lc},\mathcal{L}^\mathrm{lc})&\le \mathrm{Ding}_{B+T_1}(\mathcal{X}^{(d)},\mathcal{L}^{(d)})-\delta\, J^{\mathrm{NA}}(\mathcal{X}^{(d)},\mathcal{L}^{(d)})
\end{align*}
for any $\delta\in[0,1)$.

Fix a general effective $\mathbb{Q}$-divisor $T_1\sim_{\mathbb{Q}}T$ and let $\mathcal{X}^{\mathrm{lc}}$ be a test configuration constructed in the previous paragraph for $T_1$.
Here, we assert that the construction of $\mathcal{X}^{\mathrm{lc}}$ is independent of the choice of general $T_1$.
Indeed, let $T_2\sim_{\mathbb{Q}}T_1$ be a general effective $\mathbb{Q}$-divisor such that $(\mathcal{Y},g_*^{-1}(\mathcal{B}^{(d)}+\mathcal{X}^{(d)}_0)+\rho_{\mathcal{Y}}^*T_{2,\mathbb{P}^1}+E)$ is log smooth and dlt and $(\rho_{\mathcal{Y}})^{-1}_*T_{2,\mathbb{P}^1}=\rho_{\mathcal{Y}}^*T_{2,\mathbb{P}^1}$.
We note that $\mathcal{X}^\mathrm{lc}$ is the lc model of $(\mathcal{Y},g_*^{-1}(\mathcal{B}^{(d)}+\mathcal{X}^{(d)}_0)+\rho_{\mathcal{Y}}^*T_{1,\mathbb{P}^1}+E)$ by \cite[Lemma 2.1]{OX} and the construction. 
By \cite[Lemma 3.6.9]{BCHM}, the lc models of $(\mathcal{Y},g_*^{-1}(\mathcal{B}^{(d)}+\mathcal{X}^{(d)}_0)+\rho_{\mathcal{Y}}^*T_{i,\mathbb{P}^1}+E)$ over $\mathcal{X}^{(d)}$ are the same for $i=1,2$.
Thus, $\mathcal{X}^{\mathrm{lc}}$ is independent of the choice of general $T_1$ and 
by choosing $T_1$ general enough, we have that
\begin{align*}
\mathrm{Ding}_{(B,T)}(\mathcal{X}^\mathrm{lc},\mathcal{L}^\mathrm{lc})&=\mathrm{Ding}_{B+T_1}(\mathcal{X}^\mathrm{lc},\mathcal{L}^\mathrm{lc}) \\
\mathrm{DF}_{(B,T)}(\mathcal{X}^\mathrm{lc},\mathcal{L}^\mathrm{lc})&=\mathrm{DF}_{B+T_1}(\mathcal{X}^\mathrm{lc},\mathcal{L}^\mathrm{lc}).
\end{align*}
Hence the assertion follows.
\end{proof}
 
\begin{thm}[cf., {\cite[Theorem 3]{LX}}, {\cite[Theorem 6.8]{Fjt}}]\label{6.8}
Let $(X,B,T,L=-K_X-B-T)$ be a log-twisted Fano pair such that $T$ is semiample, and $(\mathcal{X},\mathcal{L})$ be a normal ample test configuration such that $(\mathcal{X},\mathcal{B}+\rho^{-1}_*T_{0,\mathbb{P}^1}+\mathcal{X}_0)$ is lc, where $\rho\colon\mathcal{X}\dashrightarrow X_{\mathbb{P}^1}$ is canonical, $\mathcal{B}=\rho^{-1}_*B_{\mathbb{P}^1}$ and $T_0\sim_{\mathbb{Q}}T$ is a general effective $\mathbb{Q}$-divisor. 

Then there exists a normal and ample test configuration $(\mathcal{X}^{\mathrm{ac}},\mathcal{L}^\mathrm{ac})$ for $(X,L)$ such that $(\mathcal{X}^\mathrm{ac},\mathcal{B}^\mathrm{ac}+(\rho^\mathrm{ac})^{-1}_*T_{1,\mathbb{P}^1}+\mathcal{X}^\mathrm{ac}_0)$ is lc, $-(K_{\mathcal{X}^\mathrm{ac}}+\mathcal{B}^\mathrm{ac}+(\rho^\mathrm{ac})^{-1}_*T_{1,\mathbb{P}^1})\sim_{\mathbb{P}^1,\mathbb{Q}}\mathcal{L}^\mathrm{ac}$, $\mathrm{DF}_{(B,T)}(\mathcal{X}^\mathrm{ac},\mathcal{L}^\mathrm{ac})\le \mathrm{DF}_{(B,T)}(\mathcal{X},\mathcal{L})$ and 
\[
\mathrm{Ding}_{(B,T)}(\mathcal{X}^\mathrm{ac},\mathcal{L}^\mathrm{ac})-\delta\, J^{\mathrm{NA}}(\mathcal{X}^\mathrm{ac},\mathcal{L}^\mathrm{ac})\le \mathrm{Ding}_{(B,T)}(\mathcal{X},\mathcal{L})-\delta\, J^{\mathrm{NA}}(\mathcal{X},\mathcal{L}),
\]
for any $\delta\in[0,1)$ and any general effective $\mathbb{Q}$-divisor $T_1\sim_{\mathbb{Q}}T$, where $\rho^{\mathrm{ac}}\colon\mathcal{X}^{\mathrm{ac}}\dashrightarrow X_{\mathbb{P}^1}$ is canonical and $\mathcal{B}^{\mathrm{ac}}:=(\rho^{\mathrm{ac}})^{-1}_*B_{\mathbb{P}^1}$.
\end{thm}
 
\begin{proof}
Fix a general effective $\mathbb{Q}$-divisor $T_1\sim_{\mathbb{Q}}T$ such that $(\mathcal{X},\mathcal{B}+\rho_*^{-1}T_{1,\mathbb{P}^1}+\mathcal{X}_0)$ is lc.
Hence $(\mathcal{X},\mathcal{B}+\rho_*^{-1}T_{1,\mathbb{P}^1})$ is klt. 
Here, we apply \cite[Corollary 1.4.3]{BCHM} to obtain a $\mathbb{G}_m$-equivariant small $\mathbb{Q}$-factorial modification $\sigma:(\mathcal{X}^0,\mathcal{B}^0+(\rho^0)_*^{-1}T_{1,\mathbb{P}^1})\to(\mathcal{X},\mathcal{B}+\rho_*^{-1}T_{1,\mathbb{P}^1})$. 
Here, we set $\rho^0:=\rho\circ\sigma$ and $\mathcal{B}^0:=\sigma^{-1}_*\mathcal{B}$.
Let $\mathcal{L}^0=\sigma^*\mathcal{L}$ and for sufficiently large $l>1$, set $\mathcal{H}^0$ as a sufficiently general effective $\mathbb{Q}$-divisor such that
\[
\mathcal{H}^0\sim_{\mathbb{Q}}\mathcal{L}^0-\frac{1}{l+1}(\mathcal{L}^0+K_{\mathcal{X}^0}+\mathcal{B}^0+(\rho^0)_*^{-1}T_{1,\mathbb{P}^1}).
\]
Then, we run a $K_{\mathcal{X}^0}+\mathcal{B}^0+(\rho^0)_*^{-1}T_{1,\mathbb{P}^1}+\mathcal{H}^0$-MMP over $\mathbb{P}^1$ with scaling of $\mathcal{H}^0$ as in the proof of \cite[Theorem 3.2]{Fjt} and obtain a certain good minimal model $(\mathcal{X}^k,\mathcal{B}^k+(\rho^k)_*^{-1}T_{1,\mathbb{P}^1}+\mathcal{H}^k)$, where $\rho^k\colon\mathcal{X}^k\dashrightarrow X_{\mathbb{P}^1}$ is canonical and $\mathcal{B}^k$ is the strict transform of $\mathcal{B}^0$. 
By \cite[Lemma 2]{LX}, we also have that $-(K_{\mathcal{X}^k}+\mathcal{B}^k+(\rho^k)_*^{-1}T_{1,\mathbb{P}^1})$ is semiample over $\mathbb{P}^1$ and set $\mathcal{X}^\mathrm{ac}:=\mathbf{Proj}_C(\oplus_{l\ge0}\pi^k_*(\mathcal{O}_{\mathcal{X}^k}(-lr_k(K_{\mathcal{X}^k}+\mathcal{B}^k+(\rho^k)_*^{-1}T_{1,\mathbb{P}^1})))$, where $r_k\in\mathbb{Z}_{>0}$ satisfies that $r_k(K_{\mathcal{X}^k}+\mathcal{B}^k+(\rho^k)_*^{-1}T_{1,\mathbb{P}^1})$ is Cartier and $\pi^k\colon\mathcal{X}^k\to\mathbb{P}^1$ is the canonical morphism.
Let $\rho^{\mathrm{ac}}\colon\mathcal{X}^{\mathrm{ac}}\dashrightarrow X_{\mathbb{P}^1}$ be the canonical birational map and $\mathcal{B}^{\mathrm{ac}}$ is the strict transform of $\mathcal{B}^k$.
Then the proof of \cite[Theorem 6.8]{Fjt} shows that if we set $\mathcal{L}^\mathrm{ac}:=-(K_{\mathcal{X}^\mathrm{ac}}+\mathcal{B}^\mathrm{ac}+(\rho^\mathrm{ac})^{-1}_*T_{1,\mathbb{P}^1})$, then 
$(\mathcal{X}^\mathrm{ac},\mathcal{B}^\mathrm{ac}+(\rho^\mathrm{ac})^{-1}_*T_{1,\mathbb{P}^1}+\mathcal{X}^\mathrm{ac}_0)$ is lc, $\mathrm{DF}_{B+T_1}(\mathcal{X}^\mathrm{ac},\mathcal{L}^\mathrm{ac})\le \mathrm{DF}_{B+T_1}(\mathcal{X},\mathcal{L})$ and 
\[
\mathrm{Ding}_{B+T_1}(\mathcal{X}^\mathrm{ac},\mathcal{L}^\mathrm{ac})-\delta\, J^{\mathrm{NA}}(\mathcal{X}^\mathrm{ac},\mathcal{L}^\mathrm{ac})\le \mathrm{Ding}_{B+T_1}(\mathcal{X},\mathcal{L})-\delta\, J^{\mathrm{NA}}(\mathcal{X},\mathcal{L})
\]
for any $\delta\in[0,1)$.

Next, take a sufficiently general effective $\mathbb{Q}$-divisor $T_2\sim_{\mathbb{Q}}T$ such that
\begin{align}
\mathrm{Ding}_{(B,T)}(\mathcal{X}^\mathrm{ac},\mathcal{L}^\mathrm{ac})&=\mathrm{Ding}_{B+T_2}(\mathcal{X}^\mathrm{ac},\mathcal{L}^\mathrm{ac}) ,\quad\mathrm{DF}_{(B,T)}(\mathcal{X}^\mathrm{ac},\mathcal{L}^\mathrm{ac})=\mathrm{DF}_{B+T_2}(\mathcal{X}^\mathrm{ac},\mathcal{L}^\mathrm{ac})\label{eq--four--equation s}\\
\mathrm{Ding}_{(B,T)}(\mathcal{X},\mathcal{L})&=\mathrm{Ding}_{B+T_2}(\mathcal{X},\mathcal{L}),\quad\mathrm{DF}_{(B,T)}(\mathcal{X},\mathcal{L})=\mathrm{DF}_{B+T_2}(\mathcal{X},\mathcal{L}).\nonumber
\end{align}
Then, we assert that the properties of $\mathcal{X}^{\mathrm{ac}}$, which we stated in the last paragraph, are also satisfied for $T_2$.
First, we check that the $K_{\mathcal{X}^0}+\mathcal{B}^0+(\rho^0)_*^{-1}T_{1,\mathbb{P}^1}+\mathcal{H}^0$-MMP over $\mathbb{P}^1$ is also a $K_{\mathcal{X}^0}+\mathcal{B}^0+(\rho^0)_*^{-1}T_{2,\mathbb{P}^1}+\mathcal{H}^0$-MMP.
Indeed, for any sufficiently general effective $\mathbb{Q}$-divisor $T_2\sim_{\mathbb{Q}}T$, 
we see that $(\rho^0)_*^{-1}T_{2,\mathbb{P}^1}\sim_{\mathbb{Q}}(\rho^0)_*^{-1}T_{1,\mathbb{P}^1}$ as \cite[Lemma 2.12]{BLZ}.
Thus, we may assume that $K_{\mathcal{X}^0}+\mathcal{B}^0+(\rho^0)_*^{-1}T_{i,\mathbb{P}^1}+\mathcal{H}^0$-minimal model programs over $\mathbb{P}^1$ with scaling of $\mathcal{H}^0$ are the same for $i=1,2$
.
Furthermore, we see that
\[
K_{\mathcal{X}^k}+\mathcal{B}^k+(\rho^k)_*^{-1}T_{1,\mathbb{P}^1}\sim_{\mathbb{Q}}K_{\mathcal{X}^k}+\mathcal{B}^k+(\rho^k)_*^{-1}T_{2,\mathbb{P}^1}.
\]
This means that $\mathcal{X}^{\mathrm{ac}}$ constructed in the previous paragraph for $T_1$ can be also constructed in the same way for $T_2$.
Hence, $(\mathcal{X}^\mathrm{ac},\mathcal{B}^\mathrm{ac}+(\rho^\mathrm{ac})^{-1}_*T_{2,\mathbb{P}^1}+\mathcal{X}^\mathrm{ac}_0)$ is lc, $-(K_{\mathcal{X}^\mathrm{ac}}+\mathcal{B}^\mathrm{ac}+(\rho^\mathrm{ac})^{-1}_*T_{2,\mathbb{P}^1})\sim_{\mathbb{P}^1,\mathbb{Q}}\mathcal{L}^\mathrm{ac}$, $\mathrm{DF}_{B+T_2}(\mathcal{X}^\mathrm{ac},\mathcal{L}^\mathrm{ac})\le \mathrm{DF}_{B+T_2}(\mathcal{X},\mathcal{L})$ and 
\[
\mathrm{Ding}_{B+T_2}(\mathcal{X}^\mathrm{ac},\mathcal{L}^\mathrm{ac})-\delta\, J^{\mathrm{NA}}(\mathcal{X}^\mathrm{ac},\mathcal{L}^\mathrm{ac})\le \mathrm{Ding}_{B+T_2}(\mathcal{X},\mathcal{L})-\delta\, J^{\mathrm{NA}}(\mathcal{X},\mathcal{L}),
\]
for any $\delta\in[0,1)$.
Combining these properties with (\ref{eq--four--equation s}), we complete the proof.
\end{proof}
 
\begin{thm}[cf., {\cite[Theorem 4]{LX}}, {\cite[Theorem 6.9]{Fjt}}]
Let $(X,B,T,L=-K_X-B-T)$ be a log-twisted Fano pair such that $T$ is semiample and $(\mathcal{X},\mathcal{L})$ be a normal ample test configuration such that $(\mathcal{X},\mathcal{B}+\rho^{-1}_*T_{0,\mathbb{P}^1}+\mathcal{X}_0)$ is lc and $-(K_{\mathcal{X}}+\mathcal{B}+\rho^{-1}_*T_{0,\mathbb{P}^1})\sim_{\mathbb{Q},\mathbb{P}^1}\mathcal{L}$ for any general effective $\mathbb{Q}$-divisor $T_0\sim_{\mathbb{Q}}T$, where $\rho\colon\mathcal{X}\dashrightarrow X_{\mathbb{P}^1}$ is the canonical birational map and $\mathcal{B}=\rho^{-1}_*B_{\mathbb{P}^1}$. 

Then there exist $d\in\mathbb{Z}_{>0}$ and a log-twisted special test configuration $(\mathcal{X}^\mathrm{s},\mathcal{L}^\mathrm{s}=-(K_{\mathcal{X}^\mathrm{s}}+\mathcal{B}^\mathrm{s}+(\rho^\mathrm{s})^{-1}_*T_{1,\mathbb{P}^1}))$ for $(X,B,T,L)$ such that 
\[
\mathrm{Ding}_{(B,T)}(\mathcal{X}^\mathrm{s},\mathcal{L}^\mathrm{s})-\delta\, J^{\mathrm{NA}}(\mathcal{X}^\mathrm{s},\mathcal{L}^\mathrm{s})\le d\, (\mathrm{Ding}_{(B,T)}(\mathcal{X},\mathcal{L})-\delta\, J^{\mathrm{NA}}(\mathcal{X},\mathcal{L}))
\]
for any $\delta\in[0,1)$ and general effective $\mathbb{Q}$-divisor $T_1\sim_{\mathbb{Q}}T$, where $\rho^{\mathrm{s}}\colon\mathcal{X}^\mathrm{s}\dashrightarrow X_{\mathbb{P}^1}$ is the canonical birational map and $\mathcal{B}^{\mathrm{s}}=(\rho^\mathrm{s})^{-1}_*B_{\mathbb{P}^1}$.

\end{thm}
 
\begin{proof}
Note that \begin{align*}
    \mathrm{Ding}_{(B,T)}(\mathcal{X}^\mathrm{s},\mathcal{L}^\mathrm{s})&= \mathrm{DF}_{(B,T)}(\mathcal{X}^\mathrm{s},\mathcal{L}^\mathrm{s}),\quad\text{and}\\
    \mathrm{DF}_{(B,T)}(\mathcal{X},\mathcal{L})&=\mathrm{Ding}_{(B,T)}(\mathcal{X},\mathcal{L}) 
\end{align*}
hold now due to Proposition \ref{fjtprop} (3).
Thus, it suffices to consider the Donaldson-Futaki invariant instead of the Ding invariant.
Recall the proof of \cite[Theorems 6.9 and 6.10]{Fjt}. 
We first take a $\mathbb{G}_m$-invariant semistable reduction $\mathcal{Y}$ of $(\mathcal{X},\mathcal{B}+\mathcal{X}_0)$. 
More precisely, there exist $d\in\mathbb{Z}_{>0}$ and a $\mathbb{G}_m$-equivariant birational morphism $g:\mathcal{Y}\to\mathcal{X}^{(d)}$ such that $\mathcal{Y}$ is smooth, $\mathcal{Y}_0+\mathrm{Ex}(g)+g^{-1}_*\mathcal{B}^{(d)}$ is snc and $\mathcal{Y}_0$ is reduced. We may assume that there exists a $\mathbb{G}_m$-equivariant birational morphism $\rho_{\mathcal{Y}}:\mathcal{Y}\to X_{\mathbb{P}^1}$ such that $\rho_{\mathcal{Y}}$ coincides with $g$ over $\mathbb{P}^1\setminus\{0\}$. 
Fix a general effective divisor $T_1\sim_{\mathbb{Q}}T$ such that
\[
\mathrm{DF}_{B+T_1}(\mathcal{X},\mathcal{L})=\mathrm{DF}_{(B,T)}(\mathcal{X},\mathcal{L}),
\]
$\mathcal{Y}_0+\mathrm{Ex}(g)+g^{-1}_*\mathcal{B}^{(d)}+\rho_{\mathcal{Y}}^*T_{1,\mathbb{P}^1}$ is also snc and the support of $\rho_{\mathcal{Y}}^*T_{1,\mathbb{P}^1}$ contains no component of $\mathcal{Y}_0$. Then, we can take a special test configuration $(\mathcal{X}^{\mathrm{s}},\mathcal{L}^{\mathrm{s}})$ for the log Fano pair $(X,B+T_1)$ as \cite[Theorem 6.10]{Fjt} by running some MMP of $\mathcal{Y}$.
In particular, the birational map $\mathcal{Y}\dashrightarrow\mathcal{X}^{\mathrm{s}}$ is a birational contraction.
This means that $T_1$ does not contain the center of $v_{\mathcal{X}^{\mathrm{s}}_0}$ and we can conclude that 
\[
\mathrm{DF}_{(B,T)}(\mathcal{X}^{\mathrm{s}},\mathcal{L}^{\mathrm{s}})=\mathrm{DF}_{B+T_1}(\mathcal{X}^{\mathrm{s}},\mathcal{L}^{\mathrm{s}})
\]
by Lemma \ref{dtw}. 
On the other hand, it follows from the proof of \cite[Theorem 6.9]{Fjt} that
\[
\mathrm{DF}_{B+T_1}(\mathcal{X}^\mathrm{s},\mathcal{L}^\mathrm{s})-\delta\, J^{\mathrm{NA}}(\mathcal{X}^\mathrm{s},\mathcal{L}^\mathrm{s})\le d\, (\mathrm{DF}_{B+T_1}(\mathcal{X},\mathcal{L})-\delta\, J^{\mathrm{NA}}(\mathcal{X},\mathcal{L})).
\]
Thus the assertion holds.
\end{proof}
 
\begin{cor}[cf., {\cite[Corollary 1]{LX}}, {\cite[Corollary 6.11]{Fjt}}]\label{6.11}
Let $(X,B,T,L=-K_X-B-T)$ be a log-twisted Fano pair such that 
$T$ is semiample. Fix $\delta\in [0,1)$. Then, it holds that $\mathrm{Ding}_{(B,T)}(\mathcal{X},\mathcal{L})\ge \delta \,J^{\mathrm{NA}}(\mathcal{X},\mathcal{L})$ for any semiample test configuration if and only if the inequality holds for any log-twisted special test configuration 
for $(X,B,T,L)$.

Furthermore, $(X,B,T,L)$ is log-twisted (uniformly) K-(semi)stable if and only if $(X,B,T,L)$ is log-twisted (uniformly) Ding-(semi)stable.
 
\end{cor}

 On the other hand, we obtain the following.
 
 
 
\begin{thm}[cf., {\cite[Theorem 6.13]{Fjt}}]\label{6.13}
Let $(X,B,T,L=-K_X-B-T)$ be a log-twisted Fano pair such that $T$ is semiample and let $(\mathcal{X},\mathcal{L}:=-(K_{\mathcal{X}}+\mathcal{B}+\rho^{-1}_*T_{1,\mathbb{P}^1}))$ be a normal ample test configuration for $(X,L)$ with $\mathcal{X}_0$ irreducible and reduced, 
where $T_{1}\in|T|_{\mathbb{Q}}$ is general and $\rho\colon\mathcal{X}\dashrightarrow X_{\mathbb{P}^1}$ is the canonical birational map.
Then the divisorial valuation $v_{\mathcal{X}_0}$ on $X$ satisfies that
\begin{align*}
\mathrm{DF}_{(B,T)}(\mathcal{X},\mathcal{L})&=\frac{\beta_{(X,B,T)}(v_{\mathcal{X}_0})}{\mathrm{vol}(L)},\\
J^\mathrm{NA}(\mathcal{X},\mathcal{L})&=\frac{j_{L}(v_{\mathcal{X}_0})}{\mathrm{vol}(L)}.
\end{align*}
\end{thm}
 
\begin{proof}
For general effective $\mathbb{Q}$-divisor $T_1\sim_{\mathbb{Q}}T$, we have
\begin{align*}
\beta_{(X,B,T)}(v_{\mathcal{X}_0})&=\beta_{(X,B)}(v_{\mathcal{X}_0}),\\
\mathrm{DF}_{(B,T)}(\mathcal{X},\mathcal{L})&=\mathrm{DF}_{(B+T_1)}(\mathcal{X},\mathcal{L}).
\end{align*}
Therefore, this theorem immediately follows from \cite[Theorem 6.13]{Fjt}.
\end{proof}

Thus, we obtain the following.
  
\begin{thm}[cf., {\cite[Theorem 6.5]{Fjt}}]\label{6.5}
Let $\delta\in[0,1)$ and $(X,B,T,L=-K_X-B-T)$ be a log-twisted Fano pair such that $T$ is semiample. Then the following are equivalent,
\begin{enumerate}[(i)]
\item $\mathrm{DF}_{(B,T)}(\mathcal{X},\mathcal{L})\ge \delta J^\mathrm{NA}(\mathcal{X},\mathcal{L})$ for any normal semiample test configuration $(\mathcal{X},\mathcal{L})$ for $(X,L)$,
\item $\beta_{(X,B,T)}(F)\ge\delta j_{L}(F)$ for any prime divisor over $X$,
\end{enumerate}
\end{thm}

\begin{proof}
It follows from Corollary \ref{6.11} and Theorems \ref{6.12} and \ref{6.13}.
\end{proof}
 
\begin{cor}[cf., {\cite[Theorem 3]{Li2}}]\label{appc}
Let $(\mathbb{P}^1,\Delta,T)$ be a log-twisted Fano pair with $\Delta=\sum_{i=1}^m a_ip_i$ and a semiample $\mathbb{Q}$-line bundle $T$, where $a_i>0$, $m\ge 0$ and $p_i$ are distinct closed points of $\mathbb{P}^1$.
Then, the log-twisted Fano pair $(\mathbb{P}^1,\Delta,T,-K_{\mathbb{P}^1}-\Delta-T)$ is uniformly K-stable (resp., K-semistable) if and only if
\[
\max_{1\le i\le m}\{a_i\}<\frac{\mathrm{deg}(\Delta+T)}{2}\, \left(\mathrm{resp}. \, \max\{a_i\}\le\frac{\mathrm{deg}(\Delta+T)}{2}\right).
\]
\end{cor}
 
\begin{proof}
This immediately follows from Theorem \ref{6.5} as \cite[Example 6.6]{Fjt}.
\end{proof}

Now, we are ready to prove Theorem \ref{appendix}. 

\begin{proof}[Proof of Theorem \ref{appendix}]
By assumption, Corollary \ref{6.11}, Theorem \ref{6.13} and Theorem \ref{6.5}, there exists a prime divisor $F$ over $X$ such that $\beta_{(X,B,T)}(F)=0$.
Here, we set 
\[
\alpha_F:=\inf_{D\in|L|_{\mathbb{Q}}}\frac{A_{(X,B)}(F)}{\mathrm{ord}_F(D)}\ge0.
\]
For any rational positive number $0<\tau<T_L(F)$, we can take $k\in \mathbb{Z}_{>0}$ such that $H^0(X,kL-k\tau F)\ne 0$. Then there exists an effective $\mathbb{Q}$-Cartier $\mathbb{Q}$-divisor $D\sim_{\mathbb{Q}}L$ such that $\mathrm{ord}_F(D)\ge\tau$ and hence we obtain
\[
\frac{A_{(X,B)}(F)}{\tau}\ge \alpha_F.
\]
Therefore, we obtain \begin{equation}\label{eq-alphaf}
\alpha_FT_{L}(F)\le A_{(X,B)}(F).
\end{equation}

Next, we have by $\beta_{(X,B,T)}(F)=0$ that
\[
A_{(X,B)}(F)\mathrm{vol}(L)=\int^{T_{L}(F)}_0\mathrm{vol}(L-xF)dx.
\]
By (\ref{eq-alphaf}), this means that 
\[
\int^{T_{L}(F)}_0(\alpha_F\mathrm{vol}(L)-\mathrm{vol}(L-xF))dx\le0
\]
 and hence $\alpha_F<1$.
 Thus, there exists a $\mathbb{Q}$-Cartier $\mathbb{Q}$-divisor $D\in |L|_{\mathbb{Q}}$ such that $A_{(X,B)}(F)<{\mathrm{ord}_F(D)}$. 
Let
\[
\beta_+:=A_{(X,B)}(F)\mathrm{vol}(L)=\int^{T_{L}(F)}_0\mathrm{vol}(L-xF)dx>0.
\]
Then it holds that
\begin{align*}
\beta_{(X,B+\epsilon D,T)}(F)&=A_{(X,B+\epsilon D)}(F)\mathrm{vol}((1-\epsilon)L)-\int^{(1-\epsilon)T_{L}(F)}_0\mathrm{vol}((1-\epsilon)L-xF)dx\\
&=(1-\epsilon)^n(A_{(X,B)}(F)-\epsilon\,\mathrm{ord}_F(D))\mathrm{vol}(L)-(1-\epsilon)^{n+1}\int^{T_{L}(F)}_0\mathrm{vol}(L-xF)dx\\
&=\beta_+\epsilon(1-\epsilon)^n\left(1-\frac{\mathrm{ord}_F(D)}{A_{(X,B)}(F)}\right)<0
\end{align*}
for sufficiently small $\epsilon>0$.
Thus the assertion follows from Theorem \ref{6.5}.
\end{proof}


\begin{thebibliography}{99}
\bibitem[AZ22]{AZ}H. Abban, Z. Zhuang, K--stability of Fano varieties via admissible flags, Forum. of. Mathematics. Pi. {\bf10} (2022), 1--43.
\bibitem[ABHLX20]{ABHLX}  J. Alper, H. Blum, D. Halpern-Leistner, and C. Xu, Reductivity of the automorphism group of K-polystable Fano varieties, Invent. Math. {\bf222} (2020), 995--1032.
\bibitem[A04]{A} F. Ambro, {\it Shokurov's boundary property}. J. Differential Geom. {\bf 67} (2004), no. 2, 229-255.
\bibitem[A05]{Am} F. Ambro, {\it The moduli b-divisor of an lc-trivial fibration}. Compos. Math. {\bf 141}(2) (2005), 385-403.
\bibitem[And01]{A2} M. Andreatta, {\it Actions of linear algebraic groups on projective manifolds and minimal model program}. Osaka J. Math. 38 (2001), no. 1, 151-166.
\bibitem[BHPV04]{BHPV} W. Barth, K. Hulek, C. Peters, and A. Van de Ven, {\it Compact complex surfaces}, second edition, Springer-Verlag, Berlin, 2004.
\bibitem[BDL20]{BDL} R.J. Berman, T. Darvas, C.H. Lu, {\it Regularity of weak minimizers of the K-energy and applications to properness and K-stability}. Ann. Sci. Ec. Norm. Super. 53 (2020), no. 4, 267-289
\bibitem[B12]{B} C. Birkar, {\it Existence of log canonical flips and a special LMMP}, Publ. math. IHES {\bf115}, 325-368 (2012).
\bibitem[BCHM10]{BCHM} C. Birkar, P. Cascini, C. Hacon, and J. M$^c$Kernan, {\it Existence of minimal models for varieties of log general type}, J. Amer. Math. Soc. {\bf23} (2010), no. 2, 405-468.
\bibitem[BJ20]{BlJ} H. Blum, M. Jonsson. {\it Thresholds, valuations, and K-stability}, Adv. Math. 365 (2020).
\bibitem[BLX22]{BLX} H.~Blum, Y.~Liu, C.~Xu,
Openness of K-semistability for Fano varieties.
Duke Math. J.  {\bf171}(13) (2022) 2753--2797
\bibitem[BLZ22]{BLZ} H. Blum, Y. Liu and C. Zhou. {\it Optimal destabilization of K-unstable Fano varieties via stability thresholds,}  Geom. Topol. {\bf26} (2022) 2507–2564
\bibitem[BHJ17]{BHJ} S. Boucksom, T. Hisamoto, M. Jonsson. {\it Uniform K-stability, Duistermaat-Heckman measures and singularities of pairs,} Ann. Inst. Fourier (Grenoble) {\bf 67}(2), 743-841, 2017.
\bibitem[BJ18]{BJ2} S. Boucksom, M. Jonsson. {\it A non-Archimedean approach to K-stability.} arXiv:1805.11160
\bibitem[C21]{C} G. Chen, {\it The J-equation and the supercritical deformed Hermitian-Yang-Mills equation}, Invent. math. {\bf225}, 529-602, 2021.
\bibitem[CDS15]{CDS} X.X. Chen, S. K. Donaldson and S. Sun. {\it K\"{a}hler-Einstein metrics on Fano manifolds}, I-III. J. Amer. Math. Soc. {\bf 28} (2015), 183-197, 199-234, 235-278.
\bibitem[CC21]{Ch} X.X. Chen and J. Cheng. On the constant scalar curvature K\"{a}hler metrics (II), existence results, J.
Amer. Math. Soc. {\bf34} (2021), 937–1009. 
\bibitem[CD89]{CD} F.R. Cossec, I.V. Dolgachev, {\it Enriques surfaces}. I. Progress in Mathematics, vol. {\bf76}.
Birkh\"{a}user, Boston (1989)
\bibitem[DP21]{DP} V. Datar, V. Pingali. {\it A numerical criterion for generalised Monge-Amp\`{e}re equations on projective manifolds.} Geom. Funct. Anal. {\bf31} (2021), 767--814.
\bibitem[Der16a]{Der} R. Dervan, {\it Alpha invariants and coercivity of the Mabuchi functional on Fano manifolds}. Ann. Fac. Sci. Toulouse S\'{e}r. 6, 25 no. 4 (2016), 919-934
\bibitem[Der16b]{De2} R. Dervan, {\it Uniform stability of twisted constant scalar curvature K\"{a}hler metrics}, Int. Math. Res. Not. IMRN {\bf 15} (2016), 4728-4783.
\bibitem[DR19]{DR2} R. Dervan, J. Ross, {\it Stable maps in higher dimensions.} Math. Ann. {\bf 374}, 1033-1073(2019).

\bibitem[DS21a]{DS2} R. Dervan and L. M. Sektnan. {\it Optimal symplectic connections on holomorphic submersions}, Comm. Pure Appl. Math. {\bf 74}(10) (2021) 2132-2184.
\bibitem[DS21b]{DS4} R. Dervan and L. M. Sektnan. {\it Moduli theory, stability of fibrations and optimal symplectic connections}, Geom. Topol. {\bf25} (2021) 2643-2697.
\bibitem[Don01]{Dn} S.K. Donaldson, {\it Scalar Curvature and Projective Embeddings,} I. J. Differential Geom. 59(3): 479-522 (2001).
\bibitem[Don02]{Dn2} S. K. Donaldson, Scalar curvature and stability of toric varieties, J. Differential Geom. {\bf62} (2002), no. 2, 289--349.
\bibitem[Fi04]{J2} J.~Fine, {\it Constant scalar curvature K\"{a}hler metrics on fiberd complex surfaces}, J. Differential Geom. 68(3), (2004) 397-432.
\bibitem[Fi07]{Fine} J. Fine, {\it Fibrations with constant scalar curvature Kahler metrics and the CM-line bundle,} Math. Res. Lett. {\bf 14} (2007), no. 2, 239-247. MR 2318622
\bibitem[FL20]{FL} E. Floris, V. Lazi\'{c}. {\it A travel guide to the canonical bundle formula}. Birational Geometry and Moduli Spaces (E. Colombo, B. Fantechi, P. Frediani, D. Iacono, R. Pardini, eds.), Springer INdAM Series, vol. 39, Springer, 2020, pp. 37-55
\bibitem[FM91]{FM} R. Friedman, J. W. Morgan. {\it Smooth Four-Manifolds and Complex Surfaces}. Ergebnisse der Mathematik und ihrer Grenzgebiete 3. Folge. Band 27 (1991).
\bibitem[FS90]{FS}A.~Fujiki, G.~Schumacher. The moduli space of extremal compact K\"{a}hler manifolds and generalized Weil-Petersson metrics, Publ.~Res.~Inst.~Math.~Sci. {\bf26} (1990), 101--183.
\bibitem[Fuj11]{Fuj} O.~Fujino, {\it Semi-stable minimal model program for varieties with trivial canonical divisor}, Proc. Japan Acad. Ser. A Math. Sci. {\bf87}(3), 2011, 25-30.
\bibitem[Fuj12]{fujino-bpf} O.~Fujino, Basepoint-free theorems: saturation, $b$-divisors, and canonical bundle formula, Algebra Number Theory {\textbf{6}} (2012), no.~4, 797--823.
\bibitem[Fuj17]{fujino-foundation}O. Fujino, Foundations of the minimal model program, MSJ Mem. {\bf35}, Mathematical Society of
Japan, Tokyo, 2017.
\bibitem[FG14]{FG} O. Fujino, Y. Gongyo, {\it On the moduli b-divisors of lc-trivial fibrations}, Ann. Inst. Fourier (Grenoble) {\bf64}(4), 1721-1735, 2014.
\bibitem[FM00]{FM2} O. Fujino, S. Mori, {\it A canonical bundle formula}. J. Differential Geom. {\bf 56} (2000), no. 1, 167-188.
\bibitem[Fujita18]{Fjt2} K. Fujita. {\it Optimal bounds for the volumes of Kähler-Einstein Fano manifolds.} American Journal of Mathematics, vol. 140 no. 2, 2018, p. 391-414.
\bibitem[Fuj19a]{Fjt} K. Fujita, {\it A valuative criterion for uniform K-stability of $\mathbb{Q}$-Fano varieties}. J. reine angew. Math. {\bf751} (2019), 309-338.
\bibitem[Fuj19b]{Fjtb} K.~Fujita, Uniform K-stability and plt blowups of log Fano pairs, Kyoto J. Math. {\bf 59} (2019), no.~2, 399--418.
\bibitem[FO18]{FO} K. Fujita, Y. Odaka, {\it On the K-stability of Fano varieties and anticanonical divisors}. Tohoku Math. J. (2) {\bf70}(4): 511-521 (2018).
%
\bibitem[Ful84]{F} W. Fulton. {\it Intersection Theory}. Springer, 1984.
\bibitem[EGA]{EGA} A. Grothendieck, J.A. Dieudonn\'{e}, {\it El\'{e}ments de G\'{e}om\'{e}trie Alg\'{e}brique} IV,
Publ. Math. IH\'{E}S {\bf20} (1964), 5-259, {\bf24} (1965), 5-231, {\bf 28} (1966), 5-255, {\bf 32} (1967),
5-361
\bibitem[HX13]{HX} C. Hacon, C. Xu, {\it Existence of log canonical closures}, Invent. Math. {\bf192}, (2013) 161-195.
\bibitem[Har77]{Ha} R. Hartshorne. {\it Algebraic Geometry}, Springer-Verlag, 1977.

\bibitem[HH23]{HH} K. Hashizume, M. Hattori, On boundedness and moduli spaces of K-stable Calabi-Yau fibrations over curves. arXiv:2305.01244
.
\bibitem[Hat21]{Hat} M. Hattori, {\it A decomposition formula for J-stability and its applications}, arXiv:2103.04603
\bibitem[Hat22a]{Hat2} M. Hattori, {\it On fibration stability after Dervan-Sektnan and singularities}, arXiv:2202.09992.
\bibitem[Hat22b]{Hatpre} M. Hattori, {\it On K-stability of Calabi-Yau fibrations}, arXiv:2203.11460
\bibitem[Hat22c]{CM} M. Hattori, Minimizing CM degree and specially K-stable varieties. arXiv:2211.03108

\bibitem[JSS19]{JSS} W. Jian, Y. Shi and J. Song. {\it A remark on constant scalar curvature K\"{a}hler metrics on minimal models.} Proc. Amer. Math. Soc. 147 (2019), 3507-3513.
\bibitem[KKMSD73]{KKMS} G. Kempf, F. Knudsen, D. Mumford, B. Saint-Donat, {\it Toroidal embeddings.} I, Lecture Notes in Mathematics, Vol. {\bf 339}. Springer-Verlag, Berlin-New York, 1973. viii+209pp.

\bibitem[K93]{Ko3} J. Koll\'{a}r, {\it Effective base point freeness}, Math. Ann. {\bf296} (1993), no. 4, 595-605.
\bibitem[K97]{Ko1} J. Koll\'{a}r, {\it Singularities of pairs}. Algebraic geometry--Santa Cruz 1995, 1997, pp. 221–287
\bibitem[K07]{Ko2}J. Koll\'{a}r, {\it Lectures on Resolution of Singularities} (AM-166). Princeton University Press, 2007.
\bibitem[K15]{Ko} J. Koll\'{a}r. {\it Singularities of the Minimal Model Program.} Cambridge Tracts in Mathematics {\bf 200}. Cambridge University Press, Cambridge, 2015.
\bibitem[Ko22]{kollar-moduli} J. Koll\'{a}r. Families of varieties of general type. 2022
\bibitem[KM98]{KoMo} J. Koll\'{a}r and S. Mori. {\it Birational geometry of algebraic varieties.} Cambridge Tracts in Mathematics {\bf 134}. Cambridge University Press, Cambridge, 1998
\bibitem[L04]{Laz} R. Lazarsfeld. {\it Positivity in algebraic geometry.} I. Ergebnisse der Mathematik und ihrer Grenzgebiete.
3. Folge, Springer. (2004).
\bibitem[La11]{Lai} C.J. Lai, {\it Varieties fibered by good minimal models}. Math. Ann. 350, 533-547 (2011). 
\bibitem[Li15]{Li2} C. Li. {\it Remarks on logarithmic K-stability.} Commun. Contemp. Math. {\bf17} (2015), no. 2, Article ID 1450020.
\bibitem[Li17]{Li3} C. Li, {\it K-semistability is equivariant volume minimization}, Duke Math. {\bf166}, no. 16 (2017), 3147-3218.
\bibitem[Li20]{Li} C. Li. {\it Geodesic rays and stability in the cscK problem.} to appear in Ann. Sci. \'{E}c. Norm. Sup\'{e}r.
\bibitem[LX14]{LX} C. Li, C. Xu. {\it Special test configuration and K-stability of Fano varieties.} Ann. Math. {\bf 180} (2014), no. 1, p. 197-232.
\bibitem[LXZ22]{LXZ} Y. Liu, C. Xu, and Z. Zhuang. {\it Finite generation for valuations computing stability thresholds and applications to K-stability}. Ann. of Math. {\bf196} (2022), 507--566. 
\bibitem[Mi80]{Mi2} R. Miranda, {\it Stability of pencils of cubic curves}, Amer. J. Math. {\bf102} (1980), no. 6, 1177-1202.
\bibitem[Mi81]{Mi} R. Miranda, {\it The moduli of Weierstrass fibrations over $\mathbb{P}^1$}, Math. Ann. {\bf255} (1981), no. 3, 379-394.
\bibitem[Mi83]{Mi3} R. Miranda, {\it Projectively unstable elliptic surfaces}. Illinois J. Math. {\bf27}, no. 3, (1983), 404-420.
\bibitem[MFK94]{GIT} D. Mumford, J. Fogarty, F. Kirwan, {\it Geometric Invariant Theory}, 3rd. edition, Ergebnisse der Mathematik und ihrer Grenzgebiete, {\bf34}, Springer-Verlag (1994).
\bibitem[O10]{O5} Y.~Odaka, On the GIT stability of polarized
varieties{\rm:} a survey, Proceeding of Kinosaki algebraic geometry symposium 2010. 
\bibitem[O12]{O2} Y. Odaka. {\it The Calabi conjecture and K-stability.} Int. Math Res. Not. IMRN (2012), no. 10, 2272-2288.
\bibitem[O13a]{Od} Y. Odaka. {\it A generalization of the Ross-Thomas slope theory.} Osaka J. Math. 50(1): 171-185 (2013).
\bibitem[O13b]{GIToda} Y. Odaka, {\it The GIT stability of polarized varieties via discrepancy}, Ann. of Math. {\bf 177} (2013), 645-661. 


\bibitem[OS12]{OS2} Y. Odaka, Y. Sano. {\it Alpha invariant and K-stability of $\mathbb{Q}$-Fano varieties}. Adv. Math. {\bf229} (2012), 2818-2834.
\bibitem[OS15]{OS} Y. Odaka, S. Sun, {\it Testing log K-stability by blowing up formalism.} Ann. Fac. Sci. Toulouse Math., 24(3) 505-522, 2015
\bibitem[OX12]{OX} Y. Odaka, C. Xu. {\it Log-canonical models of singular pairs and its applications}. Math. Res. Lett. {\bf 19} (2012), no. 2, 325-334.
\bibitem[Ort23]{Ort} A.~Ortu, Optimal symplectic connections and deformations of holomorphic submersions. Adv.~Math.~{\bf414} (2023).
\bibitem[SD22]{Sj} Z. Sj\"{o}str\"{o}m Dyrefelt, {\it Existence of cscK metrics on smooth minimal models}, Ann. Sc. Norm. Super. Pisa Cl. Sci. (5) Vol. XXIII (2022), 223--232.
\bibitem[S20]{S} J. Song. {\it Nakai-Moishezon criterion for complex hessian equations}. arXiv:2012.07956v1, 2020.
\bibitem[SW08]{SW} J. Song and B. Weinkove: {\it On the convergence and singularities of the J-flow with applications to the Mabuchi energy}. Comm. Pure Appl. Math. {\bf61} (2008), no.2, 210-229.
\bibitem[Tak08]{Ta}S.~Takayama On uniruled degenerations of algebraic varieties with trivial canonical divisor, Math. Z. (2008) 259:487-501.
\bibitem[T87]{T3} G. Tian, {\it On K\"{a}hler-Einstein metrics on certain K\"{a}hler manifolds with $C_1(M)>0$},  Invent. Math. {\bf89}, 225-246 (1987).
\bibitem[T90]{T} G. Tian. {\it On Calabi's conjecture for complex surfaces with positive first Chern class}. Invent. Math. {\bf101} (1990), no. 1, 101-172.
\bibitem[T97]{T4} G. Tian, K\"{a}hler-Einstein metrics with positive scalar curvature, Invent. Math. {\bf130} (1997), no. 1, 1--37. 
\bibitem[T15]{T2} G. Tian. {\it K-stability and K\"{a}hler-Einstein metrics}. Comm. Pure Appl. Math. {\bf 68} (2015), 1085-1156.
\bibitem[Z21a]{Z2} K. Zhang, {\it Continuity of delta invariants and twisted K\"{a}hler-Einstein metrics}, Adv. Math. 388 (2021), Paper No. 107888.
\bibitem[Z21b]{Z} K. Zhang, {\it A quantization proof of the uniform Yau-Tian-Donaldson conjecture}, to appear in J. Eur. Math. Soc., arXiv:2102.02438v2
\end{thebibliography}
\end{document}